\documentclass[12pt]{amsart}
\usepackage{amsmath,mathrsfs,amsthm, xypic,verbatim,amscd,color}
\usepackage{txfonts}
\usepackage{graphicx}
\usepackage{colordvi}
\usepackage{times}

\DeclareMathOperator{\Exp}{Exp}
\DeclareMathOperator{\codim}{codim}
\DeclareMathOperator{\Gr}{Gr}

\def\L{\mathcal{L}}

\newtheorem{theorem}{Theorem}[section]
\newtheorem{lemma}[theorem]{Lemma}

\newtheorem{proposition}[theorem]{Proposition}
\newtheorem{corollary}[theorem]{Corollary}
\newtheorem{conjecture}[theorem]{Conjecture}

\newtheorem{remark}[theorem]{Remark}

\newtheorem{definition}[theorem]{Definition}

 \newcommand{\SL}{\operatorname{SL}}

  \newcommand{\al}{\alpha}
\newcommand{\defpar}{\tau}

  \newcommand{\lam}{\lambda}

 \newcommand{\ip}[1]{\langle #1 \rangle}
 
  \renewcommand{\tilde}{\widetilde}

  \def\b1{\text{\bf\large 1}}  

  \def\ip<#1>{\langle#1\rangle}   

  \newcommand{\Aut}{\operatorname{Aut}}
  
  \newcommand{\GL}{\operatorname{GL}}
  \newcommand{\OG}{\operatorname{OG}}
  \newcommand{\IG}{\operatorname{IG}}
  \newcommand{\Sp}{\operatorname{Sp}}
  \newcommand{\SU}{\operatorname{SU}}

  \newcommand{\SO}{\operatorname{SO}}

  \newcommand{\Hom}{\operatorname{Hom}}
  \newcommand{\Imo}{\operatorname{Im}}
  \newcommand{\slope}{\operatorname{slope}}

\newcommand{\Ad}{\operatorname{Ad}}
  \newcommand{\Ker}{\operatorname{Ker}}
   \newcommand{\Spin}{\operatorname{Spin}}

\DeclareMathOperator{\Pic}{Pic}
  \newcommand\elal{\Bbb{L}}
    \newcommand\xx{\Bbb{X}}
    \newcommand\exx{\Bbb{X}}
    
 \newcommand\bz{\Bbb{Z}}
\newcommand{\beqn}{\begin{equation}}
\newcommand{\eeqn}{\end{equation}}
\DeclareMathOperator{\Iim}{Im}

\newcommand{\tset}[1]{\underset{#1}{\times}}
  
\renewcommand{\ni}{\noindent}

 \newcommand{\fb}{\mathfrak{b}}
 \newcommand{\fg}{\mathfrak{g}}
 \newcommand{\fh}{\mathfrak{h}}
 \newcommand{\fl}{\mathfrak{l}}
 
 \newcommand{\fp}{\mathfrak{p}}
 \newcommand{\fs}{\mathfrak{s}}
 \newcommand{\ft}{\mathfrak{t}}
 \newcommand{\fu}{\mathfrak{u}}

\def\u.{^{\bullet}}

\newcommand{\bc}{\mathbb{C}}

\newcommand{\br}{\mathbb{R}}
\newcommand{\X}{\tilde{X}}
\newcommand{\Z}{\tilde{Z}}

\newcommand{\fra}{\mathfrak{a}}
 \newcommand{\frb}{\mathfrak{b}}

\newcommand{\frg}{\mathfrak{g}}
\newcommand{\frh}{\mathfrak{h}}

\newcommand{\frk}{\mathfrak{k}}
\newcommand{\frl}{\mathfrak{l}}

\newcommand{\frp}{\mathfrak{p}}

\newcommand{\fru}{\mathfrak{u}}

\newcommand\bull{\sssize{\bullet}}

\newcommand\mt{\mathcal{T}}

\newcommand\G{\hat{G}}
\newcommand\B{\hat{B}}
\newcommand\g{\hat{g}}
\newcommand\W{\hat{W}}

 \newcommand{\cl}{\mathcal{L}}

\newcommand\OMM{\Omega}
\newcommand\OMC{\bar{\Omega}}
\newcommand\LAM{\Phi}
\newcommand\FS{\mathfrak{S}}

\newtheorem{dfn}{Definition}[section]

\newtheorem{rem}[dfn]{Remark}
\newtheorem{thm}[dfn]{Theorem}

\newtheorem{lem}[dfn]{Lemma}

\newtheorem{cor}[dfn]{Corollary}
\newtheorem{conj}[dfn]{Conjecture}
\newtheorem{ex}[dfn]{Example}

\newtheorem{defn}[dfn]{Definition}

\def\la{\lambda}

\def\C{{\mathbb C}}
\def\R{{\mathbb R}}

\def\Z{{\mathbb Z}}
\def\K{{\mathbb K}}

\def\si{\sigma}
\def\Si{\Sigma}
\def\ga{\gamma}
\def\al{\alpha}
\def\be{\beta}
\def\De{\Delta}
\def\la{\lambda}
\def\La{\Lambda}

\def\O{{\mathcal O}}
\def\T{{\mathcal T}}
\def\L{\mathcal L}

\def\ol{\overline}

\def\P{{\mathbb P}}

\def\Q{{\mathbb Q}}
\def\B{{\mathbb B}}
\def\N{{\mathbb N}}

\def\g{{\mathfrak g}}
\def\X{{\mathfrak X}}

\def\acts{\curvearrowright}

\def\tits{\partial_{Tits}}

\def\<{\langle}
\def\>{\rangle}

\def\ov{\overrightarrow}

\begin{document}

\title[Additive Eigenvalue Problem (a survey)] {Additive Eigenvalue Problem (a survey)\\ (With appendix by M. Kapovich)}

\author{Shrawan Kumar}

\maketitle
\begin{center}
{Contents}
\end{center}
1. Introduction\\
2. Notation\\
3. Determination of the eigencone/saturated tensor cone\\
4. Specialization of Theorem \ref{thm1}  to $G=\SL(n)$: Horn inequalities\\
5. Deformed product\\
6. Efficient determination of the eigencone\\
7. Study of the saturated restriction semigroup and irredundancy of its
inequalities\\
8. Notational generalities on classical groups\\
9. Comparison of the eigencones under diagram automorphisms\\
10. Saturation problem\\
11. Deformed product and Lie algebra cohomology\\
12. A restricted functoriality of the deformed product and a product formula\\
13. Tables of the deformed product $\odot$ for the groups of type $B_2$, $G_2$,
$B_3$ and $C_3$\\
14. An explicit determination of the eigencone  for the ranks 2 and 3 simple Lie
algebras\\
15. Appendix (by M. Kapovich) \\
16. Bibliography

\section{Introduction}

This is a fairly self contained survey article on the classical Hermitian eigenvalue problem and its
 generalization to an arbitrary connected reductive group.

For any $n\times n$ Hermitian matrix $A$, let $\lambda_A =
(\lambda_1\geq \cdots \geq \lambda_n)$ be its set of eigenvalues written in
descending order.
Recall the following classical problem,
known as the  {\it  Hermitian eigenvalue problem}: Given two
 $n$-tuples of nonincreasing real
numbers: $\lambda=(\lambda_1\geq \cdots \geq \lambda_n)$ and
 $\mu=(\mu_1\geq \cdots \geq \mu_n)$,
determine all possible
$\nu=(\nu_1\geq \cdots \geq \nu_n)$  such that there exist Hermitian matrices
$A,B,C$ with $\lambda_A=\lambda, \lambda_B=\mu, \lambda_C=\nu$ and $C=A+B$.
This problem has  a long history starting with the work of Weyl (1912) followed
by  works of Fan (1949),
 Lidskii (1950),  Wielandt (1955), and culminating into the following conjecture
 given by  Horn (1962).

For any positive integer $r<n$,
inductively define the set $S_r^n$ as the set of triples $(I,J,K)$ of subsets of
$[n]:=\{1,\dots, n\}$ of cardinality $r$
such that
\begin{equation}
\sum_{i\in I}i +\sum_{j\in J}j =r(r+1)/2
+\sum_{k\in K} k
\end{equation}
 and for all $0<p<r$ and $(F,G,H)\in S^r_p$ the following
inequality holds:
\begin{equation}
\sum_{f\in F}i_f +\sum_{g\in G}j_g\leq p(p+1)/2 +
\sum_{h\in H} k_h.
\end{equation}
\begin{conjecture}  A triple $\lambda,\mu,\nu$ occurs as eigenvalues of Hermitian $n\times n$
matrices $A,B,C$ respectively such that $C=A+B$  if and only if
$$\sum_{i=1}^n\nu_i=\sum_{i=1}^n\lambda_i + \sum_{i=1}^n\mu_i,$$
and
for all $1\leq r <n$ and all triples   $(I,J,K)\in S^n_r$, we have
$$\sum_{k\in K}\nu_k\leq \sum_{i\in I}\lambda_i+\sum_{j\in J}\mu_j.$$
\end{conjecture}

 Horn's above
conjecture was settled in the affirmative (cf. Corollary \ref{coro5}) by combining the work of
Klyachko (1998) with the work of Knutson-Tao (1999) on the `saturation' problem.

The above system of inequalities is overdetermined. Belkale (2001)
 proved that a certain subset of the above set of inequalities suffices. Subsequently,
  Knutson-Tao-Woodward (2004) proved that the subsystem of inequalities given by
  Belkale
forms an irredundant system of inequalities.

Now, we  discuss a generalization of the above Hermitian eigenvalue
 problem  (which can be rephrased in terms of the special unitary group $\SU (n)$ and
 its complexified Lie algebra $\mathfrak{s} \mathfrak{l}(n)$) to an arbitrary
 complex semisimple  group.
Let $G$ be a connected, semisimple complex algebraic
group. We fix a Borel subgroup $B$, a maximal torus $H$, and a maximal
compact subgroup $K$. We denote their Lie algebras by the corresponding
Gothic characters: $\fg$,
$ \fb, \fh, \frk$ respectively.
 Let  $R^+$ be the set of positive roots (i.e., the set of roots of $\frb$)
  and let $\Delta = \{\alpha_1, \dots, \alpha_\ell\} \subset R^+$ be the set of
simple roots.
There is a natural homeomorphism
$\delta:\frk/K\to \frh_{+}$, where $K$ acts on $\frk$ by the adjoint representation
and $ \frh_{+}:=\{h\in  \frh: \alpha_i(h)\geq 0 \,\forall \, i\}$  is the positive Weyl chamber in $\frh$. The inverse map
$\delta^{-1}$ takes any $h\in \frh_+$ to the $K$-conjugacy class of $\surd{-1} h$.

For any positive integer $s$, define the {\it eigencone}
 $$\bar{\Gamma}_s(\frg):=
\{(h_1,\dots,h_s)\in\frh_{+}^s\mid \exists (k_1,\dots,k_s)\in \mathfrak k^s
:
\, \sum_{j=1}^s k_j=0\,\,\text{and }\, \delta (k_j)=h_j \forall j\}.$$

By virtue of the convexity result in symplectic geometry, the subset
 $\bar{\Gamma}_s(\frg)\subset \frh_{+}^s$  is a convex rational polyhedral cone
(defined by certain  inequalities with rational coefficients). The aim
of the {\it general additive eigenvalue problem} is to find the inequalities
describing  $\bar{\Gamma}_s(\frg )$ explicitly.
(The case $\frg=\mathfrak{s} \mathfrak{l}(n)$ and $s=3$ specializes to the Hermitian eigenvalue
problem
if we replace $C$ by $-C$.)

Let $\Lambda =\Lambda(H)$ denote the character group of $H$ and let
 $\Lambda_+:=\{\lambda\in \Lambda: \lambda(\alpha^\vee_i)\geq 0 \,\forall \,\,
 \text{simple coroots}\,\, \alpha^\vee_i\}$
denote the set of all the dominant characters.
Then, the set of isomorphism classes of irreducible (finite dimensional) representations
of $G$ is parameterized by $\Lambda_+$ via the highest weights of  irreducible
representations.
For $\lambda \in \Lambda_+$, we denote by $V(\lambda)$ the corresponding irreducible
representation (of highest weight $\lambda$).

Similar to the eigencone $\bar{\Gamma}_s(\frg)$, one defines the
{\it saturated tensor semigroup}
$$
\Gamma_{s}(G)=\left\{(\lambda_{1},\ldots,\lambda_{s})\in \Lambda_+^{s}:
      \left[V(N\lambda_1)\otimes\cdots\otimes
        V(N\lambda_{s})\right]^{G}\neq 0,\text{~ for some ~} N\geq 1\right\}.
$$
Then, under the
identification
$\varphi:\mathfrak{h}\xrightarrow{\sim}\mathfrak{h}^{*}$ (via the
Killing form)
\beqn\label{intro1}
\varphi(\overline{\Gamma}_{s}(\fg))\cap \Lambda_+^{s}=\Gamma_{s}(G).
\eeqn
 (cf. Theorem \ref{sj}).

For any $ 1\leq
j\leq \ell$, define the element $x_j\in \fh$
by
\begin{equation}\alpha_i(x_{j})=\delta_{i,j},\text{ }\forall\text{ } 1\leq i\leq \ell.
\end{equation}

Let $P\supset B$ be a standard parabolic subgroup with Lie algebra
$\frp$ and let $\frl$ be its unique Levi component containing the Cartan subalgebra
$\frh$. Let $\Delta(P)\subset \Delta$ be the set of simple roots contained in the
set of roots of $\frl$.
Let $W_P$ be the Weyl group of $P$ (which is, by definition, the  Weyl Group of
 the Levi component $L$)
and let $W^P$ be the set of the minimal length representatives
in the cosets of $W/W_P$.
For any $w\in W^P$, define the Schubert variety:
  \[
X^P_w := \overline{ BwP/P} \subset G/P.
  \]
It is an irreducible (projective) subvariety
of $G/P$ of dimension $\ell(w)$.
Let
$\mu(X^P_w)$ denote the fundamental class of $X^P_w$ considered
as an element of the singular homology with integral coefficients $H_{2\ell(w)}(G/P, \Bbb Z)$
of $G/P$. Then, from the Bruhat decomposition, the elements
$\{\mu(X^P_w)\}_{w\in W^P}$ form a $\Bbb Z$-basis of
 $H_*(G/P, \Bbb Z)$.
  Let $\{[X_w^P]\}_{w\in W^P}$ be the Poincar\'e dual basis of the singular
cohomology $ H^*(G/P, \Bbb Z)$. Thus,
$[X_w^P]\in H^{2(\dim G/P-\ell (w))}(G/P, \Bbb Z)$.
Write the standard cup product in $H^*(G/P, \Bbb Z)$ in the $\{[X_w^P]\}$ basis as follows:
\begin{equation}
[X^P_u]\cdot[X^P_v]=\sum_{w\in W^P} c^w_{u,v}[X^P_w].
\end{equation}
Introduce the indeterminates ${\defpar}_i$ for each $\alpha_i\in
\Delta\setminus\Delta(P)$ and define a deformed cup product $\odot$
as follows:
$$
[X^P_u] \odot [X^P_v]=
\sum_{w\in W^P} \bigl(\prod_{\alpha_i\in \Delta\setminus\Delta(P)}
{\defpar}_i^{(w^{-1}\rho-u^{-1}\rho -v^{-1}\rho -\rho)(x_i)}
\bigr)
c^w_{u,v} [X^P_w],
$$
where $\rho$ is the (usual) half sum of positive roots of $\frg$.
 By  Corollary ~\ref{product} and the identity \eqref{eqn5}, whenever $c^w_{u,v}$
 is nonzero,
the exponent of $\tau_i$ in the above is a nonnegative integer. Moreover, the product
$\odot$ is associative (and clearly commutative).
The  cohomology algebra of $G/P$
 obtained by setting each ${\defpar}_i=0$ in
$(H^*(G/P, \Bbb Z)\otimes\Bbb{Z}[{\defpar}_i],\odot)$ is denoted by
$(H^*(G/P, \Bbb Z),\odot_0)$. Thus, as a $\Bbb Z$-module, this  is the same as the singular cohomology
 $H^*(G/P, \Bbb Z)$ and under the product $\odot_0$ it is associative
 (and commutative). The definition of the deformed product $\odot_0$ (now known as the
{\it Belkale-Kumar product}) was arrived at from the
crucial concept of Levi-movability as in Definition \ref{D1}.
 For a cominuscule maximal parabolic $P$, the product $\odot_0$ coincides with the standard
cup product (cf. Lemma \ref{minuscule}).

Now we are ready to state the main result on  solution of the
eigenvalue problem for any connected semisimple $G$ (cf. Corollaries \ref{eigen} and \ref{7.5}).
 For a maximal parabolic $P$, let
$\alpha_{i_P}$ be the unique simple root
not in the Levi of $P$ and let $\omega_P:= \omega_{i_P}$ be the corresponding
fundamental weight.

\begin{theorem} \label{intromain} Let $(h_1,\dots,h_s)\in\frh_{+}^s$.  Then, the following
are equivalent:

(a) $(h_1,\dots,h_s)\in\bar\Gamma_s(\fg)$.

(b)
 For every standard maximal parabolic subgroup $P$ in $G$ and every choice of
  $s$-tuples  $(w_1, \dots, w_s)\in (W^P)^s$ such that
$$[X^P_{w_1}]\cdots [X^P_{w_s}] = d[X^P_{e}] \,\,\,\text{for some}\,\, d\neq 0
,$$
the following inequality holds:
$$I^P_{(w_1, \dots, w_s)}:\,\,\,\,\omega_P(\sum_{j=1}^s\,w_j^{-1}h_j)\leq 0.$$

(c) For every standard maximal parabolic subgroup $P$ in $G$ and every choice of
  $s$-tuples  $(w_1, \dots, w_s)\in (W^P)^s$ such that
$$[X^P_{w_1}] \cdots [X^P_{w_s}] = [X^P_{e}],$$
the above  inequality $I^P_{(w_1, \dots, w_s)}$ holds.

(d) For every standard maximal parabolic subgroup $P$ in $G$ and every choice of
  $s$-tuples  $(w_1, \dots, w_s)\in (W^P)^s$ such that
$$[X^P_{w_1}]\odot_0 \cdots \odot_0 [X^P_{w_s}] = [X^P_{e}],$$
the above  inequality $I^P_{(w_1, \dots, w_s)}$ holds.
\end{theorem}

The equivalence of (a) and (b) in the above  theorem
for general $G$ is due to Berenstein-Sjamaar (2000). Kapovich-Leeb-Millson (2009)
showed the equivalence of (a) and (c). The equivalence of (a) and
(d) is due to Belkale-Kumar (2006).  If we specialize the above Theorem for
$G=\SL(n)$, then, in the view of Theorem \ref{thm4},
the equivalence of (a) and (b) is nothing but the Horn's conjecture
(Corollary \ref{coro5}) solved by
combining the work of
Klyachko (1998) with the work of Knutson-Tao (1999). (It may be remarked that
the proofs of the Horn's conjecture and Theorem \ref{thm4} are recursively
interdependent in the sense that the validity of Theorem \ref{thm4} for
$\SL(n)$ proves the Horn's conjecture  for $\mathfrak{s} \mathfrak{l}(n)$
(by using the equivalence of (a) and (b) in Theorem \ref{intromain} for
$G=\SL(n)$) but the proof of Theorem \ref{thm4} for
$\SL(n)$ requires the validity of the Horn's conjecture  for
$\mathfrak{s} \mathfrak{l}(r)$, for $r<n$.) In this case, the
equivalence of (a) and
(c)  is due to Belkale (2001).  In this case, every maximal
parabolic
subgroup $P$ is cominuscule and hence the deformed product
$\odot_0$ in $H^*(G/P)$ coincides with the standard cup product. Hence the parts
(c) and (d) are the same in this case.

Because of the identification \eqref{intro1}, the above theorem allows us to determine the
saturated tensor semigroup $\Gamma_s(G)$ (See Theorems \ref{thm2} and \ref{EVT}
 for a precise statement).

The proof of the equivalence of (a) and (b) parts of the above theorem
follows from the Hilbert-Mumford criterion for semistability (cf. Proposition
\ref{propn14}) and the determination
of the Mumford index as in Lemma \ref{l1}. The proof
of the equivalence of the (a) and (d) parts is more delicate and relies on
the  Kempf's  maximally
destabilizing one parameter subgroups and Kempf's  parabolic subgroups
associated to unstable points. In addition, the notion of
Levi-movability   plays a fundamental role in the proof.

As proved by Ressayre (2010), the inequalities given by the (d) part of the
above theorem
form an irredundant system of inequalities determining the cone $\bar\Gamma_s(\fg)$
(cf. Corollary \ref{6.4}). (As mentioned above, for $\fg=\mathfrak{s}\fl(n)$ it
was proved by
Knutson-Tao-Wodward.) Ressayre's proof relies on the notion of well-covering pairs
(cf. Definition \ref{wellcovering}), which is equivalent to  the notion of
Levi-movability with cup product $1$ (cf. Lemma \ref{lemma7.5ress}).

The eigencone $\bar\Gamma_3(\fg)$  for the ranks 2 and 3 simple Lie
algebras $\fg$ is explicitly determined in Section \ref{sec10}. For $\fg$ of
 rank 2, it is due to Kapovich-Leeb-Millson (2009) and for $\fg$ of rank 3, it
is due to Kumar-Leeb-Millson (2003). The description
 relies on the above theorem (the equivalence of (a) and (d)).
As shown by Kumar-Leeb-Millson (2003), the (c) part
of the above theorem  gives rise to
126 inequalities for $\fg$ of type
$B_3$ or $C_3$, whereas by the (d) part one gets only 93 inequalities.

Let $\frg$ be a simple Lie algebra with a diagram automorphism $\sigma$ and let
$\frk$ be the fixed subalgebra (which is necessarily simple again). Then,
 as shown  by Belkale-Kumar (2010) for the pairs $(sl(2n), sp(2n))$ and
 $(sl(2n+1), so(2n+1)),$ and by Braley (2012) and Lee (2012) for the other pairs that the eigencone
 $\bar\Gamma_s(\frk)$ of $\frk$ is the intersection of the eigencone
 $\bar\Gamma_s(\fg)$ of $\fg$ with the dominant chamber of $\frk$ (cf. Theorem
 \ref{5.6}). The proof for the pair $(sl(2n), sp(2n))$ (resp.
 $(sl(2n+1), so(2n+1))$) relies on the result that any collection of
 Schubert varieties in any Grassmannian can be moved by elements of $\Sp (2n)$
 (resp. $\SO(2n+1)$) so that their  intersection is proper (cf. Theorems
 \ref{wilson1} and \ref{newwilson2}). The proof in the other cases relies
on the comparison
between the intersection product of
the partial flag varieties $G/P$ of $G$ (corresponding to the maximal parabolic
subgroups $P$ of $G$)
with that of the  deformed product in the  partial flag varieties $K/Q$ of $K$ (corresponding to the maximal
parabolic subgroups $Q$ of $K$).

An `explicit' determination of the eigencone $\bar\Gamma_s(\fg)$ of $\fg$
via Theorem \ref{intromain} hinges upon
 understanding the product $\odot_0$ in $H^*(G/P)$ in the Schubert basis, for all the
 maximal parabolic
 subgroups $P$. Clearly, the product $\odot_0$ is easier to understand than the usual
 cup product (which is the subject matter of {\it Schubert Calculus}) since in
 general `many
  more' terms in the product $\odot_0$ in the Schubert basis drop out. However,
  the  product $\odot_0$ has a drawback in that it is not functorial, in general
  even for the
  standard projections $\pi:G/P \to G/Q$ for parabolic subgroups $P\subset Q$.
  But, for certain embeddings of flag varieties $\iota : G/P \hookrightarrow
  \hat{G}/\hat{P}$,
  Ressayre-Richmond (2011) defined a certain `deformed' pull-back map in
  cohomology which respects the product $\odot_0$ (cf. Theorem \ref{rr}).
  A decomposition  formula for the  structure constants in  $\odot_0$ is obtained by
  Richmond (2012) (also by Ressayre (2011)) (cf. Theorems \ref{rrthm7}, \ref{rrthm2}
  and Corollary \ref{12.14}). We give the tables of the deformed product $\odot$
  for the groups of type $B_2,G_2, B_3$ and $C_3$ and for any maximal parabolic
  subgroups in Section \ref{exemples}.

  Also, as shown by Belkale-Kumar (2006), the deformed product  $\odot_0$ in
  $H^*(G/P)$ is connected with the
  Lie algebra cohomology of the nil-radical
  $\fru_P$ of the parabolic subalgebra $\frp$ (cf. Theorem \ref{liecohomology}).

Let $G \subset \G$ be connected reductive complex algebraic groups. Fix a
maximal torus $H$ (resp. $\hat{H}$) and a
Borel subgroup $H\subset B$ (resp. $\hat{H} \subset \hat{B}$) of $G$ (resp. $\G$)
such that $H\subset \hat{H}$ and
  $B\subset \hat{B}$.
Define the {\it saturated restriction
semigroup}
$$
\Gamma (G, \G)=\left\{(\lambda, \hat{\lambda})\in \Lambda_+\times \hat{\Lambda}_+:
      \left[V(N\lambda)\otimes \hat{V}
      ( N\hat{\lambda})\right]^{G}\neq 0,\text{~ for some ~} N\geq 1\right\}.
$$
In Section \ref{section8}, Theorem \ref{intromain} is generalized to the
determination of $\Gamma (G, \G)$. Specifically, we have the following result
due to Ressayre (2010) (cf. Theorems \ref{ress} and \ref{ress2}). (A weaker result
was obtained by Berenstein-Sjamaar (2000).)

\begin{theorem}\label{introress}  Assume that no nonzero
ideal of $\frg$ is an ideal of  $\hat{\frg}$.
Let $(\lambda, \hat{\lambda})\in \Lambda_+\times \hat{\Lambda}_+$. Then, the
 following three conditions  are equivalent.

(a) $(\lambda, \hat{\lambda})\in \Gamma (G, \G)$.
\vskip1ex

(b) For any $G$-dominant $\delta\in O(H)$, and any $(w, \hat{w})\in W^{P(\delta)}
\times
{\hat{W}}^{\hat{P}(\delta)}$ such that
$[X_w^{P(\delta)}]\cdot \iota^*([\hat{X}_{\hat{w}}^{\hat{P}(\delta)}])
\neq 0$  in
 $H^*(G/ P(\delta), \bz)$, where $\hat{X}_{\hat{w}}^{\hat{P}(\delta)}:=
 \overline{\hat{B}\hat{w}\hat{P}(\delta)/\hat{P}(\delta)}\subset
 \hat{G}/\hat{P}(\delta)$ and $\iota: G/ P(\delta) \to \G/\hat{P}(\delta)$
 is the canonical embedding, we have
 \begin{equation} I^{\delta}_{(w, \hat{w})}:\,\,\,\,\,\,
\lambda(w\dot{\delta})+ \hat{\lambda}(\hat{w}\dot{\delta})\leq 0,
\end{equation}
where $P(\delta)$(resp. $\hat{P}(\delta)$) is the Kempf's parabolic in $G$
(resp. $\hat{G}$) defined by the identity \eqref{2n}.
\vskip1ex

(c) For any OPS $\delta_i\in \mathfrak{S} (G,\G)$ and any $(w, \hat{w})\in W^{P(\delta_i)} \times
{\hat{W}}^{\hat{P}(\delta_i)}$ such that

($c_1$) \,$[X_w^{P(\delta_i)}]\cdot \iota^*([\hat{X}_{\hat{w}}^{\hat{P}(\delta_i)}])
 =  [X_e^{P(\delta_i)}]\in
 H^*(G/ P(\delta_i), \bz)$, and

($c_2$) \, $ \gamma^{P(\delta_i)}_e (\dot{\delta_i})-
\gamma^{P(\delta_i)}_w (\dot{\delta_i})=
\hat{\gamma}^{\hat{P}(\delta_i)}_{\hat{w}}(\dot{\delta}_i),$

the above inequality $I^{\delta_i}_{(w, \hat{w})}$ is satisfied, where
the set $\mathfrak{S} (G,\G)$ is defined in Definition \ref{ressspecial} and
$\gamma^{P(\delta_i)}_w (\dot{\delta_i})$ (resp.
$\hat{\gamma}^{\hat{P}(\delta_i)}_{\hat{w}}(\dot{\delta}_i)$) are given by the
identities \eqref{ress1900} (resp. \eqref{ress1901}).

Moreover, the set of inequalities provided by the (c)-part
is an irredundant system of inequalities describing the cone  $ \Gamma (G, \G)_\br$
inside  $\Lambda_+(\br)\times{ \hat{\Lambda}}_+(\br)$, where  $\Lambda_+(\br)$
 denotes the cone inside  $\Lambda\otimes_{\bz}\, \br$ generated by
 $\Lambda_+$ and  $ \Gamma (G, \G)_\br$ is the cone generated by
 $ \Gamma (G, \G)$.
\end{theorem}

Let $G$ be a connected semisimple  group.
  The {\it saturation problem} aims at connecting the tensor product semigroup
 $$T_s(G):=\{(\lambda_1, \dots, \lambda_s)\in \Lambda_+^s :
  [V(\lambda_1)\otimes\dots
 \otimes V(\lambda_s)]^{G}\neq 0\}$$
 with the saturated tensor product semigroup
 ${\Gamma}_ s (G)$.
 An integer $d\geq 1$ is called a {\it saturation factor} for $G$,
 if for any $(\lam,\mu,\nu)\in {\Gamma}_3(G)$ such that $\lam+\mu+\nu \in Q$, then
 $(d\lam, d\mu,d\nu)\in T_3(G)$, where $Q$ is the root lattice of $G$.
 Such a $d$ exists by Corollary \ref{saturationnew}.
If $d=1$ is a saturation factor for $G$, we say that
 the {\it saturation property holds for $G$}.

The {\it saturation theorem} of Knutson-Tao (1999) mentioned above, proved by using
their `honeycomb model,'  asserts that
the saturation property holds for $G=\SL(n)$. Other proofs
of their result are given by  Derksen-Weyman (2000), Belkale (2006) and
 Kapovich-Millson (2008).

The following general result (though not optimal) on saturation factor is obtained
 by Kapovich-Millson (2008)
(cf. the Appendix).

\begin{theorem}\label{introKM}  For any connected simple $G$, $d=k_\frg^2$ is a
 saturation factor,
where $k_\frg$ is the least common multiple of the coefficients of the highest
root $\theta$ of the Lie algebra $\frg$ of $G$ written in terms of the simple roots
$\{\al_1,\dots, \al_\ell\}$.
\end{theorem}

Kapovich-Millson  (2006) made the very interesting conjecture
 that if $G$ is simply-laced, then the saturation
 property holds for $G$.
 Apart from  $G=\SL(n)$, the only other simply-connected, simple, simply-laced group
 $G$ for which the above conjecture is known so far is $G=\Spin (8)$, proved by
 Kapovich-Kumar-Millson (2009) by explicit calculation using the equivalence of
 (a) and (d) in Theorem \ref{intromain}.

For the classical
groups $\SO(n) \, (n\geq 5)$ and $\Sp(2\ell) \,(\ell \geq 2)$, 2 is a saturation factor.
It was proved by
Belkale-Kumar (2010) for the groups $\SO(2\ell+1)$ and $\Sp(2\ell)$ by using
geometric techniques.  Sam  (2012) proved it for $\SO(2\ell)$ (and also for
$\SO(2\ell+1)$ and $\Sp(2\ell)$) via the quiver approach (following the proof
by Derksen-Weyman (2010) for $G=\SL(n)$).
(Observe that the general result of Kapovich-Millson gives  a saturation factor
of $4$ in these cases.)

We recall, in Section \ref{secsaturation}, a `rigidity' result for the
$\SL(n)$-representations due to Knutson-Tao-Woodward (2004),
 which was conjectured by Fulton and also its generalization to an arbitrary
 reductive group by Belkale-Kumar-Ressayre (2012) (cf. Theorems \ref{ktwfulton}
 and \ref{ktwfulton1}).

We refer the
reader to
the survey article of Fulton [F$_2$] on the Hermitian eigenvalue problem; and
the Bourbaki talk by Brion [Br].
\vskip2ex

\noindent
{\bf Acknowledgements.} It is my pleasure to thank Prakash Belkale who
introduced me to the eigenvalue problem and my joint works with him. I thank N. Ressayre
 for going through the article and his comments. 
I also thank M. Kapovich, B. Leeb, J. Millson and N. Ressayre for my joint works with them
on the subject. I gratefully acknowledge the support from the NSF grant
 DMS-1201310.

\section{Notation}\label{sec1}

Let $G$ be a  semisimple connected complex algebraic group. We choose a Borel subgroup
$B$ and a maximal torus $H\subset B$  and let $W=W_G:=N_G(H)/H$ be the associated Weyl group,
 where $N_G(H)$ is the normalizer of $H$ in $G$.  Let $P\supseteq B$ be a (standard)
 parabolic subgroup of $G$ and let $U=U_P$ be its unipotent radical. Consider the
 Levi subgroup $L=L_P$ of $P$ containing $H$, so that $P$ is the semi-direct product
 of $U$ and $L$. Then,  $B_L:=B\cap L$ is a Borel subgroup of
$L$. Let $\Lambda =\Lambda(H)$ denote the character group of $H$, i.e., the group of all the
algebraic group morphisms $H \to \Bbb G_m$.  Clearly, $W$ acts on $\Lambda$.
We denote the Lie algebras of $G,B,H,P,U,L,B_L$ by the corresponding Gothic characters:
$\fg,\fb,\fh,\fp,\fu,\fl,\fb_L$ respectively.  We will often identify an element
$\lambda$ of $\Lambda$
(via its derivative $\dot{\lambda}$) by an element of $\fh^*$. Let $R=R_\fg \subset \frh^*$ be
the set of roots of $\fg$ with respect to the Cartan subalgebra $\fh$ and let $R^+$ be the
 set of positive roots (i.e., the set of roots of $\fb$).
Similarly, let $R_\fl$ be the set of roots of $\fl$ with respect to $\fh$ and
$R_\fl^+$ be the set of roots of $\fb_L$. Let  $\Delta = \{\alpha_1, \dots, \alpha_\ell\}
\subset R^+$ be the set of simple
roots, $\{\alpha_1^\vee, \dots, \alpha_\ell^\vee\}
\subset \fh$ the
corresponding simple coroots and $\{s_1,\dots, s_\ell\}\subset W$ the
corresponding simple reflections,  where $\ell$ is the rank of $G$.
 We denote by $\Delta(P)$ the set of simple roots
 contained in $R_\fl$.  For any $ 1\leq
j\leq \ell$, define the element $x_j\in \fh$
by
\begin{equation}\label{eqn0}\alpha_i(x_{j})=\delta_{i,j},\text{ }\forall\text{ } 1\leq i\leq \ell.
\end{equation}

Recall that if $W_P$ is the Weyl group of $P$ (which is, by definition, the  Weyl
Group $W_L$ of $L$), then in each coset of $W/W_P$ we have a unique member $w$ of minimal length.
 This satisfies (cf. [K$_1$, Exercise 1.3.E]):
\begin{equation}\label{eqn1}
wB_L w^{-1} \subseteq B.
\end{equation}
Let $W^P$ be the set of the minimal length representatives
in the cosets of $W/W_P$.

For any $w\in W^P$, define the Schubert cell:
\[
C_w^P:=  BwP/P \subset G/P.
  \]
Then, it is a locally closed subvariety of $G/P$ isomorphic with the affine
space $\Bbb A^{\ell(w)}, \ell(w)$ being the length of $w$ (cf. [J, Part II,
Chapter 13]). Its closure is denoted by $X^P_w$, which is an irreducible (projective) subvariety
of $G/P$ of dimension $\ell(w)$. We denote the point $wP\in C_w^P$ by $\dot{w}$.
We abbreviate $X_w^B$ by $X_w$.

Let
$\mu(X_w^P)$ denote the fundamental class of $X_w^P$
considered as an element of the singular homology with integral coefficients
$H_{2\ell(w)}(G/P, \Bbb Z)$ of $G/P$. Then, from the Bruhat decomposition, the elements
$\{\mu(X_w^P)\}_{w\in W^P}$ form a $\Bbb Z$-basis of
 $H_*(G/P, \Bbb Z)$. Let $\{[X_w^P]\}_{w\in W^P}$ be the Poincar\'e dual basis of the singular
cohomology with integral coefficients $ H^*(G/P, \Bbb Z)$. Thus,
$[X_w^P]\in H^{2(\dim G/P-\ell (w))}(G/P, \Bbb Z)$. Similarly, let
$\{\epsilon^P_w\}_{w\in W^P}$ be the  basis of  $ H^*(G/P, \Bbb Z)$ dual to the
basis $\{\mu(X_w^P)\}_{w\in W^P}$ of $H_*(G/P, \Bbb Z)$ under the standard pairing, i.e., for any
$v,w\in W^P$, we have
\[\epsilon^P_v(\mu (X^P_w))=\delta_{v,w}.\]
Then, for any $w\in W^P$, by [KuLM, Proposition 2.6],
\begin{equation}\label{eqn1n}
\epsilon^P_w= [X_{w_oww_o^P}^P],
\end{equation}
where $w_o$ (resp. $w_o^P$) is the longest element of the Weyl group $W$ (resp.
$W_P$). (Observe that $w_oww_o^P\in W^P$ for any $w\in W^P$.)

An element $\lambda\in \Lambda$ is called dominant (resp. dominant regular) if
$\dot{\lambda}(\alpha^\vee_i)\geq 0$ (resp. $\dot{\lambda}(\alpha^\vee_i) > 0$)
for all the simple coroots $\alpha^\vee_i$. Let $\Lambda_+$ (resp. $\Lambda_{++}$)
denote the set of all the dominant (resp. dominant regular) characters.
The set of isomorphism classes of irreducible (finite dimensional) representations
of $G$ is parameterized by $\Lambda_+$ via the highest weights of  irreducible
representations.
For $\lambda \in \Lambda_+$, we denote by $V(\lambda)$ the corresponding irreducible
representation (of highest weight $\lambda$). The dual representation $V(\lam)^*$
is isomorphic with $V(\lam^*)$, where $\lam^*$ is the weight $-w_o\lam$. The $\mu$-weight space of $V(\lambda)$
is denoted by $V(\lambda)_\mu$. For $\lam\in \Lambda_+$, let $P(\lam )$ be
the set of weights of $V(\lam )$. We denote the fundamental weights by
$\{\omega_i\}_{1\leq i\leq \ell}$, i.e.,
$$\omega_i(\alpha_j^\vee)=\delta_{i,j}.$$

For any $\lambda \in \Lambda$, we have a  $G$-equivariant
 line bundle $\cl (\lambda)$  on $G/B$ associated to the principal $B$-bundle $G\to G/B$
via the one dimensional $B$-module $\lambda^{-1}$. (Any  $\lambda \in \Lambda$ extends
 uniquely to a character of $B$.) The one dimensional $B$-module $\lambda$ is also denoted by
 $\Bbb C_\lambda$.

All the schemes are considered over the base field of complex numbers $\Bbb C$. The varieties
are reduced (but not necessarily irreducible) schemes.

\section{Determination of the eigencone/saturated tensor cone}\label{section2}

Let the notation and assumptions be as in Section \ref{sec1}. In particular, $G$ is a
connected semisimple complex algebraic group.
 Fix a positive integer $s$ and define the {\it saturated tensor
semigroup}
$$
\Gamma_{s}(G)=\left\{(\lambda_{1},\ldots,\lambda_{s})\in \Lambda_+^{s}:
      \left[V(N\lambda_1)\otimes\cdots\otimes
        V(N\lambda_{s})\right]^{G}\neq 0,\text{~ for some ~} N\geq 1\right\}.
$$
Let $$\mathfrak{h}_{+}:=\{x\in \fh: \alpha_i(x)\in \mathbb{R}_+ \,\,\text{for all
the simple roots}\,\, \alpha_i\}$$ be the dominant chamber in $\mathfrak{h}$,
where  $\mathbb{R}_+$ is the set of nonnegative real numbers.
Define the {\em eigencone}
$$
\overline{\Gamma}_{s}(\fg)=\left\{(h_{1},\ldots , h_{s})\in
\mathfrak{h}^{s}_{+}:\,\text{there exist}\,\,\ni k_{1},\ldots,k_{s}\in K
\,\,\text{with}\,\,
\sum^{s}_{j=1}(\text{Ad~} k_{j})h_{j}=0\right\},
$$
where $K\subset G$ is a fixed maximal compact subgroup.
Then, $\overline{\Gamma}_{s}(\fg)$ depends only upon the Lie algebra $\fg$
and the choices of its Borel subalgebra $\fb$ and the Cartan subalgebra $\fh$.

Under the
identification
$\varphi:\mathfrak{h}\xrightarrow{\sim}\mathfrak{h}^{*}$ (via the
Killing form) $\Gamma_{s}(G)$ corresponds to the set of integral
points of $\overline{\Gamma}_{s}(\fg)$. Specifically, we have the following result essentially  following from Mumford
[N, Appendix] (also see [Sj, Theorem 7.6] and [Br, Th\'eor\`eme 1.3]).
\begin{theorem}\label{sj}
$$
\varphi(\overline{\Gamma}_{s}(\fg))\cap \Lambda_+^{s}=\Gamma_{s}(G).
$$
\end{theorem}
\begin{proof} For $h=(h_{1},\ldots ,  h_{s})\in
\mathfrak{h}^{s}_{+}$, let
$$\mathcal{O}_h:= (K\cdot h_1) \times \dots \times (K\cdot h_s) \subset ( i\frk)^s.$$
Then, $K$ acts on $\mathcal{O}_h$ diagonally. Let $m_h: \mathcal{O}_h\to  i\frk^*\simeq  i\frk$ be the
corresponding moment map, where the last identification is via the Killing form. Then, it is easy to see that
$m_h(y_1, \dots, y_s)=y_1+ \dots + y_s$, for $y_j\in K\cdot h_j$. Hence,
\beqn \label{eq0.1}  h\in \overline{\Gamma}_{s}(\fg) \Leftrightarrow 0\in \text{Im} (m_h).
\eeqn
Now, take  $h=(h_{1},\ldots ,  h_{s})\in
\mathfrak{h}^{s}_{+}$ such that $\lambda=\varphi (h)=(\lambda_1, \dots , \lambda_s)\in
 \Lambda_+^{s}$, where $\lambda_j:= \varphi (h_j)$. Consider the closed subvariety
$$X_\lambda:=G\cdot [v_{\lambda_1}] \times \dots \times  G\cdot [v_{\lambda_s}] \subset
 \mathbb{P}(V(\lambda_1))  \times \dots \times  \mathbb{P}(V(\lambda_s)),$$
where $ [v_{\lambda_j}]$ is the line through the highest weight vector in $V(\lambda_j)$.

It is easy to see that $ K\cdot h_j$ is diffeomorphic with $  G\cdot [v_{\lambda_j}]$ as
symplectic  $K$-manifolds. In particular, there  exists a $K$-equivariant symplectic diffeomorphism $\theta:
 \mathcal{O}_h\to  X_\lambda$ (under the diagonal action of $K$). Hence the following
diagram is commutative:
\[
\xymatrix{
X_\lambda\ar[dr]_{m_\lambda} & & \mathcal{O}_h
\ar[ll]_{\sim}^{\theta}\ar[dl]^{m_h}\\
 & i\frk^* &
}
\]
where $m_\lambda: X_\lambda \to  i\frk^*$ is the moment map for the $K$-variety $X_\lambda$.
Let ${\bar{\elal}} (\lambda)$  be the ample line bundle on $X_\lambda$ which is the restriction of the line bundle
$\mathcal{O}(1)\boxtimes \dots \boxtimes \mathcal{O}(1)$ on $\mathbb{P}(V(\lambda_1))  \times \dots \times  \mathbb{P}(V(\lambda_s)).$ Now, $X_\lambda$ has a $G$-semistable point with respect to the ample line bundle
 ${\bar{\elal}} (\lambda)$  if and only if $0\in \text{Im}(m_\lambda)$ (cf. [MFK, Theorem 8.3]). Further,
by the definition, $X_\lambda$ has a $G$-semistable point with respect to the line bundle  ${\bar{\elal}} (\lambda)$  if and
only if  $H^0(X_\lambda, {\bar{\elal}} (\lambda)^{\otimes N})^G\neq 0$
 for some $N>0$. The latter of course is equivalent (by the Borel-Weil theorem) to
$$[V(N\lambda_1)^* \otimes \dots \otimes  V(N\lambda_s)^*]^G\neq 0  \Leftrightarrow
[V(N\lambda_1) \otimes \dots \otimes  V(N\lambda_s)]^G\neq 0 \Leftrightarrow \lambda \in
\Gamma_{s}(G).$$
This, together with \eqref{eq0.1},  proves the theorem.
\end{proof}

We recall the following
transversality theorem due to Kleiman (cf. [BK$_1$, Proposition 3]).

\begin{theorem}\label{kleiman}
Let a connected algebraic group $S$ act transitively on a smooth
variety $X$ and let  $X_1$, $\dots, X_s$ be irreducible locally closed subvarieties of $X$.
Then, there exists a non empty open subset $V\subseteq S^s$ such that for
$(g_1,\dots, g_s)\in V$, the intersection $\bigcap_{j=1}^s \,g_j X_j$ is
proper (possibly empty) and dense in $\bigcap_{j=1}^s \,g_j \bar{X}_j$.

Moreover, if  $X_j$ are
smooth varieties, we can find such a
$V$ with the additional property that for $(g_1,\dots, g_s)\in V$,
$\bigcap_{j=1}^s \,g_j X_j$ is transverse at each
point of intersection.
\end{theorem}
The following result follows from [F$_1$, Proposition 7.1 and Section 12.2].
\begin{proposition}\label{fultonproper}
Let $w_1, \dots,  w_s\in W^P$ and let  $g_1, \dots,  g_s\in G$ be such that the
intersection $Y:=g_1X^P_{w_1}\cap \dots \cap g_sX^P_{w_s}$ is proper (or empty) inside $G/P$.
Then, we have
$$[X^P_{w_1}]\dots [X^P_{w_s}]=[Y],$$
where $[Y]$ denotes the Poincar\'e dual of the fundamental class of the pure
(but not necessarily irreducible) scheme $Y$.

Moreover, for any irreducible component $C$ of $Y$, writing
$$ [C]= \sum_{w\in W^P}\, n_w [X_w^P],$$
for some (unique) $n_w\in \mathbb{Z}$, we have $n_w\in \mathbb{Z}_+$.
\end{proposition}

\noindent
{\bf A Review of Geometric Invariant Theory.}
We will need to consider in Sections \ref{section6} and \ref{section7} the Geometric Invariant Theory (GIT) in a nontraditional setting, where
a {\it nonreductive} group acts on a {\it nonprojective} variety.  First we recall the
following definition due to Mumford.
\begin{definition}\label{git} Let $S$ be any (not necessarily reductive) algebraic group
acting on a  (not necessarily projective) variety  $\exx$ and let  $\elal$ be
an $S$-equivariant line bundle on $\exx$. Let $O(S)$ be the set of all one parameter
subgroups (for short OPS) in $S$.
 Take any $x\in \exx$ and
 $\delta \in O(S)$ such that the limit
  $\lim_{t\to 0}\delta(t)x$
exists in $\exx$ (i.e., the morphism ${\delta}_x:\Bbb{G}_m\to X$ given by
$t\mapsto \delta(t)x$ extends to a morphism $\tilde{\delta}_x : \Bbb{A}^1\to X$).
Then, following Mumford, define a number $\mu^{\elal}(x,\delta)$ as follows:
Let $x_o\in X$ be the point  $\tilde{\delta}_x(0)$. Since $x_o$ is $\Bbb{G}_m$-invariant
via $\delta$, the fiber of  $\elal$ over $x_o$ is a
$\Bbb{G}_m$-module; in particular, is given by a character of $\Bbb{G}_m$. This  integer is defined as  $\mu^{\elal}(x,\delta)$.
\end{definition}

We record the following standard properties of $\mu^{\elal}(x,\delta)$ (cf.
 [MFK, Chap. 2, $\S$1]):
\begin{proposition}\label{propn14} For any $x\in \exx$ and $\delta \in O(S)$ such that $\lim_{t\to 0}\delta(t)x$
exists in $\exx$, we have the following (for any $S$-equivariant line bundles
$\elal, \elal_1, \elal_2$):
\begin{enumerate}
\item[(a)]
$\mu^{\elal_1\otimes\elal_2}(x,\delta)=\mu^{\elal_1}(x,\delta)+\mu^{\elal_2}(x,\delta).$
\item[(b)] If there exists $\sigma\in H^0(\exx,\elal)^S$ such that $\sigma(x) \neq 0$, then  $\mu^{\elal}(x,\delta)\geq 0.$
\item[(c)] If $\mu^{\elal}(x,\delta)=0$, then any element of $H^0(\exx,\elal)^S$
which does not vanish at $x$ does not vanish at $\lim_{t\to 0}\delta(t)x$ as well.
\item[(d)] For any $S$-variety $\exx'$ together with an $S$-equivariant morphism $f:\exx'\to \exx$ and any $x'\in \exx'$ such that  $\lim_{t\to 0}\delta(t)x'$
exists in $\exx'$, we have
$\mu^{f^*\elal}(x',\delta)=\mu^{\elal}(f(x'),\delta).$
\item[(e)] (Hilbert-Mumford criterion) Assume that $\exx$ is projective, $S$ is
 connected and reductive
and $\elal$ is ample. Then, $x\in\exx$ is semistable (with respect to $\elal$) if
and only if $\mu^{\elal}(x,\delta)\geq 0$, for all $\delta\in O(S)$.

In particular, if $x\in \exx$ is semistable and $\delta$-fixed, then
$\mu^{\elal}(x,\delta)= 0$.
\end{enumerate}
\end{proposition}

Let $S$ be a connected reductive group. For an OPS  $\delta \in O(S)$, define the
{\it associated Kempf's parabolic subgroup} $P(\delta)$ of $S$ by
\begin{equation} \label{2n} P(\delta):= \{g\in S: \lim_{t\to 0}\delta(t)g\delta(t)^{-1} \,
\text{exists in}\,S\}.
\end{equation}
For an OPS  $\delta \in O(S)$, let
$\dot{\delta}\in \mathfrak s$ be its derivative at $1$.

 Let $P$ be any standard parabolic subgroup
of $G$ acting on $P/B_L$ via the left multiplication. We call $\delta\in O(P)$
{\it $P$-admissible} if, for all $x\in P/B_L$, $\lim_{t\to 0}\,\delta(t)\cdot x$
exists in $P/B_L$. If $P=G$, then $P/B_L= G/B$ and any $\delta\in O(G)$ is
$G$-admissible since $G/B$ is a projective variety.

Observe that, $B_L$ being the
semidirect product of its commutator $[B_L,B_L]$ and $H$, any $\lambda \in \Lambda$
extends uniquely to a character of $B_L$.
Thus, for any $\lambda \in \Lambda$, we have a  $P$-equivariant
 line bundle $\cl_P (\lambda)$  on $P/B_L$ associated to the principal $B_L$-bundle $P\to P/B_L$
via the one dimensional $B_L$-module $\lambda^{-1}$. Thus, $\cl_G (\lambda)=
\cl (\lambda)$, as defined in Section \ref{sec1}.
We have taken the following lemma from [BK$_1$, Lemma 14]. It is a generalization
of the corresponding result in [BS, Section 4.2].
\begin{lemma}\label{l1}Let $\delta\in O(H)$ be such that $\dot{\delta}\in \frh_+.$
Then, $\delta$ is  $P$-admissible and, moreover, for any $\lam\in \Lambda$
and $x=ulB_L\in P/B_L$ (for $u\in U_P, l\in L_P$), we have the  following formula:
$$\mu^{\cl_P(\lam)}(x,\delta)=-\lam(w\dot{\delta}),$$
where $P_L(\delta):=P(\delta)\cap L$ and $w\in W_P/W_{P_L(\delta)}$ is any coset
representative  such that $l^{-1}\in B_LwP_L(\delta)$.
\end{lemma}
\begin{proof}
We first show that $\delta$ is $P$-admissible. Take $x=ulB_L\in P/B_L$, for $u\in
U_P$ and $l\in L_P$. Then,
$\delta (t)x=\delta (t) u \delta (t)^{-1}(\delta (t) lB_L).$ Now, since
$\dot{\delta}\in \fh_+$ and $u\in U_P$, it is easy to see that
 $\lim_{t\to 0} \,\delta (t) u \delta (t)^{-1}$ exists in $U_P$. Also,
 $\lim_{t\to 0} \,\delta (t) lB_L$ exists in $L/B_L$ since $L/B_L$ is a
 projective variety.
 Thus, $\delta$ is $P$-admissible.

 We next calculate $\mu^{\cl_P(\lam)}(x,\delta)$ for $x=ulB_L\in P/B_L$. Write
 $l^{-1}=b_l\dot{w}q$, for some $b_l\in B_L$ and $q\in P_L(\delta)\supset B_L$
 (where $\dot{w}$ is a representative of $w$ in the normalizer $N_L(H)$ of $H$ in
 $L$). Consider the OPS $b:\Bbb{G}_m \to B_L$ dfined by $b(t)=b_l\dot{w}
 \delta(t)^{-1}{\dot{w}}^{-1}b_l^{-1}$. Then,
 $$\delta(t)ulb(t)=\delta(t)uq^{-1}\delta(t)^{-1}{\dot{w}}^{-1}b_l^{-1}.$$
 In particular, by the definition of $P(\delta)$, $\lim_{t\to 0} \,\delta (t)
 ulb(t)$ exists in $P$. Consider the $\Bbb{G}_m$-invariant section
 $\sigma (t)= [\delta (t)ul, 1]:=(\delta (t)ul, 1)$ mod $B_L$ of
 $\delta_x^*(\cl_P(\lam))$ over $\Bbb{G}_m$, where $\delta_x:\Bbb{G}_m \to P/B_L$
 is the map $t \mapsto \delta (t)x$. Then, the section $\sigma (t)$ corresponds
 to the function  $\Bbb{G}_m \to  \Bbb{A}^1, \, t\mapsto \lambda^{-1}(b(t)^{-1}).$
 From this we see that $\mu^{\cl_P(\lam)}(x,\delta)=-\lam(w\dot{\delta}).$
 \end{proof}

Let  $\lambda=(\lambda_{1},\ldots,\lambda_{s})\in \Lambda_+^s$ and let
$\elal (\lambda)$ denote the $G$-linearized line bundle $\cl(\lam_1)\boxtimes
\cdots \boxtimes\cl(\lam_s)$ on
 $(G/B)^s$ (under the diagonal action of $G$). Then,  there exist unique
 standard parabolic subgroups $P_1, \dots, P_s$ such that the line bundle
$\elal (\lambda)$ descends as an ample line bundle
${\bar{\elal}} (\lambda)$ on
$\xx (\lambda):=G/P_1 \times \cdots \times G/P_s$. We call a point
 $x\in (G/B)^s$ {\it $G$-semistable} (with respect to, not necessarily ample,
$\elal (\lambda)$) if its image in $\xx (\lambda)$ under the canonical map $\pi: (G/B)^s
\to \xx (\lambda)$ is semistable with respect to the ample line bundle
${\bar{\elal}} (\lambda)$.
Now, one has the following celebrated theorem due to Klyachko [Kly] for $G=\SL(n)$,
extended to general $G$ by Berenstein-Sjamaar
[BS].

\begin{theorem}\label{thm1}
Let $\lambda_{1},\ldots,\lambda_{s}\in \Lambda_+$. Then, the following are
equivalent:
\begin{itemize}
\item[\rm(a)] $(\lambda_{1},\ldots,\lambda_{s})\in \Gamma_{s}(G)$

\item[\rm(b)] For every standard maximal parabolic subgroup $P$ and
  every Weyl group elements $w_{1},\ldots,w_{s}\in W^{P}\simeq
  W/W_{P}$ such that
\begin{equation}\label{3n}
[X^{P}_{w_1}]\dots [X^{P}_{w_{s}}]=d[X^{P}_{e}],\,\,\,\text{for some}\,d\neq 0,
\end{equation}
 the following inequality is satisfied:
$$
I^{P}_{(w_{1},\ldots,w_{s})}:\qquad\qquad\qquad \sum^{s}_{j=1}\lambda_{j}(w_{j}x_{P})\leq 0,
$$
where $\alpha_{i_P}
$
is the unique simple root not in the Levi of $P$ and $x_{P}:=x_{i_P}$.
\end{itemize}
\end{theorem}

\begin{proof}
Define the set $Y_{s}\subset G^{s}$
consisting of those $ (g_{1},\ldots,g_{s})\in
G^{s}$ such that $g_{1}X^{Q}_{w_{1}}\cap\ldots\cap g_{s}X^{Q}_{w_{s}}$
  and
$(g_{1}Bw_{1}Q/Q)\cap\ldots \cap (g_{s}Bw_{s}Q/Q)$
are proper intersections and such that the latter intersection is dense in
$g_{1}X^{Q}_{w_{1}}\cap\ldots\cap g_{s}X^{Q}_{w_{s}}$ for all the standard
parabolic subgroups $Q$ and all $w_{1},\ldots,w_{s}\in W$. Then, by
Theorem \ref{kleiman}, $Y_{s}$ contains a nonempty open subset of
$G^{s}$.

Now, $\lambda=(\lambda_{1},\ldots\lambda_{s})\in \Lambda_+^{s}$ belongs to
$\Gamma_{s}(G)$
\begin{itemize}
\item[$\Leftrightarrow$] $X^{s}=(G/B)^{s}$ contains a $G$-semistable
point $y$ with respect to  the line bundle
$\elal (\lambda)$
on $X^{s}$.

\item[$\Leftrightarrow$] $\mu^{\elal (\lambda)}
(y,\sigma)\geq 0$ for all one
  parameter subgroups $\sigma$ in $G$.
\end{itemize}

(To prove the first equivalence, observe that, for any $N \geq 0$,
$$H^0(X^s,\elal (N\lambda))\simeq  H^0( \xx (\lambda),{\bar{\elal}} (N\lambda))$$
under the pull-back map. The second equivalence of course follows by the
Hilbert-Mumford criterion Proposition \ref{propn14}(e) together with \ref{propn14}
(d).)

\medskip
\noindent
{\em Proof of $(a) \Rightarrow (b)$ in the above Theorem:}
\smallskip

Take $\lambda=(\lambda_{1},\ldots,\lambda_{s})\in \Gamma_{s}(G)$. Then,
$X^{s}$ has a $G$-semistable point for the line bundle
$\elal (\lambda)$. Moreover,
since the set of semistable points is open, we can take a semistable
point $y=(g_{1}B,\ldots,g_{s}B)$ with $(g_{1},\ldots,g_{s})\in
Y_{s}$.

Now, take a maximal parabolic $P$ and $w_{1},\ldots,w_{s}\in W^{P}$
satisfying \eqref{3n}. Thus,
$$
(g_{1}Bw_{1}P/P)\cap\ldots\cap (g_{s}Bw_{s}P/P)\neq \emptyset .
$$

Take $gP\in (g_{1}Bw_{1}P/P)\cap\ldots\cap (g_{s}Bw_{s}P/P)$.
Take the one parameter subgroup of $G$:
$
\sigma=\Exp (tx_{P}).
$

Then, by Lemma \ref{l1} and Proposition \ref{propn14},
$$
\mu^{\elal (\lambda)}(g^{-1}y,\sigma)=\sum-\lambda_{j}(w_{j}x_{P})\geq 0,
$$
where the last inequality is by the Hilbert-Mumford criterion
Proposition \ref{propn14}(e). This proves
$(a)\Rightarrow (b)$.

\medskip
\noindent
{\em Proof of $(b) \Rightarrow (a)$ :}
If (a) were false, then $\lambda=(\lambda_{1},\ldots,\lambda_{s})\not\in
\Gamma_{s}(G)$, i.e., $X^{s}$ has no $G$-semistable points for
$\elal (\lambda)$. Take any
$(g_{1},\ldots,g_{s})\in Y_{s}$ and consider the point
$y=(g_{1}B,\ldots,g_{s}B)\in X^{s}$. Since it is not a semistable
point, there exists a one parameter subgroup $\sigma=g^{-1}\Exp(tx)g$,
for $x\in \mathfrak{h}_{+}$ and $g\in G$ such that
$$
\mu^{\elal (\lambda)}(y,\sigma)<0\Leftrightarrow \mu^{\elal (\lambda)}(gy,\Exp(tx))<0.
$$

Let $Q$ be the Kempf's parabolic attached to $\Exp(tx)$. Then, by
definition, $Q\supset B$ and the simple roots of the Levi of $Q$ are
precisely those $\alpha_{i}$ such that $\alpha_{i}(x)=0$. Take
$w_{1},\ldots,w_{s}\in W^{Q}$ such that
\begin{equation}\label{e2}
(gg_{j})^{-1}\in Bw_{j}Q\quad\forall 1\leq j\leq s.
\end{equation}

Thus,  by Lemma \ref{l1},
$$
\mu^{\elal (\lambda)}(gy,\Exp tx)=-\sum^{s}_{j=1}\lambda_{j}(w_{j}x)<0.
$$

In particular, there exists a maximal parabolic $P\supset Q$ such that
\begin{equation}\label{e3}
\sum^{s}_{j=1}\lambda_{j}(w_{j}x_{P})>0.
\end{equation}

Now, by \eqref{e2},
$$
gg_{1}X^{Q}_{w_{1}}\cap\ldots\cap gg_{s}X^{Q}_{w_{s}}\quad\text{is
  nonempty}.
$$

In particular,
$gg_{1}X^{P}_{w_{1}}\cap\ldots\cap gg_{s}X^{P}_{w_s}$ is
nonempty and since $(gg_{1},\ldots,gg_{s})\in Y_{s}$,
$gg_{1}X^{P}_{w_{1}}\cap\ldots\cap gg_{s}X^{P}_{w_{s}}$ is a proper intersection.
Thus, by Proposition \ref{fultonproper}, the cup product
$$
[X^{P}_{w_{1}}]\ldots [X^{P}_{w_{s}}]\neq 0.
$$

Hence, there exists a $w'_{s}\leq w_{s}$ such that $w'_{s}\in W^P$ and
$$
[X^{P}_{w_{1}}]\ldots [X^{P}_{w_{s-1}}]\cdot [X^P_{w'_{s}}]=d[X^{P}_{e}],\,\,\,
\text{for some}\, d\neq 0.
$$

Now, by the inequality $I^{P}_{(w_{1},\ldots,w_{s-1},w'_{s})}$ in (b), we get that
$$
\bigl(\sum^{s-1}_{j=1}\lambda_{j}(w_{j}x_{P})\bigr)+\lambda_{s}(w'_{s}x_{P})\leq 0 .
$$

But since $w'_{s}\leq w_{s}$, we have
$$
\lambda_{s}(w'_{s}x_{P})\geq \lambda_{s}(w_{s}x_{P}),
$$
by, e.g., [K$_1$, Lemma 8.3.3].
Thus, we get
$$
\sum^{s}_{j=1}\lambda_{j}(w_{j}x_{P})\leq \bigl(\sum^{s-1}_{j=1}
\lambda_{j}(w_{j}x_{P})\bigr)+\lambda_{s}(w'_{s}x_{P})\leq 0.
$$

This contradicts \eqref{e3} and hence proves that (a) is true.
\end{proof}

The following result follows easily by combining Theorems \ref{thm1} and \ref{sj}.
For a maximal parabolic $P$, let $\alpha_{i_P}$ be the unique simple root
not in the Levi of $P$. Then, we set $\omega_P:= \omega_{i_P}$.
\begin{corollary}\label{eigen}  Let $(h_1,\dots,h_s)\in\frh_{+}^s$.  Then, the following
are equivalent:

(a) $(h_1,\dots,h_s)\in\bar\Gamma_s(\fg)$.

(b)
 For every standard maximal parabolic subgroup $P$ in $G$ and every choice of
  $s$-tuples  $(w_1, \dots, w_s)\in (W^P)^s$ such that
$$[X^P_{w_1}]\cdots [X^P_{w_s}] = d[X^P_{e}] \,\,\,\text{for some}\,\, d\neq 0
,$$
the following inequality holds:
\begin{equation}\label{4new}\omega_P(\sum_{j=1}^s\,w_j^{-1}h_j)\leq 0.
\end{equation}
\end{corollary}

\begin{proof} Observe first that, under the identification of $\frh$ with $\frh^*$
induced from the Killing form,   $\frh_+$ is isomorphic with  the set
$$\Lambda_+(\mathbb{R}):=\{\lambda \in \fh^*: \lambda(\alpha^\vee_i)\in \mathbb{R}_+,
\,\,\text{for all the simple roots}\, \alpha_i\}$$ of dominant real weights in
$\frh^*$. In fact, under this
identification, $x_j$
corresponds with  $2\omega_j/\langle  \alpha_j, \alpha_j\rangle$, where $\omega_j$
denotes the $j$-th fundamental weight.
Thus, the corollary follows from Theorems \ref{thm1} and \ref{sj}.
\end{proof}

The same proof as above of Theorem \ref{thm1} gives the following result, which is weaker in
the direction `$(a)\Leftarrow (b)$' and stronger in `$(a)\Rightarrow (b)$'
direction.

\begin{theorem}\label{thm2}
For $\lambda_{1},\ldots,\lambda_{s}\in \Lambda_+$, the following are
equivalent:
\begin{itemize}
\item[\rm(a)] $(\lambda_{1},\ldots,\lambda_{s})\in \Gamma_{s}(G)$

\item[\rm(b)] For any (not necessarily maximal) parabolic  subgroup $Q$
  and any $w_{1},\ldots,w_{s}\in W^{Q}$ such that
  $[X^{Q}_{w_{1}}]\ldots [X^{Q}_{w_{s}}]\neq 0$ (not necessarily in
  the top cohomology class), the following inequality holds for any maximal
  parabolic $P\supset Q$:
$$
I^{P}_{(w_{1},\ldots,w_{s})}:\qquad\qquad\qquad
\sum^{s}_{j=1}\lambda_{j}(w_{j}x_{P})\leq 0.
$$
\end{itemize}
\end{theorem}
\begin{remark} \label{2.9} (a) Following Theorem \ref{thm2}, we can easily see that Corollary
\ref{eigen} remains true if we replace (b) (in Corollary \ref{eigen}) by demanding
the inequalities \eqref{4new} for any (not necessarily maximal) parabolic subgroup $Q$
and any $w_{1},\ldots,w_{s}\in W^{Q}$ such that
  $[X^{Q}_{w_{1}}]\ldots [X^{Q}_{w_{s}}]\neq 0.$

  (b) As proved by Belkale [B$_1$] for $G=\SL(n)$ and extended for an arbitrary $G$
  by Kapovich-Leeb-Millson [KLM$_1$], Theorem \ref{thm1} (and hence Corollary \ref{eigen})
  remains true if we replace $d$ by $1$ in the identity \eqref{3n}. A much sharper
  (and optimal) result for an arbitrary $G$ is obtained in Theorem \ref{EVT}.
  \end{remark}

\section{Specialization of Theorem \ref{thm1}  to $G=\SL(n)$: Horn Inequalities}\label{section3}

We first need to recall the Knutson-Tao saturation theorem [KT], conjectured by
Zelevinsky [Z]. Other proofs
of their result are given by  Derksen-Weyman [DK], Belkale [B$_3$] and
 Kapovich-Millson [KM$_2$].
\begin{theorem} \label{knutson-tao} Let $G=\SL(n)$ and let $(\lambda_{1},\dots,
\lambda_{s})\in \Gamma_{s}(G)
$
be such that $\lambda_{1}+\dots +\lambda_{s}$ belongs to the root lattice. Then,
$$\left[V(\lambda_1)\otimes\cdots\otimes
        V(\lambda_{s})\right]^{G}\neq 0.$$
 \end{theorem}
        Specializing Theorem \ref{thm1} to $G=\SL(n)$, as seen below, we obtain the classical Horn
inequalities.

In this case, the partial flag varieties corresponding to the maximal parabolics
$P_r$ are precisely the Grassmannians
of  $r$-planes in $n$-space $G/P_{r}=\Gr(r,n)$, for $0<r<n$. The
Schubert cells in $\Gr(r,n)$ are parameterized by the subsets of cardinality $r$:
$$
I=\{i_1<\ldots<i_{r}\}\subset \{1,\ldots,n\}.
$$

The corresponding Weyl group element $w_{I}\in W^{P_{r}}$ is nothing
but the permutation
$$
1\mapsto i_{1},\quad 2\mapsto i_{2},\cdots,r\mapsto i_{r}
$$
and $w_{I}(r+1),\ldots,w_{I}(n)$ are the elements in
$\{1,\ldots,n\}\backslash I$ arranged in ascending order.

Let $I'$ be the `dual' set
$$
I'=\{n+1-i, \ \ i\in I\},
$$
arranged in ascending order.

Then, the Schubert class $[X_{I}:=X^{P_{r}}_{w_{I}}]$ is Poincar\'e dual
to the Schubert class $[X_{I'}]\in H^*(\Gr(r,n), \mathbb{Z})$. Moreover,
\begin{equation}\label{e5}
\dim X_{I}=\codim X_{I'}=(\sum_{i\in I}i)-\frac{r(r+1)}{2}.
\end{equation}

We recall the following definition due to Horn.

\begin{definition}\label{defi3}
For $0<r<n$, inductively define the set $S^{r}_{n}$ of triples
$(I,J,K)$ of subsets of $\{1,\ldots,n\}$ of cardinality $r$ (arranged in
ascending order) satisfying
\begin{itemize}
\item[\rm(a)] $\sum\limits_{i\in I}i+\sum\limits_{j\in
  J}j=\dfrac{r(r+1)}{2}+\sum\limits_{k\in K}k$

\item[\rm(b)] For all $0<p<r$ and $(F,G,H)\in S^{p}_{r}$, the following
  inequality holds:
$$
\sum_{f\in F}i_{f}+\sum_{g\in G}j_{g}\leq \frac{p(p+1)}{2}+\sum_{h\in
  H}k_{h}.
$$
\end{itemize}
\end{definition}
The following theorem follows by Theorem \ref{thm1} for $G=\SL(n)$ (proved by Klyachko)
and Theorem \ref{knutson-tao} (proved by Knutson-Tao). Belkale [B$_3$] gave
another geometric proof of the theorem.
\begin{theorem}\label{thm4}
For subsets $(I,J,K)$ of $\{1,\ldots,n\}$ of cardinality $r$, the
product
$$
[X_{I'}]\cdot[X_{J'}]\cdot[X_{K}]=d[X_{e}^{P_{r}}],\,\,
\text{for some~ }\, d\neq 0\Leftrightarrow
(I,J,K)\in S^{r}_{n}.$$
\end{theorem}

\begin{proof}
For $\SL(n)/P_{r}=\Gr(r,n)$,
$$
x_{P_{r}}=\epsilon_{1}+\cdots+\epsilon_{r}-\frac{r}{n}(\epsilon_{1}+\cdots+
\epsilon_{n}),
$$
where $\epsilon_i$ is the $n\times n$ diagonal matrix with $1$ in the $i$-th place
and $0$ elsewhere.
Thus, for $I=\{i_{1}<\ldots<i_{r}\}$,
\begin{equation}\label{e6}
w_{I}(x_{P_{r}})=\epsilon_{i_{1}}+\cdots+\epsilon_{i_{r}}-\frac{r}{n}(\epsilon_{1}+
\cdots+\epsilon_{n}).
\end{equation}
Recall the classical result that the tensor product structure constants of
$\GL_{r}$-polynomial representations with highest weights
\begin{equation}\label{e7}\lambda:n-r\geq \lambda_{1}\geq \lambda_{2}\geq \ldots\geq
\lambda_{r}\geq 0\end{equation}
 correspond to the intersection product structure
constants for the Schubert varieties $X_{I'_{\lambda}}\subset
\SL(n)/P_{r}=\Gr(r,n)$, where $I_{\lambda}\subset \{1,\ldots,n\}$ is
the sequence:
$\lambda_{r}+1<\lambda_{r-1}+2<\ldots<\lambda_{1}+r$. Specifically,
for $\lambda^{(1)},\ldots,\lambda^{(s)}$ satisfying \eqref{e7}  with
$\sum\limits^{s}_{j=1}\dim X_{I_{\lambda^{(j)}}}=(n-r)r$, we have
(cf., e.g., [F$_0$, $\S$ 9.4])
\begin{equation}\label{e8}
\dim [V(\lambda^{(1)})\otimes\ldots\otimes
  V(\lambda^{(s)})]^{\SL(r)}=\text{coeff. of~ } [X^{P_{r}}_{e}]\text{~
  in~ } \prod^{s}_{j=1}[X_{I'_{\lambda^{(j)}}}].
\end{equation}

\noindent
{\em Proof of Theorem \ref{thm4} `$\Rightarrow$'~:} Take subsets $I$, $J$,
$K\subset \{1,\ldots,n\}$ of cardinality $r$ such that
\begin{equation}\label{e9}
[X_{I'}]\cdot [X_{J'}]\cdot [X_{K}]=d[X^{P_{r}}_{e}]\in
H^{*}(\SL(n)/P_{r})\,\,\,\text{for some constant}\,\, d\neq 0.
\end{equation}

From the above, we see that
$$
\dim \Gr(r,n) = \codim X_{I'}+\codim X_{J'}+\codim X_{K},$$
which gives, by the identity \eqref{e5},
\begin{equation}\label{e10} \dim X_{K} = \dim X_{I}+\dim X_{J}.
\end{equation}

From the identities \eqref{e5} and \eqref{e10}, the identity
(a) follows for $(I,J,K)$.

From the identities \eqref{e8} and \eqref{e9}, we see that as $\SL(r)$-representations,
$$
[V(\lambda_{I})\otimes V(\lambda_{J})\otimes
  V(\lambda_{K'})]^{\SL(r)}\neq 0,
$$
where $\lambda_{I}$ is the partition
$$
\lambda_{I}:n-r\geq i_{r}-r\geq i_{r-1}-(r-1)\geq \ldots\geq
i_{1}-1\geq 0.
$$

Thus,
$$
(\lambda_{I},\lambda_{J},\lambda_{K'})\in \Gamma_{3}(\SL(r)).
$$
Hence, by Theorem \ref{thm1} applied to $\SL(r)$, for any
maximal parabolic subgroup $P_{p}\subset \SL(r)$, $0<p<r$, and subsets
$F$, $G$, $H$ of $\{1,\ldots,r\}$ of cardinality $p$ with
\begin{equation}\label{5n}
[X_{F'}]\cdot [X_{G'}]\cdot [X_{H}]=d'[X^{P_{p}}_{e}],\,\,\,\text{for some}
\, d'\neq 0,
\end{equation}
 we have
$$
\lambda_{I}(w_{F'}x_{P_{p}})+\lambda_{J}(w_{G'}x_{P_{p}})+\lambda_{K'}(w_{H}
x_{P_{p}})\leq 0.
$$
Observe that, from the identity (a) of Definition \ref{defi3},

\begin{equation}\label{6n}|\lambda_{I}|+
|\lambda_{J}|+|\lambda_{K'}|=(n-r)r.
\end{equation}
Thus, by the identities \eqref{e6} and \eqref{6n},
$$\sum_{f\in F}i_{f}-\sum_{f\in F}f+\sum_{g\in
  G}j_{g}-\sum_{g\in G}g
-\sum_{h\in H}k_{h}+\sum_{h\in H}h\leq 0,$$
i.e.,
$$\sum_{f\in F}i_{f}+\sum_{g\in G}j_{g}\leq
\sum_{h\in H}k_{h}+\sum_{f\in F}f+\sum_{g\in G}g-\sum_{h\in H}h
=\sum_{h\in H}k_{h}+\frac{p(p+1)}{2},$$
where the last equality follows from the analogue of the identities \eqref{e5} and
\eqref{e10} corresponding to the identity \eqref{5n}.
Now, by induction, assuming the validity of Theorem \ref{thm4} for the
nonvanishing of cup products in $\SL(r)/P_{p}$ (since $p<r<n$), we get
that
$$(F,G,H)\in S^{p}_{r}\Leftrightarrow
[X_{F'}]\cdot [X_{G'}]\cdot [X_{H}]=d'[X^{P_{p}}_{e}],\quad\text{for
    some ~} d'\neq 0.$$
Thus, we get that $(I,J,K)\in S^{r}_{n}$, proving the `$\Rightarrow$'
implication.

Conversely, assume that the subsets $(I,J,K)$ each of cordinality $r$
contained in $\{1,\ldots,n\}$ belong to $S^{r}_{n}$. We want to prove
that
$$
 [X_{I'}]\cdot [X_{J'}]\cdot [X_{K}]=d[X^{P_{r}}_{e}],\quad\text{for
    some~} d\neq 0.$$
    By the identity \eqref{e8} and the condition (a) of Definition \ref{defi3}, this is equivalent to the nonvanishing
$ [V(\lambda_{I})\otimes V(\lambda_{J})\otimes
  V(\lambda_{K'})]^{\SL(r)}\neq 0$.
  By Theorem \ref{knutson-tao} for $G=\SL(r)$, the latter is equivalent to
$ (\lambda_{I},\lambda_{J},\lambda_{K'})\in \Gamma_{3}(\SL(r)),$
since $\lambda_{I}+\lambda_{J}+\lambda_{K'}$ belongs to the root lattice
of $\SL(r)$ because of the condition (a) (cf. the identity \eqref{6n}).

By Theorem \ref{thm1} for $G=\SL(r)$ and by assuming the validity of Theorem
\ref{thm4} by induction on $n$,
$
(\lambda_{I},\lambda_{J},\lambda_{K'})\in
\Gamma_{3}(\SL(r))\Leftrightarrow\,\,\text{for all}
$
maximal parabolic subgroups $P_{p}$, $0<p<r$, of $\SL(r)$, and all
$(F,G,H)\in S^{p}_{r}$, we have
$$
\lambda_{I}(w_{F'}x_{P_{p}})+\lambda_{J}(w_{G'}x_{P_{p}})+\lambda_{K'}
  (w_{H}x_{P_{p}})\leq
  0 ,$$ which is equivalent to the inequality
$$\sum_{f\in F}i_{f}+\sum_{g\in G}j_{g}\leq
\sum_{h\in H}k_{h}+\frac{p(p+1)}{2},$$
by the previous calculation.

But the last inequality is true by the  definition of
$S^{r}_{n}$. This proves the theorem.
\end{proof}

\begin{remark}
(1) Belkale-Kumar have given two inductive criteria (though only necessary
conditions) to determine when the product of a number of Schubert cohomology classes
in any $G/P$ is nonzero.
The first criterion is in terms of the characters (cf. [BK$_1$, Theorem 29])
and the second criterion is in terms of  dimension counts (cf.
[BK$_1$, Theorem 36]).

(2) Purbhoo [P] has given a criterion (again only a necessary condition) to
determine which of the Schubert
intersections vanish in terms of a combinatorial recipe called `root game'. He has a
similar recipe to determine the vanishing for branching Schubert calculus.

(3) For any cominuscule flag variety $G/P$, Purbhoo-Sottile have
determined a recursive set of inequalities (coming only from the class of
cominuscule flag varieties) which determines when the intersection product
$[X^P_{w_1}]\dots  [X^P_{w_s}]$ is nonzero in $H^*(G/P)$ (cf. [PS, Theorem 4]).
\end{remark}

For a Hermitian $n\times n$ matrix $A$, let $e_A=(e_1\geq  \dots \geq e_n)$ be
its set of eigenvalues (which are all real). Let $\mathfrak{a}$ be the standard Cartan subalgebra of
$sl(n)$ consisting of traceless diagonal matrices and let $\frb \subset sl(n)$
be the standard Borel subalgebra   consisting of traceless upper triangular
matrices (where $sl(n)$ is the Lie algebra of $\SL(n)$ consisting of traceless
$n\times n$-matrices). Then, the Weyl chamber
$$\fra_+=\{\text{diag} \,(e_1\geq  \dots \geq e_n): \sum e_i=0\}.$$
Define the {\it Hermitian eigencone}
\begin{align*}\bar{\Gamma}(n)=&\{(a_1,a_2,a_3)\in \fra_+^3:\,\text{there exist
$n\times n$ Hermitian matrices}\, A,B,C \,\text{with}\\
&e_A=a_1, e_B=a_2, e_C=a_3
\,\text{and}\, A+B=C\}.
\end{align*}
It is easy to see that  $\bar{\Gamma}(n)$ essentially coincides with the eigencone
$\bar{\Gamma}_3(sl(n))$ defined in Section \ref{section2}. Specifically,
$$(a_1,a_2,a_3)\in \bar{\Gamma}(n)\Leftrightarrow (a_1,a_2,a_3^*)\in
\bar{\Gamma}_3(sl(n)),$$
where $(e_1\geq  \dots \geq e_n)^*:= (-e_n\geq  \dots \geq -e_1).$

Combining Corollary \ref{eigen} with Theorem  \ref{thm4}, we get the following
Horn's conjecture [Ho] established by the works of Klyachko (Corollary \ref{eigen}
for $\fg=sl(n)$)  and Knutson-Tao (Theorem \ref{knutson-tao}).

\begin{corollary}\label{coro5}
For $(a_1,a_2,a_3)\in \fra_+^3$, the
following are equivalent.
\begin{itemize}
\item[\rm(a)] $(a_1,a_2,a_3)\in \bar{\Gamma}(n)$

\item[\rm(b)] For all $0<r<n$ and all $(I,J,K)\in S^{r}_{n}$,
$$|a_3(K)|\leq |a_1(I)|+|a_2(J)|,$$
where for a subset $I=(i_1<\dots <i_r)\subset \{1, \dots, n\}$ and
$a=(e_1\geq  \dots \geq e_n)\in \fra_+$,
$a(I):= (e_{i_1}\geq  \dots \geq e_{i_r}), \,\,\text{and}\,\,
|a(I)|:= e_{i_1}+ \dots +e_{i_r}.$
\end{itemize}
\end{corollary}
\begin{proof} Clearly $(a_1,a_2,a_3)\in \bar{\Gamma}(n) \Leftrightarrow
(a_1^*,a_2^*,a_3^*)\in \bar{\Gamma}(n)$. Thus, by Corollary \ref{eigen}
and Theorem \ref{thm4}, (a) is equivalent to the condition that for all
$0<r <n$ and $(I,J,K) \in S^r_n$,
\begin{equation}\label{7n}
\omega_{P_r}(w_{I'}^{-1}a_1^*+w_{J'}^{-1}a_2^*+w_{K}^{-1}a_3)\leq
  0.
  \end{equation}
  Now, since $\omega_{P_r}$ corresponds to $x_{P_r}$ under the isomorphism
  of $\fra^*$ with $\fra$ induced from the Killing form $\langle\,,\,\rangle$,
  the inequality \eqref{7n} is equivalent to
  \begin{equation}\label{8n}
\langle a_1^*, w_{I'}x_{P_r}\rangle +\langle a_2^*, w_{J'}x_{P_r}\rangle
+\langle a_3, w_{K}x_{P_r}\rangle\leq
  0.
  \end{equation}
Now, from the identity \eqref{e6}, the inequality \eqref{8n} is equivalent to
(since trace $a_1=$ trace $a_2=$ trace $a_3^*=0$ by assumption):
 $$|a_3(K)|\leq |a_1(I)|+|a_2(J)|.$$
 This proves the corollary.
 \end{proof}
 We have the following representation theoretic analogue of the above corollary,
 obtained by combining Theorems \ref{thm1}, \ref{knutson-tao} and \ref{thm4}.
 \begin{corollary} Let $\lambda=(\lambda_1 \geq \dots \geq \lambda_n\geq 0),
 \mu=(\mu_1 \geq \dots \geq \mu_n\geq 0)$ and $\nu=(\nu_1 \geq \dots \geq \nu_n\geq 0)$
 be three partitions such that $|\lambda|+|\mu|-|\nu|\in n\mathbb{Z}$, where
 $|\lambda|:=\lambda_1+ \dots +\lambda_n.$ Then, the following are equivalent:

 (a) $V(\nu)$ appears as a $\SL(n)$-submodule of $V(\lambda)\otimes V(\mu)$.

 (b) For all $0<r<n$ and all $(I,J,K)\in S^{r}_{n}$,
$$|\nu(K)|\leq |\lambda(I)|+|\mu(J)|-\frac{r}{n}(|\lambda|+|\mu|-|\nu|),$$
where for a subset $I=(i_1<\dots <i_r)\subset \{1, \dots, n\}$, $\lambda(I)$ denotes
$(\lambda_{i_1} \geq \dots \geq \lambda_{i_r})$ and $|\lambda(I)|:=
\lambda_{i_1}+ \dots + \lambda_{i_r}$.
\end{corollary}
\begin{proof} The condition $|\lambda|+|\mu|-|\nu|\in n\mathbb{Z}$ is equivalent
to the condition that $\lambda+\mu+\nu^*$ belongs to the root lattice of
$sl(n)$, where $\nu^*$ is the partition $(\nu_1-\nu_n \geq \dots \geq \nu_1-\nu_2
\geq 0 \geq 0)$. Moreover,
$V(\nu) \subset V(\lambda)\otimes V(\mu)$  (as an $\SL(n)$-submodule) if and only if
 $V(\nu^*) \subset V(\lambda^*)\otimes V(\mu^*)$  (as an $\SL(n)$-submodule).
 Thus, by Theorems \ref{thm1}, \ref{knutson-tao} and \ref{thm4}, (a) is equivalent to the
 condition that for all $0<r<n$ and all $(I,J,K)\in S^{r}_{n}$,
\begin{equation}\label{9n}
\lambda^*(w_{I'}x_{P_r}) +\mu^* ( w_{J'}x_{P_r})
+\nu ( w_{K}x_{P_r})\leq
  0.
  \end{equation}
By using the identity \eqref{e6}, the above inequality \eqref{9n} is equivalent to
$$|\nu(K)|\leq |\lambda(I)|+|\mu(J)|-\frac{r}{n}(|\lambda|+|\mu|-|\nu|).$$
This proves the corollary.
\end{proof}
\begin{definition}\label{defi6}
For $0<r<n$, inductively define $\hat{S}_{n}^{r}$ as the set of
triples $(I,J,K)$, where $I$, $J$, $K$ are subsets of $\{1,\ldots,n\}$
of cardinality $r$ satisfying the condition (b) of Definition
\ref{defi3} for $\hat{S}^{p}_{r}$ and the condition (a$'$) (instead of
the condition (a)).
\begin{equation*}
\sum_{i\in I}i+\sum_{j\in J}j\leq \frac{r(r+1)}{2}+\sum_{k\in K}k. \tag{a$'$}
\end{equation*}
\end{definition}

The following result is due to Belkale [B$_3$, Theorem
  0.1],
which is parallel to Theorem \ref{thm4}.

\begin{theorem}\label{thm7} Let $0<r<n$.
For subsets $(I,J,K)$ of $\{1,\ldots,n\}$ of cardinality $r$, the
product
$$[X_{I'}]\cdot [X_{J'}]\cdot [X_{K}]\quad \text{is nonzero}
\Leftrightarrow\quad  (I,J,K)\in \hat{S}^{r}_{n}.
$$
\end{theorem}

\begin{remark}   The Hermitian eigencone  $\bar{\Gamma}(n)$ has extensively been studied
 since the initial work of H. Weyl in $1912$ [W] followed by the works of
 Fan [Fa], Lidskii [Li], Wielandt [Wi] and culminating into the conjecture of
 Horn [Ho] and then its proof by combining the works of Klyachko [Kly] and
 Knutson-Tao [KT]. (Also see Thompson-Freede [TF].) For a detailed survey on the subject, we
 refer to Fulton's article [F$_2$].
\end{remark}

\section{Deformed product}\label{section6}

This section is based on the  work [BK$_1$] due to Belkale-Kumar.

We continue to
follow the notation
and assumptions from Secton \ref{sec1}; in particular,
 $G$ is a  semisimple  connected complex
algebraic group and $P\subset G$ is a standard parabolic subgroup.

Consider the shifted Bruhat cell:
\[
\Phi^P_w := w^{-1} BwP \subset G/P.
  \]
Let $T^P=T(G/P)_e$ be the tangent space of $G/P$ at $e\in G/P$. It carries a
canonical  action of $P$.  For $w\in W^P$, define $T_w^P$ to be the tangent space
of $\Phi_w^P$ at $e$. We shall abbreviate $T^P$ and $T_w^P$ by $T$ and $T_w$ respectively
when the reference to $P$ is clear. By \eqref{eqn1}, $B_L$ stabilizes $\Phi^P_w$ keeping $e$
fixed. Thus,
  \begin{equation}\label{eqn2}  B_L T_w \subset T_w.
\end{equation}
The following result follows easily from the Kleiman transversality theorem Theorem
\ref{kleiman} and
Proposition \ref{fultonproper} by observing that
$g \Phi^P_w$ passes through $e\Leftrightarrow
g\Phi^P_w = p\Phi^P_w$ for some $p\in P$.
  \begin{proposition}\label{Firstnew}
Take any $(w_1, \dots, w_s)\in (W^P)^s$  such that
\begin{equation}\label{expected}
\sum_{j=1}^s\,\codim \Phi_{w_j}^P
 \leq \dim G/P.
 \end{equation}
Then, the following three conditions are equivalent:

\begin{enumerate}
\item[(a)] $[X^P_{w_1}] \dots [X^P_{w_s}] \neq 0 \in H^*(G/P)$.
\item[(b)] For general $(p_1, \dots, p_s)\in P^s$, the intersection
 $p_1\Phi^P_{w_1}\cap \cdots \cap p_s\Phi^P_{w_s}$
 is transverse at $e$.
\item[(c)] For general $(p_1,\dots,p_s)\in P^s$,
 \[\dim(p_1T_{w_1} \cap \cdots \cap p_sT_{w_s}) = \dim G/P -\sum_{j=1}^s\,\codim \Phi_{w_j}^P .\]
\end{enumerate}
 The set of $s$-tuples in (b) as well as (c) is an open subset of $P^s$.
 \end{proposition}

\begin{definition}\label{D1}
Let $w_1, \dots, w_s\in W^P$ be such that
  \begin{equation}\label{First}
\sum_{j=1}^s\codim \Phi^P_{w_j}  = \dim G/P.
  \end{equation}
We then call the $s$-tuple $(w_1, \dots, w_s)$ {\it Levi-movable} for short $L$-{\it movable}
if, for general $(l_1, \dots, l_s)\in L^s$, the intersection
$l_1\Phi^P_{w_1}\cap \cdots \cap l_s\Phi^P_{w_s}$ is transverse at $e$.
\end{definition}

By Proposition ~\ref{Firstnew}, if $(w_1, \dots, w_s)$ is $L$-movable,
 then
$[X^P_{w_1}]\dots  [X^P_{w_s}] = d[X^P_{e}]$ in  $H^*(G/P)$,
for some nonzero $d$.

\begin{definition} \label{basicbundle} Let $w\in W^P$.  Since $T_w$ is a $B_L$-module
 (by \eqref{eqn2}), we have the  $P$-equivariant vector bundle $\mathcal{T}_w:=P \tset{B_L}
T_w$ on $P/B_L$. In particular, we have the $P$-equivariant
 vector bundle $\mathcal{T}:=P \tset{B_L} T$ and  $\mathcal{T}_w$ is canonically
 a $P$-equivariant subbundle of $\mathcal{T}$.
Take the top exterior powers $\det (\mt/\mathcal{T}_w)$
and $\det (\mathcal{T}_w)$, which are  $P$-equivariant
 line bundles on $P/B_L$. Observe that, since $T$ is a $P$-module, the
$P$-equivariant  vector bundle $\mathcal{T}$ is $P$-equivariantly isomorphic with the product bundle  $P/B_L \times T$ under the map $\xi: P/B_L \times T
\to \mathcal{T}$ taking $(pB_L,v)\mapsto [p, p^{-1}v]$, for
$p\in P$ and $v\in T$; where $P$ acts on $P/B_L \times T$ diagonally. We will
often identify  $ \mathcal{T}$  with the product bundle  $P/B_L \times T$ under  $\xi$.

For $w\in W^P$, define the character $\chi_w\in\Lambda$ by
\begin{equation}\label{eqn5new} \chi_w=\sum_{\beta\in (R^+\setminus R^+_\fl)\cap w^{-1}R^{+}} \beta \,.
\end{equation}
 Then, from [K$_1$, 1.3.22.3] and \eqref{eqn1},
\begin{equation}\label{eqn5}
\chi_w = \rho -2\rho^L + w^{-1}\rho ,
\end{equation}
where $\rho$ (resp. $\rho^L$) is half the sum of roots in $R^+$ (resp. in
$ R^+_\fl$).
\end{definition}

The following lemma is easy to establish.

\begin{lemma}\label{mt}
For $w\in W^P$,  as $P$-equivariant
 line bundles  on $P/B_L$, we have:
$\det (\mt/\mathcal{T}_w)=\cl_P (\chi_w).$
\end{lemma}

 Let $\mt_s$ be the $P$-equivariant product  bundle
$(P/B_L)^s \times T \to (P/B_L)^s$ under the diagonal action of $P$ on
$(P/B_L)^s \times T $. Then, $\mt_s$ is canonically  $P$-equivariantly
 isomorphic with the pull-back bundle $\pi_j^*(\mathcal{T})$, for any
$1\leq j\leq s$, where   $\pi_j: (P/B_L)^s  \to P/B_L$ is the projection onto the $j$-th factor.  For any  $w_1, \dots, w_s \in W^P$,
 we have a  $P$-equivariant map of vector bundles on $(P/B_L)^s$:
 \begin{equation}\label{map}
\Theta=\Theta_{(w_1, \dots, w_s)
}:\mt_s \to \oplus_{j=1}^s
\pi_j^*({\mt}/\mathcal{T}_{w_j})
\end{equation}
obtained as the direct sum of the projections $\mt_s \to
\pi_j^*({\mt}/\mathcal{T}_{w_j})$ under the identification $\mt_s\simeq
\pi_j^*(\mathcal{T})$.
Now, assume that  $w_1, \dots, w_s \in W^P$ satisfies the condition
\eqref{First}. In this case,
 we have
the same rank bundles on the two sides of the map (\ref{map}). Let $\theta$ be the
 bundle map obtained from $\Theta$ by taking the top exterior power:
 \begin{equation}\label{determinant}
\theta=\det(\Theta): \det \bigl(\mt_s\bigr) \to \det \bigl(
\mt/\mathcal{T}_{w_1}\bigr) \boxtimes \cdots \boxtimes
\det\bigl(\mt/\mathcal{T}_{w_s}\bigr).
  \end{equation}
Clearly, $\theta$
is  $P$-equivariant
 and hence one can view $\theta$ as a  $P$-invariant
element in $$H^0\Biggl( (P/B_L)^{s}, \det (\mt_s )^*\otimes
\Bigl(\det \bigl(\mt/\mt_{w_1})\boxtimes \cdots \boxtimes \det
\bigl(\mt/\mt_{w_s}\bigr) \Bigr)\Biggr)$$
\begin{equation}\label{GIT}
=H^0\bigl( (P/B_L)^{s}, \elal\bigr), \,\,\text{where}\,\,
\elal:=\cl_P (\chi_{w_1}-\chi_1) \boxtimes
\cdots \boxtimes \cl_P
(\chi_{w_s}).
\end{equation}

The following lemma follows easily from Proposition \ref{Firstnew}.
\begin{lemma}\label{movL} Let $(w_1, \dots, w_s)$ be an $s$-tuple of elements
of $W^P$ satisfying the condition \eqref{First}. Then, we have the following:

(a) The section $\theta$ is nonzero if and only if
$[X_{w_1}^P]\dots [X_{w_s}^P]\neq 0\in H^*(G/P).$

(b) The $s$-tuple $(w_1, \dots, w_s)$ is $L$-movable if and only if the section
 $\theta$ restricted to $(L/B_L)^s$ is not identically $0$.
\end{lemma}

\begin{proposition}\label{T1}
Assume that  $(w_1, \dots, w_s)\in (W^P)^s$  satisfies equation \eqref{First}.
Then, the following are equivalent.

(a) $(w_1, \dots, w_s)$ is $L$-movable.

(b) $[X^P_{w_1}]\dots  [X^P_{w_s}] = d[X^P_{e}]$ in  $H^*(G/P)$, for some nonzero $d$,
 and for each  $\alpha_i\in \Delta\setminus\Delta(P)$, we have
$$\bigl((\sum_{j=1}^s\,\chi_{w_j})-\chi_1\bigr)(x_{i})=0.$$
\end{proposition}

\begin{proof} (a)$\Rightarrow$(b):
Let $(w_1,\dots, w_s)\in (W^P)^s$ be $L$-movable. Consider the restriction
$\hat{\theta}$   of the $P$-invariant section
$\theta$
to $(L/B_L)^s$. Then, $\hat{\theta}$  is non-vanishing by the above lemma. But, for
$$H^0\left( (L/B_L)^{s},  \elal\right)^L$$
 to be nonzero, the center of $L$ should act trivially (under the diagonal action)
on  $\elal$
restricted to $(L/B_L)^s$, where $\elal$ is as in the identity \eqref{GIT}. This gives
$\sum_{j=1}^s\chi_{w_j}(h)=\chi_1(h),$
for all $h$ in the Lie algebra $\mathfrak z_L$ of the center of $L$; in particular,
for $h=x_i$ for $\al_i\in \Delta\setminus\Delta(P)$. Further, the assertion that
$[X^P_{w_1}]\dots  [X^P_{w_s}] = d[X^P_{e}]$, for some nonzero $d$, follows from
Proposition \ref{Firstnew} and the condition \eqref{First}.

(b)$\Rightarrow$(a): By the above lemma,  $\theta(\bar{p}_1, \dots, \bar{p}_s)\neq 0$,
for some $\bar{p}_j\in P/B_L$. Consider the central OPS of $L$:
$\delta (t):=\prod_{\alpha_i\in\Delta\setminus\Delta(P)} t^{x_i}.$
 For any $x=ulB_L\in P/B_L$, with $u\in U_P$ and $l\in L_P$,
$$\lim_{t\to 0}\delta(t)x=\lim_{t\to 0}\delta(t)u\delta(t)^{-1}
(\delta(t)l) B_L.$$
But, since $\beta (\dot\delta) >0$, for all $\beta \in R^+\setminus R^+_\mathfrak l$, we get
$\lim_{t\to 0}\delta(t)u\delta(t)^{-1}=1.$
Moreover, since  $\delta(t)$ is central in $L$, $\delta(t)l B_L=l B_L$; in
particular,
  the limit
$\lim_{t\to 0}\delta(t)lB_L$ exists and equals $lB_L$. Thus,
$\lim_{t\to 0}\delta(t)x$ exists and lies in $L/B_L$.

Let  $\bar{p}:=
(\bar{p}_1,\dots,\bar{p}_s)\in \exx :=(P/B_L)^s$. Then, by Lemma
\ref{l1} (since $\delta$ is central in $L$), we get
$$\begin{aligned}
\mu^{\elal}(\bar{p},\delta)
&=-\sum_{\alpha_i\in\Delta\setminus\Delta(P)}\bigl(
\bigl(\bigl(\sum_{j=1}^s\chi_{w_j}\bigr)-\chi_1\bigr)(x_i)\bigr)\\
&=0,\,\,\text{by assumption}.
\end{aligned}$$
Therefore, using Proposition \ref{propn14}(c) for $S=P$,   $\theta$ does not vanish at
$\lim_{t\to 0}\delta(t)\bar{p}.$
But, from the above, this limit exists as an element of $(L/B_L)^s$.
Hence, $(w_1, \dots, w_s)$ is $L$-movable by Lemma ~\ref{movL}.
\end{proof}

\begin{corollary}\label{product}
For any $u,v,w\in W^P$ such that $c^w_{u,v}\neq 0$ (cf. equation (\ref{constants})),  we have
\begin{equation}\label{eqn18}
 (\chi_w-\chi_u-\chi_v)(x_i)\geq 0,\,\,\text{ for each}\,\, \alpha_i\in \Delta\setminus\Delta(P).
\end{equation}

  \end{corollary}
  \begin{proof}
 By the assumption of the corollary and the identity \eqref{eqn1n},
$[X^P_u]\cdot[X^P_v]\cdot[X^P_{w_oww_o^P}]=d[X^P_e],
$ for some nonzero $d$ (in fact $d=c^w_{u,v}$). Thus, by taking $(w_1,w_2,w_3)=(u,v,w_oww_o^P)$
in Lemma  ~\ref{movL}, the section
$\theta$ is nonzero. Now, apply Proposition ~\ref{propn14}(b) for the OPS
$\delta(t)=t^{x_i}$ and Lemma ~\ref{l1} (together with the identity (\ref{eqn5}))
 to get the corollary.
  \end{proof}

The definition of the following deformed product $\odot_0$ (now known as the
{\it Belkale-Kumar product}) was arrived at from the
crucial concept of Levi-movability as in Definition \ref{D1}.
 This deformed product is used in determining
 the facets (codimension $1$ faces)  of  $\bar{\Gamma}_s
 (\frg)$.

 \begin{definition}
Let $P$ be any standard parabolic subgroup of $G$.
Write the standard cup product in $H^*(G/P, \Bbb Z)$ in the $\{[X_w^P]\}_{w\in W^P}$ basis as follows:
\begin{equation}\label{constants}
[X^P_u]\cdot[X^P_v]=\sum_{w\in W^P} c^w_{u,v}[X^P_w].
\end{equation}
Introduce the indeterminates ${\defpar}_i$ for each $\alpha_i\in
\Delta\setminus\Delta(P)$ and define a deformed cup product $\odot$
as follows:
$$\label{EE'}
[X^P_u] \odot [X^P_v]=
\sum_{w\in W^P} \bigl(\prod_{\alpha_i\in \Delta\setminus\Delta(P)}
{\defpar}_i^{(w^{-1}\rho-u^{-1}\rho -v^{-1}\rho -\rho)(x_i)}
\bigr)
c^w_{u,v} [X^P_w],
$$
where $\rho$ is the (usual) half sum of positive roots of $\frg$.

By  Corollary ~\ref{product} and the identity \eqref{eqn5}, whenever $c^w_{u,v}$ is nonzero,
the exponent of $\tau_i$ in the above is a nonnegative integer. Moreover, it is
 easy to see that
the product
$\odot$ is associative and clearly commutative.
This product should not be confused with the small quantum cohomology product
of $G/P$.

The  cohomology of $G/P$
 obtained by setting each ${\defpar}_i=0$ in
$(H^*(G/P, \Bbb Z)\otimes\Bbb{Z}[{\defpar}_i],\odot)$ is denoted by
$(H^*(G/P, \Bbb Z),\odot_0)$. Thus, as a $\Bbb Z$-module, it  is the same as the singular cohomology
 $H^*(G/P, \Bbb Z)$ and under the product $\odot_0$ it is associative (and commutative).
Moreover,  it continues to satisfy the Poincar\'e duality (cf. [BK$_1$, Lemma 16(d)]).

It should be remarked that, in general, the canonical pull-back map \
$H^*(G/P_2,
\mathbb{Z}) \to H^*(G/P_1, \mathbb{Z}),$  for $P_1\subset P_2$, does not respect the
product $\odot_0$.

In the $\{\epsilon^P_w\}_{w\in W^P}$ basis, by the identity \eqref{eqn1n}, the
deformed product takes the form
\begin{equation}\label{10n}
\epsilon^P_u \odot \epsilon^P_v=
\sum_{w\in W^P} \bigl(\prod_{\alpha_i\in \Delta\setminus\Delta(P)}
{\defpar}_i^{(u^{-1}\rho+v^{-1}\rho -w^{-1}\rho -\rho)(x_i)}
\bigr)
d^w_{u,v} \epsilon^P_w,
\end{equation}
where
$\epsilon^P_u \cdot \epsilon^P_v=\sum_{w\in W^P}\,d^w_{u,v} \epsilon^P_w.$

\end{definition}

\begin{lemma} \label{minuscule} Let $P$ be a cominuscule maximal standard parabolic
 subgroup of $G$ (i.e., the unique simple root $\alpha_P \in \Delta \setminus \Delta(P)$
 appears with coefficient
$1$ in the highest root of $R^+$). Then, the product
$\odot$ coincides with the cup product in $H^*(G/P)$.
\end{lemma}
\begin{proof} By the definition of $\odot$, it suffices to show that for any $u,v, w\in W^P$ such that $c_{u,v}^w\neq 0$,
\begin{equation} \label{eqn5.1}
 (\chi_w-(\chi_u+\chi_v))(x_P)= 0.
\end{equation}
By the definition of $\chi_w$ (cf. \eqref{eqn5new}), since $P$ is cominuscule,
\begin{equation} \label{eqn5.2}
 \chi_w (x_P)=  \mid w^{-1}R^+\cap \bigl(R^+\setminus R^+_{\mathfrak l}\bigr)\mid
=\codim(\Phi^{P}_w:G/P),
\end{equation}
where the last equality follows since
$$R(T_w)=w^{-1}R^+\cap (R^-\setminus R^-_\frl),$$
where $R^-:=R\setminus R^+$ and $R^-_\frl := R_\frl\setminus  R^+_\frl.$
 Moreover, since
 $c_{u,v}^w\neq 0$,
\begin{equation} \label{eqn5.3}
\codim(\Phi^{P}_u:G/P)+\codim(\Phi^{P}_v:G/P)
=\codim(\Phi^{P}_w:G/P).
\end{equation}
Combining equations (\ref{eqn5.2}) and (\ref{eqn5.3}), we get equation (\ref{eqn5.1}).

Alternatively, one can prove the lemma by observing that the unipotent radical
$U_P$ of $P$ acts trivially on the tangent space $T_P(G/P)$ and using the
definition of Levi-movability together with Proposition \ref{Firstnew}.
\end{proof}
\begin{remark}
 Belkale-Kumar have given a criterion (though only necessary
conditions) to determine when the deformed product of a number of Schubert
cohomology classes
in any $G/P$ is nonzero.
The criterion is in terms of the characters (cf. [BK$_1$, Theorem 32]).
\end{remark}

\section{Efficient determination of the eigencone}\label{section7}

This section is again based on the work [BK$_1$] due to Belkale-Kumar. The following
theorem [BK$_1$, Theorem 22] determines the saturated tensor semigroup $\Gamma_s(G)$
efficiently. Specifically, as proved by Ressayre (see Corollary \ref{6.4}), the set
of inequalities given by (b) of the following theorem (resp. (b) of Corollary
\ref{7.5}) is an irredundant set of inequalities determining $\Gamma_s(G)$
(resp. $\bar{\Gamma}_s(\frg)$).

For $G=\SL(n)$, each maximal parabolic $P$ is cominuscule, and hence, by Lemma
 \ref{minuscule}, $\odot_0$ coincides with the standard cup product in $H^*(G/P)$. Thus,
 the following theorem (resp. Corollary \ref{7.5}) in this case reduces to Theorem
 \ref{thm1} (resp. Corollary \ref{eigen}) with $d=1$ in the identity \eqref{3n}.
 In this case the irredundancy of the system was proved by Knutson-Tao-Woodward
 [KTW].

 It may be mentioned that replacing the product $\odot_0$
 in the (b)-part of the following theorem  by the standard
cup product (i.e., Theorem \ref{thm1} with $d=1$ in the identity \eqref{3n}; cf.
Remark \ref{2.9} (b)), we
get, in general,  `far more' inequalities for simple groups other than  $\SL(n)$.
For example,
for $G$ of type $B_3$ (or $C_3$), Theorem \ref{thm1} with $d=1$ gives rise to
$126$ inequalities, whereas the following theorem gives
only $93$ inequalities (cf. [KuLM]).

\begin{theorem}\label{EVT} Let $G$ be a connected semisimple group and let  $(\lambda_1, \dots, \lambda_s)\in \Lambda_+^s$.
Then, the following are equivalent:

(a)  $\lambda=(\lambda_1, \dots, \lambda_s)\in \Gamma_s(G)$.

(b) For every standard maximal parabolic subgroup $P$ in $G$ and every choice of
$s$-tuples  $(w_1, \dots, w_s)\in (W^P)^s$ such that
$$[X^P_{w_1}]\odot_0\, \cdots \,\odot_0 [X^P_{w_s}] = [X^P_e]
\in \bigl(H^*(G/P,\Bbb{Z}), \odot_0\bigr),$$
  the following inequality  holds:
    \[\label{eqn29}
\sum_{j=1}^s \lambda_j(w_jx_{P})\leq 0, \tag{$I^P_{(w_1,\dots, w_s)}$}
\]
where $\alpha_{i_P}$ is the (unique) simple root in $\Delta\setminus \Delta (P)$
and $x_P:=x_{i_P}$.
\end{theorem}
Before we come to the proof of the theorem, we need the following.

\begin{definition} (Maximally destabilizing one parameter subgroups)\label{Kempf}
We recall the definition of  Kempf's OPS attached to an unstable point, which is
in some sense `most destabilizing' OPS.
Let $\exx$ be a projective variety with the action of a connected reductive group $S$ and let
$\elal$ be  a $S$-linearized ample line bundle on $\exx$.
Introduce the set $M(S)$ of fractional OPS in $S$. This is the set
 consisting of the ordered pairs $(\delta,a)$, where $\delta\in O(S)$
and $a\in \Bbb{Z}_{>0}$, modulo the equivalence relation $(\delta,a)\simeq(\gamma,b)$ if
$\delta^b=\gamma^a$. The equivalence class of $(\delta, a)$ is denoted by
$[\delta,a]$. An OPS $\delta$ of $S$ can be thought of as the element $[\delta,1]\in M(S)$. The
group $S$ acts on $ M(S)$ via conjugation: $g\cdot[\delta,a]=[g\delta g^{-1},a]$. Choose
a $S$-invariant norm $q:M(S)\to \Bbb R_+$.
We can extend the definition of
$\mu^{\elal}(x,\delta)$ to any element $\hat{\delta}=[\delta,a]\in M(S)$ and $x\in \exx$
by setting $\mu^{\elal}(x,\hat{\delta})= \frac{\mu^{\elal}(x,\delta)}{a}$.
We note the following elementary property: If $\hat{\delta}\in M(S)$ and $p\in
P(\delta)$ (where $P(\delta)$ is the Kempf's parabolic defined by the identity
\eqref{2n}),
then
\begin{equation}\label{translate}
\mu^{\elal}(x,\hat{\delta})=\mu^{\elal}(x,p\hat{\delta} p^{-1}).
\end{equation}
For any unstable (i.e., nonsemistable) point $x\in \exx$, define
$$q^*(x)=\inf_{\hat{\delta}\in M(S)} \{q(\hat{\delta})\mid \mu^{\elal}(x,\hat{\delta})\leq -1\},$$
and the {\it optimal class}
$$\Lambda(x)=\{\hat{\delta}\in M(S)\mid \mu^{\elal}(x,\hat{\delta})\leq -1, q(\hat{\delta})=q^*(x)\}.$$
Any $\hat{\delta} \in \Lambda(x)$ is called {\it Kempf's OPS associated to} $x$.

By a theorem of Kempf (cf. [Ki, Lemma 12.13]), $\Lambda(x)$ is nonempty and the parabolic $P(\hat{\delta}):=P(\delta)$
(for $\hat{\delta}=[\delta,a]$) does
not depend upon the choice of $\hat{\delta} \in \Lambda(x)$.
The parabolic $P(\hat{\delta})$ for $\hat{\delta}\in \Lambda(x)$ will be denoted by $P(x)$ and
called the {\it Kempf's parabolic associated to the unstable point} $x$. Moreover,
$\Lambda(x)$ is a single conjugacy class under $P(x)$.
\end{definition}
We recall the following theorem due to Ramanan-Ramanathan [RR, Proposition 1.9].
\begin{theorem}\label{RR}
For any  unstable point  $x \in \exx$ and $\hat{\delta}=[\delta,a]\in \Lambda(x)$, let
 $$x_o=\lim_{t\to 0}\,\delta(t)\cdot x \in \exx.$$
Then, $x_o$ is unstable and $\hat{\delta}\in\Lambda(x_o)$.
\end{theorem}

For a real number $d>0$ and $\hat{\delta}=[\delta,a]\in M(S)$, define
$$
\mathbb{X}_{d,\hat{\delta}}=\mathbb{X}^{\mathbb{L}}_{d,\hat{\delta}}:=\{x\in
\mathbb{X}:q^{*}(x)=d\,\,\,\text{and}\,\,\, \hat{\delta}\in \Lambda (x)\},
$$
and
$$
Z_{d,\hat{\delta}}=Z^{\mathbb{L}}_{d,\hat{\delta}}:=
\{x\in \mathbb{X}_{d,\hat{\delta}}:\delta \text{~ fixes~ } x\}.
$$

By Theorem~\ref{RR}, we have the map
$$
p_{\hat{\delta}}:\mathbb{X}_{d,\hat{\delta}}\to Z_{d,\hat{\delta}},\,\, x\mapsto
\lim_{t\to 0}\,\delta(t)\cdot x.
$$
We recall the following result from [Ki, $\S$13].

\begin{proposition}\label{ressprop1}
Let $\mathbb{X}$ and $\mathbb{L}$ be as above. Assume further that $\mathbb{X}$
 is smooth. Then, we have the following:
\begin{itemize}
\item[\rm(a)] $Z_{d,\hat{\delta}}$ is an open $S^{\delta}$-stable subset of
$\mathbb{X}^{\delta}$, where $S^{\delta}$ is the centralizer of $\delta$ in $S$.

\item[\rm(b)] $\mathbb{X}_{d,\hat{\delta}}=\{x\in \mathbb{X}:
\lim_{t\to 0}\,\delta(t)\cdot x\in Z_{d,\hat{\delta}}\}$, and it is stable
under $P(\delta)$.

\item[\rm(c)] There is a bijective morphism
$$
S\times^{P(\delta)}\mathbb{X}_{d,\hat{\delta}}\to
\mathbb{X}_{d,\langle \hat{\delta}\rangle},
$$
which is an isomorphism if $\mathbb{X}_{d,\langle \hat{\delta}\rangle}$ is normal,
where
$$
\mathbb{X}_{d,\langle \hat{\delta}\rangle}=\bigcup_{g\in S}
 \mathbb{X}_{d,\,g\cdot \hat{\delta}}.
$$
\end{itemize}
\end{proposition}

Let $\langle M(S)\rangle$ denote the $S$-conjugacy classes in
$M(S)$. We have the following result due to Hesselink [He].

\begin{proposition}\label{birula}
For $\mathbb{X}$ and $\mathbb{L}$ as in Proposition~\ref{ressprop1},
$$
\mathbb{X}=\mathbb{X}^{s}\bigcup \bigcup_{{d>0, \, \langle
    \hat{\delta}\rangle \in \langle M(S)\rangle}}\mathbb{X}_{d,\langle \hat{\delta}\rangle}
$$
is a finite stratification by locally-closed $S$-stable subvarieties
of $\mathbb{X}$, where $\mathbb{X}^{s}$ is the set of semistable points
of $\mathbb{X}$ with respect to the ample line bundle $\mathbb{L}$.
\end{proposition}

\noindent
{\bf Proof of Theorem ~\ref{EVT}:}
Let  $\elal$ denote the $G$-linearized line bundle $\cl(\lam_1)\boxtimes
\cdots \boxtimes\cl(\lam_s)$ on
 $(G/B)^s$ and let
 $P_1, \dots, P_s$ be the standard parabolic subgroups such that
$\elal$ descends as an ample line bundle  $\bar{\elal}$ on
$\xx (\lambda):=G/P_1 \times \cdots \times G/P_s$. As earlier, we call a point
 $x\in (G/B)^s$ {\it semistable} (with respect to, not necessarily ample,
$\elal$) if its image in $\xx (\lambda)$ under the canonical map $\pi: (G/B)^s
\to \xx (\lambda)$ is semistable. Since
 the map $\pi$ induces an
 isomorphism of $G$-modules:
\begin{equation}\label{equalitie}
H^0(\exx (\lambda),\bar{\Bbb{L}}^N) \simeq H^0((G/B)^s,\Bbb{L}^N), \forall N>0,
\end{equation}
the condition ($a$) of Theorem
\ref{EVT} is equivalent to the following condition:
\begin{enumerate}
\item[($c$)]The set of semistable points of
$(G/B)^s$ with respect to $\Bbb{L}$ is nonempty.
\end{enumerate}

The implication $(a)\Rightarrow (b)$ of Theorem \ref{EVT} is of course a special
 case of Theorem  \ref{thm1}.

To prove the implication $(b)\Rightarrow (a)$ in Theorem \ref{EVT}, we need to recall
 the following result due to Kapovich-Leeb-Millson [KLM$_1$]. (For a self-contained algebro-geometric
 proof of this result, see [BK$_1$, $\S$7.4].)
Suppose that $x=(\bar{g}_1,\dots,\bar{g}_s) \in (G/B)^s$ is an unstable point and $P(x)$
the Kempf's parabolic associated to $\pi(x)$. Let $\hat{\delta}=[\delta, a]$ be a Kempf's OPS
associated to $\pi(x)$. Express $\delta(t)=f\gamma(t)f^{-1}$, where $\dot{\gamma}\in
\frh_{+}$.
Then, the Kempf's parabolic $P(\gamma)$ is a standard parabolic.
 Define $w_j\in W/W_{P(\gamma)}$ by $fP(\gamma)\in g_j Bw_jP(\gamma)$ for
$j=1,\dots,s$. Let $P$ be a maximal parabolic containing
$P(\gamma)$.
\begin{theorem}\label{jm}
\begin{enumerate}
\item[($i$)]  The intersection $\bigcap_{j=1}^s  g_jBw_j P\subset G/P $ is the singleton $\{fP\}$.
\item[($ii$)] For the simple root   $\alpha_{i_P}\in \Delta\setminus \Delta(P)$,
$\sum_{j=1}^s \lam_j(w_j x_{i_P})>0.$
\end{enumerate}
\end{theorem}

 Now, we come to the  proof of the implication $(b)\Rightarrow (a)$ in Theorem \ref{EVT}.
Assume, if possible, that ($a$) (equivalently ($c$) as above) is false, i.e., the set
of semistable points of $(G/B)^s$ is empty. Thus,  any point $x=(\bar{g}_1,\dots,\bar{g}_s)\in (G/B)^s$
 is unstable. Choose a general $x$
 so that for each standard parabolic $\tilde{P}$ in $G$ and any $(z_1,\dots, z_s)\in W^s$, the
intersection
${g_1Bz_1\tilde{P}}\cap \dots\cap {g_sBz_s\tilde{P}}$ is
transverse (possibly empty) and dense in ${g_1\overline{Bz_1\tilde{P}}}\cap
 \dots\cap {g_s\overline{Bz_s\tilde{P}}}$ (cf. Theorem \ref{kleiman}). Let $\hat{\delta}=[\delta,a], P,\gamma, f, w_j$ be as above
 associated to $x$. It follows from Theorem ~\ref{jm} that $\bigcap_{j=1}^s
 g_jBw_j P\subset G/P$ is the single point $f$ and, since $x$ is general, we get
\begin{equation}\label{expdim}
[X^{P}_{w_1}] \dots  [X^{P}_{w_s}] = [X^P_{e}] \in H^*(G/P,\Bbb{Z}).
\end{equation}
 We now claim that the
$s$-tuple  $(w_1, \dots, w_s)\in (W/W_P)^s$ is $L$-movable.

Write $g_j=fp_jw_j^{-1}b_j$, for some $p_j\in P(\gamma)$ and $b_j\in B$
(where we have abused the notation to also denote a lift of $w_j$ in $N(H)$ by $w_j$). Hence,
$$\delta(t) \bar{g}_j=f \gamma(t) p_jw_j^{-1}B=f\gamma(t) p_j\gamma^{-1}(t)w_j^{-1}B\in G/B.$$
Define,
$l_j = \lim_{t\to 0} \gamma(t)p_j\gamma^{-1}(t).$
Then,  $l_j \in L(\gamma).$
Therefore,
$$\lim_{t\to 0}\delta(t)x=(f l_1 w_1^{-1}B,\dots, f l_s w_s^{-1}B).$$

By Theorem ~\ref{RR}, $\hat{\delta}\in \Lambda(\pi(\lim_{t\to 0}\delta(t)x))$.
We further note that clearly
$fP(\gamma)\in \cap_j(fl_j w_j^{-1})B w_j P(\gamma).$

Applying Theorem ~\ref{jm} to the   unstable point $x_o=\lim_{t\to 0} \delta(t)x$ yields:
 $fP$ is the only point in the  intersection
$\bigcap_{j=1}^s f l_j w_j^{-1}Bw_j P$, i.e.,
translating by $f$, we get:
$\dot{e}=eP$ is the only point in the intersection
$\Omega:=\bigcap_{j=1}^s
l_j w_j^{-1}Bw_j P.$ Thus, $\dim \Omega=0$.
By (\ref{expdim}), the  expected dimension of
$\Omega$ is $0$ as well.
 If this intersection $\Omega$ were not transverse at $\dot{e}$,
 then by  [F$_1$, Remark 8.2], the local multiplicity
at $\dot{e}$ would be  $>1$, each $w_j^{-1}Bw_jP$ being  smooth.
Further, $G/P$ being a homogeneous space, any other component of the intersection
$\bigcap
l_j \overline{w_j^{-1}Bw_j P}$
contributes nonnegatively to the intersection product
$[X^P_{w_1}] \dots [X^P_{w_s}]$
(cf. Proposition \ref{fultonproper}). Thus, from  (\ref{expdim}), we get that
the intersection   $\bigcap l_jw_j^{-1} Bw_jP$ is  transverse
 at $e\in G/P$, proving that  $(w_1,\dots,w_s)$ is $L$-movable.
Now, part (ii)  of Theorem ~\ref{jm} contradicts the inequality
$I^P_{(w_1, \dots, w_s)}$. Thus, the set of semistable points of $(G/B)^s$ is
nonempty, proving condition ($a$) of Theorem
\ref{EVT}.
\qed

The following result follows easily by combining Theorems \ref{EVT} and \ref{sj}.
For a maximal parabolic $P$, let $\alpha_{i_P}$ be the unique simple root
not in the Levi of $P$ and we set  $\omega_P:= \omega_{i_P}$.
\begin{corollary}\label{7.5}  Let $(h_1,\dots,h_s)\in\frh_{+}^s$.  Then, the following
are equivalent:

(a) $(h_1,\dots,h_s)\in\bar\Gamma_s(\fg)$.

(b)
 For every standard maximal parabolic subgroup $P$ in $G$ and every choice of
  $s$-tuples  $(w_1, \dots, w_s)\in (W^P)^s$ such that
$$[X^P_{w_1}]\odot_0 \cdots \odot_0 [X^P_{w_s}] = [X^P_{e}],$$
the following inequality holds:
\begin{equation}\label{4n}\omega_P(\sum_{j=1}^s\,w_j^{-1}h_j)\leq 0.
\end{equation}
\end{corollary}
\begin{remark}
The cone $\bar{\Gamma}_3(\frg)\subset \frh_+^3$ is quite explicitly determined
  for any simple $\frg$ of rank
 $2$ in [KLM$_1$, $\S$7]; any simple $\frg$ of rank $3$
 in [KuLM] (reproduced in Section \ref{sec10});  and for $\frg=so(8)$ in [KKM].
 It has $12 (6+6); 18 (9+9); 30 (15+15); 41 (10+21+10);  93 (18+48+27); 93 (18+48+27); 294
 (36+186+36+36); 1290 (126+519+519+126); 26661 (348+1662+4857+14589+4857+348)$
 facets inside  $\frh_+^3$ (intersecting the interior of $\frh_+^3$) for $\frg$
 of type $A_2; B_2; G_2;
 A_3; B_3; C_3; D_4; F_4; E_6$  respectively. The notation $30 (15+15)$ means
 that there are $15$  (irredundant) inequalities coming from $G/P_1$ and there are $15$
 inequalities coming from $G/P_2$ via Corollary \ref{7.5} (b).
  (The indexing convention is as in [Bo, Planche
 I - IX].)
\end{remark}

\section{Study of the saturated restriction semigroup and irredundancy of its
inequalities}\label{section8}
This section is based on the work of Ressayre [R$_1$] (also see [Br]).

Let $G \subset \G$ be connected reductive complex algebraic groups. We fix a maximal torus $H$ (resp. $\hat{H}$) and a
Borel subgroup $H\subset B$ (resp. $\hat{H} \subset \hat{B}$) of $G$ (resp. $\G$) such that $H\subset \hat{H}$ and
  $B\subset \hat{B}$. We shall follow the notation from Section \ref{sec1} for $G$ and the corresponding objects for
$\G$ will be denoted by a hat on the top.

Define the {\it saturated restriction
semigroup}
$$
\Gamma (G, \G)=\left\{(\lambda, \hat{\lambda})\in \Lambda_+\times \hat{\Lambda}_+:
      \left[V(N\lambda)\otimes \hat{V}( N\hat{\lambda})\right]^{G}\neq 0,\text{~ for some ~} N\geq 1\right\}.
$$

The aim of this section is to determine this semigroup in terms of an irredundant
system of inequalities.

\begin{lemma}\label{lemma7.1ress}
The interior of $\Gamma(G,\hat{G})_{\mathbb{R}}$ inside
$\Lambda_+(\br)\times \hat{\Lambda}_+(\br)$ is nonempty if and only
if no non-zero ideal of $\frg$ is an ideal of
$\hat{\frg}$, where $\Gamma(G,\hat{G})_{\mathbb{R}}$ is the
cone inside $\Lambda_+(\br)\times \hat{\Lambda}_+(\br)$ generated by
$\Gamma(G,\hat{G})$ and $\Lambda_+(\br)$ is the cone inside $\Lambda \otimes_\bz
\,\br$ generated by $\Lambda_+$ and $\hat{\Lambda}_+(\br)$ has a similar meaning.
\end{lemma}

\begin{proof}
By [MR, Corollaire 1], the codimension of
$\Gamma(G,\hat{G})_{\mathbb{R}}$ in $\Lambda_+(\br)\times \hat{\Lambda}_+(\br)$ is
the dimension of the kernel $H_{o}$ of
the adjoint action
$
\Ad : H\to \Aut (\hat{G}/G).
$
Clearly,
$
H_{o}=\bigcap_{\hat{g}\in\hat{G}}\hat{g}G
\hat{g}^{-1},
$
which is a normal subgroup of $\hat{G}$ contained in $G$.
Moreover, any normal subgroup $N$ of $G$ which is also normal in
$\hat{G}$ is of course contained in $H_{o}$. This proves the lemma.
\end{proof}
\begin{remark} A stronger result than the above lemma is proved in [PR, Theorem 4].
\end{remark}

For any $G$-dominant OPS $\delta\in O(H)$, (i.e., $\dot{\delta}\in \frh_+$), let $P(\delta)$ (resp. $\hat{P}(\delta)$)
be the Kempf's parabolic associated to $\delta$ inside $G$ (resp. $\G$), cf. the identity
\eqref{2n}. Since $\delta$ is dominant for $G$,  $P(\delta)$ is a standard parabolic
subgroup of $G$.

Analogous to the Definition \ref{D1}, we define the following.
\begin{definition}\label{bknewdef} Let $(w,\hat{w})\in W^{P(\delta)}\times
\hat{W}^{\hat{P}(\delta)}$ be such that
$$\ell(\hat{w})+\ell (w)=\dim \hat{G}/\hat{P}(\delta).$$
Then, we say that $(w,\hat{w})$ {\it is $L$-movable} if the canonical map
$$T_{e}(\Phi_w^{P(\delta)})
\xrightarrow{{(d\iota)_e}}\dfrac{T_{\hat{e}}(\G/\hat{P}(\delta))}{T_{\hat{e}}
(\hat{l}\,\hat{\Phi}^{\hat{P}(\delta)}_{\hat{w}})}$$
is an isomorphism for some $\hat{l}\in \hat{L}(\delta)$,
where $e$ (resp. $\hat{e}$) is the base point $1\cdot P(\delta)\in G/P(\delta)$
(resp. $1\cdot \hat{P}(\delta)\in \G/\hat{P}(\delta)$).
\end{definition}
For any $w\in W/W_{P(\delta)}$, let
$\gamma^{P(\delta)}_w$ be the sum of the $H$-weights in the normal space
$T_e(G/P(\delta))/T_e(\Phi_w^{P(\delta)})$. We similarly define
$\hat{\gamma}^{\hat{P}(\delta)}_{\hat{w}}$ for any $\hat{w}\in
\hat{W}/\hat{W}_{\hat{P}(\delta)}$. Then, it is easy to see from Lemma \ref{mt}
(since  $\delta$ is $G$-dominant) that
\begin{equation}\label{ress1900}
\gamma^{P(\delta)}_w (\dot{\delta})= -(\rho +w^{-1}\rho)(\dot{\delta}).
\end{equation}
Moreover,
\begin{equation}\label{ress1901}
\hat{\gamma}^{\hat{P}(\delta)}_{\hat{w}}(\dot{\delta})= -(\hat{v}^{-1}\hat{\rho}
+\hat{w}^{-1}\hat{\rho})(\dot{\delta}),
\end{equation}
where $\hat{v}\in \hat{W}$ is such that  $\hat{v}(\dot{\delta})\in \hat{\frh}_+$.

We have the following result analogous to Proposition~\ref{T1}.

\begin{proposition}\label{prop7.3ress}
Let $(w,\hat{w})\in W^{P(\delta)}\times \hat{W}^{\hat{P}(\delta)}$ be
such that
$$
\ell(\hat{w})+\ell(w)=\dim \hat{G}/\hat{P}(\delta).
$$
Then, the following are equivalent:
\begin{itemize}
\item[\rm(a)] $(w,\hat{w})$ is $L$-movable for the embedding
$
\iota:G/P(\delta)\to \hat{G}/\hat{P}(\delta),
$

\item[\rm(b)] $
  [X^{P(\delta)}_{w}]\cdot \iota^{*}[\hat{X}_{\hat{w}}^{\hat{P}(\delta)}]=d[pt]$,
  for some $d\neq
  0$, and
$$
\hat{\gamma}^{\hat{P}(\delta)}_{\hat{w}}(\dot{\delta})=
\gamma^{P(\delta)}_{e}(\dot{\delta})-\gamma^{P(\delta)}_{w}(\dot{\delta}).
$$
\end{itemize}
\end{proposition}

\begin{proof}
Let $T$ (resp. $\hat{T}$) be the tangent space of $G/P(\delta)$
(resp. $\hat{G}/\hat{P}(\delta)$) at the base point $1\cdot P(\delta)$
(resp. $1\cdot \hat{P}(\delta)$). Similarly, let $T_{w}$
(resp. $\hat{T}_{\hat{w}}$) be the tangent space of
$\Phi^{P(\delta)}_{w}$
(resp. $\hat{\Phi}^{\hat{P}(\delta)}_{\hat{w}}$) at the base
point. Then, $\hat{T}$ and $\hat{T}_{\hat{w}}$ are
$\hat{B}_{\hat{L}(\delta)}$ modules since $\hat{B}_{\hat{L}(\delta)}$ stabilizes
$\hat{\Phi}^{\hat{P}(\delta)}_{\hat{w}}$ keeping the base point
$1\cdot \hat{P}(\delta)$ fixed, where $\hat{B}_{\hat{L}(\delta)}$
is the Borel subgroup $\hat{B}\cap \hat{L}(\delta)$ of
$\hat{L}(\delta)$.

Let $\hat{\mathcal{T}}$
(resp. $\hat{\mathcal{T}}/\hat{\mathcal{T}}_{\hat{w}}$)
be the vector bundle
$\hat{P}(\delta) \times^{\hat{B}_{\hat{L}(\delta)}}\,\hat{T}$
(resp. $\hat{P}(\delta)\times^{\hat{B}_{\hat{L}(\delta)}}\,
(\hat{T}/\hat{T}_{\hat{w}})$)
over the base space
$\hat{P}(\delta)/\hat{B}_{\hat{L}(\delta)}$. For any vector
space $V$, we let $\epsilon(V)$ be the trivial vector bundle
$\hat{P}(\delta)/\hat{B}_{\hat{L}(\delta)}\times V$ over
$\hat{P}(\delta)/\hat{B}_{\hat{L}(\delta)}$.

We have the $B_{L(\delta)}$-equivariant bundle map
$$
\Theta: \epsilon (T_{w})\to
\hat{\mathcal{T}}/\hat{\mathcal{T}}_{\hat{w}}
$$
obtained as the composition
$$
\epsilon (T_{w})\hookrightarrow \epsilon
(\hat{T})\displaystyle{\mathop{\rightarrow}^{\alpha}}\,
\hat{\mathcal{T}}\to
  \hat{\mathcal{T}}/\hat{\mathcal{T}}_{\hat{w}},
$$
where $B_{L(\delta)}:=B\cap L(\delta), B_{L(\delta)}$ acts on
$\epsilon (T_{w})$ diagonally, the first map is the canonical inclusion, the
last map is induced
by the projection and the $\hat{P}(\delta)$-equivariant isomorphism $\alpha$ is
given by
$$
\alpha (p\hat{B}_{\hat{L}(\delta)}, v)=[\hat{p},\hat{p}^{-1}v],\quad\text{for}\,\,
\hat{p}\in \hat{P}(\delta), v\in \hat{T}.
$$
(Observe that $\hat{T}$ is
canonically a $\hat{P}(\delta)$-module.)

By assumption, the map $\Theta$ is a bundle map between the bundles
of the same rank. Hence, $\Theta$ induces a bundle map $\theta$ by
taking the top exterior powers
$$
\theta : \epsilon (\det T_{w})\to
\det(\hat{\mathcal{T}}/\hat{\mathcal{T}}_{\hat{w}}),
$$
which can be viewed as a ${B}_{L(\delta)}$-invariant section in
$$
H^{0}\left(\hat{P}(\delta)/\hat{B}_{\hat{L}(\delta)},(\epsilon (\det
T_{w})^{-1})\otimes \det (\hat{\mathcal{T}}/\hat{\mathcal{T}}_{\hat{w}})\right).
$$
By definition, $(w,\hat{w})$ is $L$-movable if and only if the section
$\theta_{|({\hat{L}(\delta)}/\hat{B}_{\hat{L}(\delta)})}\neq 0.$
Now, the rest of the proof of this proposition is identical to the
proof of Proposition~\ref{T1} and Lemma~\ref{movL}, since the image of $\delta$ is
central in $\hat{L}(\delta)$. (Since  $\Imo \,\delta$ is
central in $\hat{L}(\delta)$, it is easy to see, by the same proof as that of Lemma
\ref{l1}, that
$$\mu^{(\epsilon (\det
T_{w})^{-1})\otimes \det (\hat{\mathcal{T}}/\hat{\mathcal{T}}_{\hat{w}})}
(\hat{p},\delta)= \hat{\gamma}^{\hat{P}(\delta)}_{\hat{w}}(\dot{\delta})-
\gamma^{P(\delta)}_{e}(\dot{\delta})+\gamma^{P(\delta)}_{w}(\dot{\delta}).)
$$
This proves the proposition.
\end{proof}

For any $\delta\in 0(H)$, the centralizer of $\delta$ in $G$:
$$
G^{\delta}:=\{g\in G:g\delta(t)=\delta(t)g,\quad\text{for all}\quad
t\in \mathbb{G}_{m}\}
$$
is the Levi subgroup $L(\delta)$ (containing $H$) of the Kempf's
parabolic subgroup $P(\delta)$. Let $Y$ be a smooth projective $G$-variety. Let
$C$ be
an irreducible component of $Y^{\delta}$ and let $$C_{+}=\{y\in
Y:\lim\limits_{t\to 0}\delta(t)\cdot y\,\,\text{ lies in}\,
C\}.$$
 Then, $C$ is a closed smooth $L(\delta)$-stable subvariety of $Y$
 (since $L(\delta)$ is connected); $C_{+}$ is a $P(\delta)$-stable, smooth, irreducible, locally-closed
subvariety of $Y$ (by a result of Bialynicki-Birula); and the map $\pi_\delta:
C_+ \to C,\,
y \mapsto \lim\limits_{t\to 0}\delta(t)\cdot y$ is a morphism.

Consider the $G$-equivariant morphism
$$\eta:G\times^{P(\delta)} C_{+} \rightarrow
Y,\,
[g,y]\mapsto g\cdot y.$$
 The following definition is due  to Ressayre
[R$_1$].

\begin{definition}\label{wellcovering}
The pair $(C,\delta)$ is called a {\em well-covering pair} if there
exists a $P(\delta)$-stable open (irreducible) subset $C^{o}_{+}$ of
$C_{+}$ such that $C^{o}_{+}\cap C$ is nonempty and the map
$
\eta_{o}=\eta_{|(G\times^{P(\delta)}\,C^{o}_{+})}
$
is an isomorphism onto an open subset of $Y$.
\end{definition}

Now, we take $Y=X:=G/B\times \hat{G}/\hat{B}$ with the diagonal action of $G$ and
let $\delta\in O(H)$ be a $G$-dominant OPS. It is easy to see that
\begin{equation}\label{resse1} X^{\delta}=\sqcup \,C_{\delta}(w,\hat{w}),
\end{equation} where
\begin{equation}\label{resse2}
C_{\delta}(w,\hat{w}):=(L(\delta)\cdot w^{-1}B/B)\times
(\hat{L}(\delta)\cdot \hat{w}^{-1}\hat{B}/\hat{B}),
\end{equation}
and the union runs over $(w,\hat{w})\in W^{P(\delta)}\times
\hat{W}^{\hat{P}(\delta)}$. Further, it is easy to see that
\begin{equation}\label{resse3}
C_{\delta}(w,\hat{w})_{+}=(P(\delta)\cdot w^{-1}B/B)\times
(\hat{P}(\delta)\cdot \hat{w}^{-1}\hat{B}/\hat{B}).
\end{equation}

\begin{lemma} \label{lemma7.5ress} For any $(w,\hat{w})\in W^{P(\delta)}\times
\hat{W}^{\hat{P}(\delta)}$, the following are equivalent:

(a)
The pair $(C_{\delta}(w,\hat{w}),\delta)$ is a well-covering pair.

(b) The pair  $(w,\hat{w})$ is $L$-movable for the embedding
$
\iota:G/P(\delta)\hookrightarrow
\hat{G}/\hat{P}(\delta)$ and
$$[X^{P(\delta)}_{w}]\cdot \iota^{*}[\hat{X}^{\hat{P}(\delta)}_{\hat{w}}] =[pt].
$$
\end{lemma}

\begin{proof}
The projection
$\pi:G\times^{P(\delta)}\,C_{\delta}(w,\hat{w})_{+}\to
G/P(\delta)$ induces an isomorphism between the fiber
$\eta^{-1}((gB,\hat{g}\hat{B}))$ and the locally closed subscheme
$(gC_w^{P(\delta)})\cap (\hat{g}\hat{C}_{\hat{w}}^{\hat{P}(\delta)})$ of
$\hat{G}/\hat{P}(\delta)$, for any $g\in G$ and $\hat{g}\in \hat{G}$.

{\em Proof of $(a) \Longrightarrow (b)$:} Since $(C_{\delta}(w,\hat{w}),\delta)$
is a well-covering
pair, there exist $l\in L(\delta)$ and
$\hat{l}\in \hat{L}(\delta)$ such that $\eta^{-1}_{o}(lw^{-1}B,
\hat{l}\hat{w}^{-1}\hat{B})$ is a reduced one point. Thus,
$\eta^{-1}(lw^{-1}B, \hat{l}\hat{w}^{-1}\hat{B})\simeq
(l\Phi^{P(\delta)}_{w})\cap
(\hat{l}\hat{\Phi}^{\hat{P}(\delta)}_{\hat{w}})$ is a reduced single
point in a neighborhood of $1\cdot \hat{P}(\delta)$, showing that $(w,\hat{w})$ is
$L$-movable.

Take any (general) $y_{o}=(g_{o}B,\hat{g}_{o}\hat{B})\in \Imo (\eta_{o})$ so
that $\eta^{-1}_{o}(y_{o})=\eta^{-1}(y_{o})$ and the intersection
$(gC^{P(\delta)}_{w})\cap (\hat{g}\hat{C}^{\hat{P}(\delta)}_{\hat{w}})$
is proper inside $\hat{G}/\hat{P}(\delta)$ and dense in
$(g X^{P(\delta)}_{w})\cap (\hat{g}\hat{X}^{\hat{P}(\delta)}_{\hat{w}})$. Such
a $y_{o}$ exists since $\Imo(\eta_{o})$ is open in $X$. Now,
$\eta^{-1}_{o}(y_{o})=\eta^{-1}(y_{o})$ is a single reduced point by
the assumption. Thus,
$(gC^{P(\delta)}_{w})\cap (\hat{g}\hat{C}^{\hat{P}(\delta)}_{\hat{w}})$
is a single reduced point, showing that
$$
[X^{P(\delta)}_{w}]\cdot \iota^{*}[\hat{X}^{\hat{P}(\delta)}_{\hat{w}}] =[pt].
$$

{\em Proof of $(b) \Longrightarrow (a)$:}
 Take a $G$-stable open subset
$V\subset G/B\times \hat{G}/\hat{B}$ so that for any
$(gB,\hat{g}\hat{B})\in V$, the intersection $gC^{P(\delta)}_{w}\cap
(\hat{g}\hat{C}_{\hat{w}}^{\hat{P}(\delta)})$ is transverse  inside
$\G/\hat{P}(\delta)$ and dense in
$(gX^{P(\delta)}_{w})\cap (\hat{g}\hat{X}_{\hat{w}}^{\hat{P}(\delta)})$. Since
$[X^{P(\delta)}_{w}]\cdot
\iota^{*}[\hat{X}^{\hat{P}(\delta)}_{\hat{w}}]=[pt]$ by assumption, for
any $(gB,\hat{g}\hat{B})\in V$, the scheme $(gC^{P(\delta)}_{w})\cap
(\hat{g}\hat{C}_{\hat{w}}^{\hat{P}(\delta)})$ is a reduced single
point. Thus, $\eta_{|\eta^{-1}(V)}:\eta^{-1}(V)\to V$ is an
isomorphism; in particular,
$\eta:G\times^{P(\delta)}\,C_{\delta}(w,\hat{w})_{+}\to X$ is a
birational map. Let $V'$ be the open subset
$$V':=\{y\in G\times^{P(\delta)}\,C_{\delta}(w,\hat{w})_{+}:(d\eta)_y\,\,
\text{is an isomorphism}\}.$$
Then, $V'$ is clearly $G$-stable and hence can be written as
$G\times^{P(\delta)}\,C^o_+$, for a $P(\delta)$-stable open subset $C^o_+$
of $C_{\delta}(w,\hat{w})_{+}$. Since $\eta_{|v'}$ is a smooth birational
morphism, it is an isomorphism onto an open subset of $X$ (cf. [Sh, Corollary 1,
$\S$4.3, Chap. II]). Since $(w,\hat{w})$ is $L$-movable, the point $1\cdot \hat{P}(\delta)$
is a reduced
  isolated point of the scheme $(l\Phi^{P(\delta)}_{w})\cap
  (\hat{l}\hat{\Phi}^{\hat{P}(\delta)}_{\hat{w}})$ for some $l\in
  L(\delta)$ and $\hat{l}\in \hat{L}(\delta)$. Hence $[1, (lw^{-1}B,
  \hat{l}\hat{w}^{-1}\hat{B})]\in V'$.
 Thus,
  $(C_{\delta}(w,\hat{w}),\delta)$ is a well-covering pair.
\end{proof}

\begin{definition} \label{ressspecial} We will call a  nonzero $G$-dominant indivisible OPS $\delta \in O(H)$
 {\it special for the pair $(G,\G)$} if
$\bc \dot{\delta}=\cap\,\Ker \beta$, where the intersection runs over the set of $\fh$-weights of $\hat{\mathfrak{l}}(\delta)/
{\mathfrak{l}}(\delta)$, where ${\mathfrak{l}}(\delta)$ (resp. $\hat{\mathfrak{l}}(\delta)$) denotes the Lie algebra of
$L(\delta)$ (resp. $\hat{L}(\delta)$).

We  denote the set of all  special OPS for the pair $(G,\G)$ by $\mathfrak{S}=\mathfrak{S} (G,\G)$. Clearly, it is a
finite set. Let us enumerate
$$\mathfrak{S} (G,\G)=\{\delta_1, \dots, \delta_q\}.$$
\end{definition}

\begin{theorem}\label{ress}  With the notation as above, assume that no nonzero
ideal of $\frg$ is an ideal of  $\hat{\frg}$.
Let $(\lambda, \hat{\lambda})\in \Lambda_+\times \hat{\Lambda}_+$. Then, the
 following three conditions  are equivalent.

(a) $(\lambda, \hat{\lambda})\in \Gamma (G, \G)$.
\vskip1ex

(b) For any $G$-dominant $\delta\in O(H)$, and any $(w, \hat{w})\in W^{P(\delta)}
\times
{\hat{W}}^{\hat{P}(\delta)}$ such that
$[X_w^{P(\delta)}]\cdot \iota^*([\hat{X}_{\hat{w}}^{\hat{P}(\delta)}])
\neq 0$  in
 $H^*(G/ P(\delta), \bz)$, where $\hat{X}_{\hat{w}}^{\hat{P}(\delta)}:=
 \overline{\hat{B}\hat{w}\hat{P}(\delta)/\hat{P}(\delta)}\subset
 \hat{G}/\hat{P}(\delta)$  (even though $\hat{P}(\delta)$ may not be a standard
 parabolic subgroup) and $\iota: G/ P(\delta) \to \G/\hat{P}(\delta)$
 is the canonical embedding, we have
 \begin{equation} \label{eqn102} I^{\delta}_{(w, \hat{w})}:\,\,\,\,\,\,
\lambda(w\dot{\delta})+ \hat{\lambda}(\hat{w}\dot{\delta})\leq 0.
\end{equation}

\vskip1ex

(c) For any OPS $\delta_i\in \mathfrak{S} (G,\G)$ and any $(w, \hat{w})\in W^{P(\delta_i)} \times
{\hat{W}}^{\hat{P}(\delta_i)}$ such that

($c_1$) \,$[X_w^{P(\delta_i)}]\cdot \iota^*([\hat{X}_{\hat{w}}^{\hat{P}(\delta_i)}])
 =  [X_e^{P(\delta_i)}]\in
 H^*(G/ P(\delta_i), \bz)$, and

($c_2$) \, $ \gamma^{P(\delta_i)}_e (\dot{\delta_i})-
\gamma^{P(\delta_i)}_w (\dot{\delta_i})=
\hat{\gamma}^{\hat{P}(\delta_i)}_{\hat{w}}(\dot{\delta}_i),$

the inequality $I^{\delta_i}_{(w, \hat{w})}$ (as in \eqref{eqn102}) is satisfied.
\end{theorem}

\begin{proof} For a dominant pair $(\lambda, \hat{\lambda})\in \Lambda_+\times
\hat{\Lambda}_+$, we have the line bundle
$$\cl(\lambda\boxtimes \hat{\lambda}):=\cl(\lambda)\boxtimes \cl(\hat{\lambda})$$
on $X=G/B\times \G/\B$. Let $P(\lambda), \hat{P}(\hat{\lambda})$ be the unique
standard parabolic subgroups such that the line bundle
$\cl(\lambda\boxtimes \hat{\lambda})$ descends as an ample line bundle
$\mathbb{L}(\lambda\boxtimes \hat{\lambda})$ on $X(\lambda, \hat{\lambda}):=
G/P(\lambda)\times \G/\hat{P}(\hat{\lambda})$. As earlier, we call a point
$(gB,\g\B)\in X$ semistable with respect to the line bundle
$\cl(\lambda\boxtimes \hat{\lambda})$ if $\pi(gB,\g\B)$ is $G$-semistable
with respect to the ample line bundle
$\mathbb{L}(\lambda\boxtimes \hat{\lambda})$, where $\pi:X \to
X(\lambda, \hat{\lambda})$ is the canonical projection.

\medskip
\noindent
{\em Proof of $(a) \Rightarrow (b)$.} We abbreviate $P(\delta)$ (resp. $\hat{P}(\delta)$)
 by $P$ (resp. $\hat{P}$). Pick any (general) $(g,\g)\in G\times \G$
satisfying the following:

\begin{align} \label{eqn2001} &\g\hat{C}^{\hat{P}}_{\hat{w}}\cap g C_w^P \,\, \text{and}\,\,
 \g\hat{X}^{\hat{P}}_{\hat{w}}\cap g X_w^P\,\,\,\text{are proper intersections  in}
 \,\, \G/\hat{P}\\
 \,\,\,& \text{with}\,\,  \g\hat{C}^{\hat{P}}_{\hat{w}}\cap g C_w^P\,\,\,
 \text{dense in}\,\,\g\hat{X}^{\hat{P}}_{\hat{w}}\cap g X_w^P,\notag
 \end{align}
and
\begin{equation} (gB,\g\B)\,\,\, \text{is a $G$-semistable point of}\,\, X\,\,\,
\text{with respect to} \,\,\cl(\lambda\boxtimes \hat{\lambda}).
\end{equation}
Then, by the assumption on the cohomology product as in (b), we get $h\in G$
such that $h^{-1}P\in \g\hat{C}^{\hat{P}}_{\hat{w}}\cap g C_w^P$. Pick $\hat{v}\in
\hat{W}$ such that $\hat{v}\dot{\delta}\in \hat{\frh}_+$. Then,
\begin{align} \mu^{\cl(\lambda\boxtimes \hat{\lambda})}\bigl((hgB,h\hat{g}\B),
\delta\bigr)&=\mu^{\cl(\lambda)}\bigl(hgB,
\delta\bigr)+\mu^{\cl(\hat{\lambda})}\bigl(h\hat{g}\B,
\delta\bigr)\notag\\
&=\mu^{\cl(\lambda)}\bigl(hgB,
\delta\bigr)+\mu^{\cl(\hat{\lambda})}\bigl(\hat{v}h\hat{g}\B,
\hat{v}\cdot{\delta}\bigr)\notag\\
&=-\lambda(w\dot{\delta})-\hat{\lambda}(\hat{w}\hat{v}^{-1}\hat{v}\dot{\delta}),
\,\,\,\text{by Lemma \ref{l1}}\notag\\
&\geq 0, \,\,\,\text{by Proposition \ref{propn14}}.
\end{align} This proves (b).

\medskip
\noindent
{\em Proof of $(b) \Rightarrow (a)$.} Pick any (general) $(g,\g)\in G\times \G$
satisfying the equation \eqref{eqn2001} for any $G$-dominant weight
$\delta_o\in O(H)$ and any pairs $(w,\hat{w})\in W^{P(\delta_o)}\times
\W^{\hat{P}(\delta_o)}$. This is possible since there are only finitely many
$P(\delta_o)$ and $\hat{P}(\delta_o)$ as we run through $G$-dominant weights
$\delta_o\in O(H)$. Fix any  $(g,\g)\in G\times \G$ as above and  consider the point
$x=(gB,\g\B)\in X$. If (a) of Theorem \ref{ress} were false, then no point of
$X$ would be semistable for the line bundle $\cl(\lambda\boxtimes \hat{\lambda})$.
Thus, by Proposition \ref{propn14}, there exists an OPS $\delta \in O(G)$
(depending upon $x$) such that
\begin{equation} \label{eqn2002} \mu^{\cl(\lambda\boxtimes \hat{\lambda})}
\bigl((gB,\hat{g}\B), \delta\bigr)<0.
\end{equation}
Let $\delta = h^{-1}\delta_o h$, for $h\in G$ so that
$\delta_o$ belongs to $O(H)$ and it is $G$-dominant. Pick $w\in W, \hat{w}\in \W$
such that $(hg)^{-1}\in C_w^{P(\delta_o)}$ and $(h\g)^{-1}\in
\hat{C}_{\hat{w}}^{\hat{P}(\delta_o)}$. Thus, by Lemma \ref{l1},
\begin{align} \mu^{\cl(\lambda\boxtimes \hat{\lambda})}\bigl((gB,\hat{g}\B),
\delta\bigr)&=\mu^{\cl(\lambda)}\bigl(hgB,
\delta_o\bigr)+\mu^{\cl(\hat{\lambda})}\bigl(h\hat{g}\B,
\delta_o\bigr)\notag\\
&=-\lambda(w\dot{\delta}_o)-\hat{\lambda}(\hat{w}\dot{\delta}_o)\notag\\
&<0, \,\,\,\text{by the inequality \eqref{eqn2002}}.
\end{align}
 Now,  $[X_w^{P(\delta_0)}]\cdot \iota^*([\hat{X}_{\hat{w}}^{\hat{P}(\delta_o)}])
\neq 0$, because of the choice of $(w,\hat{w})$ and a general point $(g,\hat{g})$
satisfying the condition \eqref{eqn2001}. This contradicts (b) and hence proves (a).

We now come to the proof of the equivalence of
(a) and (c).  Since (a) $\Rightarrow$ (b) and clearly (b) $\Rightarrow$ (c),
we get (a) $\Rightarrow$ (c).

\medskip
\noindent
{\it Proof of (c) $\Rightarrow$ (a):}
We first show that for $(\lambda,\hat{\lambda})\in \Lambda_{++}\times
\hat{\Lambda}_{++}$, if $(\lambda,\hat{\lambda})\not\in
\Gamma(G,\hat{G})$, then there exists a well-covering pair
$(C_{\delta}(w,\hat{w}),\delta)$ (defined by \eqref{resse2}), for some
$G$-dominant $\delta\in O(H)$ and $w\in W^{P(\delta)}$, $\hat{w}\in
\hat{W}^{\hat{P}(\delta)}$, such that the inequality
$I^{\delta}_{(w,\hat{w})}$ is violated, i.e.,
\begin{equation}\label{eqn2003}
\lambda(w\dot{\delta})+\hat{\lambda}(\hat{w}\dot{\delta})>0.
\end{equation}

Since $(\lambda,\hat{\lambda})\not\in \Gamma(G,\hat{G})$ (by
assumption), the set of $G$-semistable points for the ample line
bundle $\mathcal{L}(\lambda\boxtimes \hat{\lambda})$
on $X$ is empty. Thus, by Proposition \ref{birula},
there exists a class $\langle \hat{\delta}=[\delta,a]\rangle \in
\langle M(G)\rangle$ with dominant $\delta$ and a number $d>0$ such
that $X_{d,\,\langle\hat{\delta}\rangle}$ is a $G$-stable nonempty open
subset of $X$; in particular, it is smooth. Hence, by Proposition \ref{ressprop1}
(c),
$X_{d,\hat{\delta}}$ is irreducible and hence so is
$Z_{d,\,\hat{\delta}}$ (because of the surjective morphism
$p_{\hat{\delta}}:X_{d,\hat{\delta}}\to Z_{d,\hat{\delta}}$). Moreover,
by Proposition \ref{ressprop1}, $Z_{d,\hat{\delta}}$ being an open subset of
$X^{\delta}$, $\overline{Z}_{d,\,\hat{\delta}}$ is an irreducible
component of $X^{\delta}$. Hence, by the identity \eqref{resse1}, there exists
$(w,\hat{w})\in W^{P(\delta)}\times \hat{W}^{\hat{P}(\delta)}$ such
that
$$
\overline{Z}_{d,\hat{\delta}}=C_{\delta}(w,\hat{w}).
$$

Since the map $C_{\delta}(w,\hat{w})_{+}\to C_{\delta}(w,\hat{w})$,
$y\mapsto \lim\limits_{t\to 0}\delta(t)\cdot y$, is a morphism
(cf. the discussion before Definition \ref{wellcovering}), $X_{d,\hat{\delta}}$ is
 an open
(and $P(\delta)$-stable) subset of $C_{\delta}(w,\hat{w})_{+}$.

By Proposition \ref{ressprop1} (c),
$$
G\times^{P(\delta)}\, X_{d,\hat{\delta}}\to X_{d,\,\langle\hat{\delta}\rangle}
$$
is an isomorphism. This shows that
$(C_{\delta}(w,\hat{w}),\delta)$ is a well-covering pair. By definition,
 for any $x\in X_{d,\hat{\delta}}$,
$$
\mu^{\mathcal{L}(\lambda\boxtimes
  \hat{\lambda})}(x,\hat{\delta})\leq -1.
$$

Thus, by Lemma~\ref{l1},
$$
-\lambda(w\dot{\delta})-\hat{\lambda}(\hat{w}\dot{\delta})\leq -a
$$
This proves the assertion \eqref{eqn2003}.

Since $\Gamma(G,\hat{G})_{\mathbb{R}}\subset
\Lambda_{+}(\mathbb{R})\times \hat{\Lambda}_{+}(\mathbb{R})$ is a convex
cone with nonempty interior (by Lemma~\ref{lemma7.1ress}), we get that
$\Gamma(G,\hat{G})_{\mathbb{R}}$ is the cone inside
$\Lambda_{+}(\mathbb{R})\times \hat{\Lambda}_{+}(\mathbb{R})$ determined
by the inequalities
$$
\lambda(w\dot{\delta})+\hat{\lambda}(\hat{w}\dot{\delta})\leq 0
$$
running over all the well-covering pairs
$(C_{\delta}(w,\hat{w}),\delta)$ with $G$-dominant indecomposable
$\delta\in O(H)$. We finally show that, for any well-covering pair
$(C_{\delta}(w,\hat{w}),\delta)$ with $G$-dominant indecomposable
$\delta\in O(H)$, if the hyperplane $F$:
$$
\lambda(w\dot{\delta})+\hat{\lambda}(\hat{w}\dot{\delta})=0
$$
is a (codimension one) facet of $\Gamma(G,\hat{G})_{\mathbb{R}}$ intersecting
$\Lambda_{++}(\br)\times \hat{\Lambda}_{++}(\br)$, then
$\delta$ is special.

Let $F_{+}:=F\cap (\Lambda_{++}(\br)\times \hat{\Lambda}_{++}(\br))$. For any
$(\lambda,\hat{\lambda})\in \Lambda_{++}\times \hat{\Lambda}_{++}$, let
$\mathcal{C}(\lambda,\hat{\lambda})$ denote the GIT
class of $(\lambda,\hat{\lambda})$ consisting of those
$(\mu,\hat{\mu})\in \Lambda_{++}\times \hat{\Lambda}_{++}$ such that the
set of $G$-semistable points
$X^{s}(\mathcal{L}(\lambda\boxtimes
\hat{\lambda}))=X^{s}(\mathcal{L}(\mu\boxtimes
\hat{\mu}))$. By [DH] (or [Ro]), $\Lambda_{++}\times
\hat{\Lambda}_{++}$ breaks up into finitely many GIT classes, such that the cone
generated by them are
all locally closed convex cones. Now, for any
$(\lambda,\hat{\lambda})\in (\Lambda_{++}\times \hat{\Lambda}_{++})\cap
\Gamma(G,\hat{G})$ and any well-covering pair $(C,\delta)$,
\begin{equation}\label{eqn2004}
\mu^{\mathcal{L}(\lambda\boxtimes
  \hat{\lambda})}(C,\delta)=0\Leftrightarrow
X^{s}(\mathcal{L}(\lambda\boxtimes \hat{\lambda}))\cap
C\neq  \emptyset.
\end{equation}

If $X^{s}(\mathcal{L}(\lambda\boxtimes
\hat{\lambda}))\cap C\neq  \emptyset$, by Proposition~\ref{propn14} (e),
$\mu^{\mathcal{L}(\lambda\boxtimes
  \hat{\lambda})}(C,\delta)=0$, since $C\subset X^\delta$. Conversely, if
$\mu^{\mathcal{L}(\lambda\boxtimes
  \hat{\lambda})}(C,\delta)=0$, take $x\in
X^{s}(\mathcal{L}(\lambda\boxtimes \hat{\lambda}))\cap
C_{+}$ (which is possible since $\Iim \eta$ contains an open
subset). By Proposition~\ref{propn14} (c), $\lim\limits_{t\to 0}
\delta(t)\cdot x\in
X^{s}(\mathcal{L}(\lambda\boxtimes \hat{\lambda}))$.
This proves \eqref{eqn2004}.

From \eqref{eqn2004}, we see that $F_{+}\cap \Lambda_{++}\times
\hat{\Lambda}_{++}$ is a (finite) union of GIT
classes. In particular, it contains a GIT class
$\mathcal{C}(\lambda_{o},{\hat{\lambda}}_{o})$ such that the cone
generated by it has
nonempty interior in $F_{+}$. Take $x_{o}\in
X^{s}\cap C$ such that its $G$-orbit is closed in
$X^{s}$, where we have abbreviated $X^s:=X^{s}(\mathcal{L}(\lambda_{o}\boxtimes
\hat{\lambda}_o))$
and $C:=C_{\delta}(w,\hat{w})$.  By the following argument, such a $x_o$ exists:

Take a $P(\delta )$-orbit $\mathfrak{O}:=P(\delta)\cdot x$ in  $C_+\cap X^s$ of
the smallest dimension. Then,  $\mathfrak{O}$ is a closed subset of $X^s$; for if
it is not closed in $X^s$,  then its closure  $\bar{\mathfrak{O}}$ in $X^s$ would
contain a $P(\delta )$-orbit  $\mathfrak{O}'$ of strictly smaller dimension.
Of course,  $\mathfrak{O}'\subset \bar{C}_+$, where $\bar{C}_+$ is the closure of
$C_+$ in $X$. Further, $\partial C_+:=\bar{C}_+\setminus C_+\subset X\setminus X^s.$
To see this, take a $G$-equivariant embedding $\theta: X\hookrightarrow \mathbb{P}(V)$
for a $G$-module $V$  such that $\mathcal{L}(\lambda_{o}\boxtimes
\hat{\lambda}_o)^N$ is $G$-equivariantly isomorphic with $\theta^*(\mathcal{O}(1))$
for some $N>0$. Decompose $V=V_{-}\oplus V_0\oplus V_+$ under the action of
$\delta(t)$, where $V_0$ is the invariant subspace and $V_+$ (resp. $V_{-}$) is the
sum of the eigenspaces of positive (resp. negative) eigenvalues. Then, it is
easy to see that $C\subset \mathbb{P}(V_0),\, C_+\subset \mathbb{P}(V_0\oplus V_+)$
and $\partial C_+\subset  \mathbb{P}(V_+).$ Thus, $\partial C_+\subset X\setminus X^s$.  Hence,  $\mathfrak{O}' \subset C_+$.
A contradiction, proving that $\mathfrak{O}$ is closed in  $C_+\cap X^s$. By Lemma
 \ref{l1}, $\mu^{\mathcal{L}(\lambda_{o}\boxtimes
\hat{\lambda}_o)}(C_+, \delta)=0.$ Hence, for any $x\in \mathfrak{O}$, by Proposition
\ref{propn14} (c), $x_o:=\lim\limits_{t\to 0}
\delta(t)\cdot x \in X^s$. Thus, $x_o\in \mathfrak{O}$. Hence, $G\cdot x_o=G\cdot
\mathfrak{O}$ is closed in $X^s$, since $G/P(\delta)$ is a projective variety.

 Since $G\cdot x_{o}$ is contained in an
affine open subset of
$X^{s}(\mathcal{L}(\lambda_{o}\boxtimes
\hat{\lambda}_{o}))$, by Matsushima's theorem, the isotropy
$G_{x_{o}}$ is a reductive group contained of course in a Borel
subgroup of $G$. Thus, $\Iim \delta\subset G_{x_{o}}\subset H'$, for some
maximal torus $H'$ of $G$.

But, since $x_{o}\in
X^{s}(\mathcal{L}(\lambda_{o}\boxtimes \lambda_{o}))$,
it is easy to see that
$\mathcal{L}(N\lambda_{o}\boxtimes
N\hat{\lambda}_{o})_{|G\cdot x_{o}}$ is $G$-equivariantly trivial for
some $N>0$. Thus,
$\mathcal{C}(\lambda_{o},\hat{\lambda}_{o})$ and hence
$F_{+}$ is contained in the kernel of the following map:
$$
\gamma:(\Lambda\times \hat{\Lambda})\otimes_{\mathbb{Z}}\mathbb{R}\to
\Pic^{G}(G\cdot x_{o})\otimes_{\mathbb{Z}}\mathbb{R}\simeq \Lambda(G^{o}_{x_{o}})
\otimes_{\mathbb{Z}}\mathbb{R},
$$
where $\Lambda(G^{o}_{x_{o}})$ is the character group of the identity
component $G^{o}_{x_{o}}$ of $G_{x_{o}}$. But, since $\gamma$ is clearly surjective and
$F_{+}$ lies in the kernel of $\gamma$,
$\Lambda(G^{o}_{x_{o}})\bigotimes\limits_{\mathbb{Z}}\mathbb{R}$ is at
most one dimensional. Further, since $\Iim \delta \subset G^{o}_{x_{o}}$,
we see that $G^{o}_{x_{o}}$ is exactly one dimensional and $\Iim
\delta=G^{o}_{x_{o}}$. Thus, the general isotropy of the action of
$L(\delta)/\Iim \delta$ on $C_{\delta}(w,\hat{w})$ is finite. As an
$L(\delta)$-variety, $C_{\delta}(w,\hat{w})$ is isomorphic with the
full flag variety $(L(\delta)/B(\delta))\times
(\hat{L}(\delta)/\hat{B}(\delta))$ of $L(\delta)\times \hat{L}(\delta)$,
where $B(\delta)$ (resp. $\hat{B}(\delta)$) is a Borel subgroup of
$L(\delta)$ (resp. $\hat{L}(\delta)$) containing $H$ (resp. $\hat{H}$). Since $L(\delta)$ acts
transitively on $L(\delta)/B(\delta)$ and centralizes $\delta$, for a
general point $y\hat{B}(\delta)$, the identity component of the
isotropy of the point
$$
(B(\delta), y\hat{B}(\delta))\in (L(\delta)/B(\delta))\times
(\hat{L}(\delta)/\hat{B}(\delta))
$$
under the action of $L(\delta)$ coincides with $\Iim \delta$.  Let
$\hat{U}(\delta)$ be the unipotent radical of $\hat{B}(\delta)$ and let
${\hat{w}}_o^\delta$ be the longest element of the Weyl group of
$\hat{L}(\delta)$. Then, we have the open cell $\hat{U}(\delta) \simeq
\hat{U}(\delta)\cdot{\hat{w}}_o^\delta \hat{B}(\delta)/\hat{B}(\delta)$
in $\hat{L}(\delta)/\hat{B}(\delta).$ Replacing the point
$(B(\delta), y\hat{B}(\delta))$ by $(lB(\delta), ly\hat{B}(\delta))$, for some
$l \in L(\delta)$, we can assume that $B(\delta)=\hat{B}(\delta)\cap L(\delta)$.
Under the action of $B(\delta)$ on $\hat{L}(\delta)/\hat{B}(\delta)$, the open cell
$\hat{U}(\delta)$ is stable and the action is given by
$$ (t\cdot u)\hat{u}=tu\hat{u}t^{-1}, \,\,\,\text{for}\,\, t\in H,
u\in U(\delta),
\hat{u}\in \hat{U}(\delta),$$
where $U(\delta)$ is the unipotent radical of $B(\delta)$. Since the isotropy of
$(B(\delta), y\hat{B}(\delta))$ under the action of $L(\delta)$ coincides with
the isotropy of $y\hat{B}(\delta)$ under the action of $B(\delta)$, for a general
point $\hat{u}\in \hat{U}(\delta),$ the connected component of the isotropy of
$U(\delta)\cdot \hat{u}\in U(\delta)\backslash \hat{U}(\delta)$ under the action of
$H$ (via the conjugation action) coincides with $\Iim \delta$.

But  $U(\delta)\backslash \hat{U}(\delta)\simeq \hat{\fu}(\delta)/ {\fu}
(\delta)$ as $H$-varieties. Thus, we get $\bc \dot{\delta}=\cap \Ker \beta$, where
the intersection runs over all the $H$-weights $\beta$ of
$\hat{\fu}(\delta)/ \fu (\delta)$. Thus, $\delta$ is special.

This proves that any facet of $\Gamma(G,\hat{G})_{\mathbb{R}}$ which
intersects $\Lambda_{++}(\mathbb{R})\times \hat{\Lambda}_{++}(\mathbb{R})$
is given by
$$
\lambda(w\dot{\delta})+\hat{\lambda}(\hat{w}\dot{\delta})=0,
$$
where $(C_{\delta}(w,\hat{w}),\delta)$ is a  well-covering pair
with $\delta\in O(H)$ special.

Thus, the theorem follows from Proposition~\ref{prop7.3ress} and
Lemma~\ref{lemma7.5ress}.
\end{proof}

\begin{remark}   (a) Berenstein-Sjamaar [BS] proved a weaker version of Theorem
\ref{ress}, where they have (in general) many more inequalities. Their set of
inequalities consists of $ I^{\delta}_{(w, \hat{w})}$, where $\delta$ runs over
(in general) a larger set of OPS in $H$ than $\mathfrak{S} (G,\G)$ and for any $\delta$ in their set, they considered the inequalities  $ I^{\delta}_{(w, \hat{w})}$ for any pair  $(w, \hat{w})\in W^{P(\delta_i)} \times
{\hat{W}}^{\hat{P}(\delta_i)}$ satisfying only
 $\iota^*([\hat{X}_{\hat{w}}^{\hat{P}(\delta_i)}])\cdot [X_w^{P(\delta_i)}]=d
 [X_e^{P(\delta_i)}]$, for some nonzero $d$.

 (b) The equivalence of (a) and (c) in Theorem \ref{ress} can also be
obtained by a proof quite similar to the proof of Theorem \ref{EVT}.
\end{remark}
\begin{lemma}\label{ress1}
If we specialize Theorem \ref{ress} to the case when $G$ is a connected semisimple group,
$\G=G^{s-1}$ and $G$ is embedded in $G^{s-1}$ diagonally, then we
recover Theorem \ref{EVT}.
\end{lemma}
\begin{proof} Since $\frg$ is semisimple, no nonzero ideal of $\frg$ is an ideal
of $\hat{\frg}:=\frg^{s-1}$. Further, the set of nonzero $H$-weights of
$\hat{\frg}/\frg$ is precisely equal to the set $R$ of the roots of $\frg$.
Now, for a root $\beta$ and a (dominant) element $x=\sum_{p=1}^kr_{i_p}x_{i_p}\in
\frh_+$ with each $r_{i_p}>0$ and $i_1, \dots, i_k$ distinct, $\beta (x)=0$
if and only if $\beta \in \sum_{j \notin \{i_1, \dots, i_k\}} \,\bz\alpha_j.$
Thus, if $k\geq 2$, then there is no OPS $\delta \in \mathfrak{S} (G,\G)$
such that $\dot{\delta}=x$. From this we see that $\mathfrak{S} (G,\G)
=\{\delta_i(z):=z^{d_ix_i}\}_{1\leq i \leq \ell}$, for some unique positive
rational numbers $d_i$. Clearly, $\{P(\delta_i)\}_{1\leq i \leq \ell}$ bijectively
parameterizes the set of the standard maximal parabolic subgroups of $G$. By using the identity
\eqref{eqn5}, it is easy to see that for $w=w_1, \hat{w}=(w_2, \dots, w_s)$, the identity
(b$_2$) of Theorem \ref{ress} is equivalent to the identity
$$\bigl((\sum_{j=1}^s\,\chi_{w_j})-\chi_1\bigr)(x_{i})=0.$$ Thus, by Proposition
\ref{T1},
the two conditions (b$_1$) and (b$_2$) of Theorem \ref{ress} are equivalent to
the condition (b) of Theorem \ref{EVT}. Hence, Theorem
\ref{ress}, for the case of the diagonal embedding $G \to G^{s-1}$, is equivalent
to Theorem \ref{EVT}.
\end{proof}

The following theorem (again due to Ressayre [R$_1$]) shows that the set of
inequalities given by the (c) part of Theorem \ref{ress} is an irredundant system.
As earlier, let $ \Gamma (G, \G)_\br$ be the cone generated by $ \Gamma (G, \G)$ inside
the vector space $\Lambda(\br)\times \hat{\Lambda}(\br)$, where  $\Lambda(\br):= \Lambda\otimes_\bz \br$.
\begin{theorem} \label{ress2} Following the assumptions of Theorem \ref{ress},
the set of inequalities provided by the (c)-part of
Theorem \ref{ress}
is an irredundant system of inequalities describing the cone  $ \Gamma (G, \G)_\br$
inside  $\Lambda_+(\br)\times{ \hat{\Lambda}}_+(\br)$, i.e.,
the hyperplanes given by the equality in $I^{\delta_i}_{(w, \hat{w})}$ are precisely those
 facets of the cone  $ \Gamma (G, \G)_\br$ which intersect the interior of
 $\Lambda_+(\br)\times{ \hat{\Lambda}}_+(\br)$,
where  $\Lambda_+(\br)$ denotes the cone inside  $\Lambda (\br)$ generated by
 $\Lambda_+$ .
\end{theorem}
\begin{proof} First of all, the inequalities $I^{\delta_i}_{(w, \hat{w})}$ (as in
\eqref{eqn102}) for $\delta_i$ and $(w, \hat{w})$ as in the (c)-part of Theorem
\ref{ress} are pairwise distinct, even up to scalar multiples:

The stabilizer of $\dot{\delta}_i$ under the action of $W$ (resp. $\hat{W}$) is
precisely equal to the subgroup $W_{P(\delta_i)}$ (resp.
$\hat{W}_{\hat{P}(\delta_i)}$). Let the pair $(w\dot{\delta}_i,
\hat{w}\dot{\delta}_i)= z(v\dot{\delta}_i,
\hat{v}\dot{\delta}_i)$, for some $1\leq i\leq q, z\in \br$ and
$(w, \hat{w}) \neq (v, \hat{v})\in W^{P(\delta_i)}\times
\hat{W}^{\hat{P}(\delta_i)}$ as in the (c)-part of Theorem \ref{ress}. Then, it
is easy to see that $z=\pm 1$. Moreover, $z\neq 1$ because of the stabilizer
assertion as above. Further, $z\neq -1$, for otherwise $\Gamma(G,\hat{G})$ would
 satisfy two inequalities with opposite signs contradicting Lemma
 \ref{lemma7.1ress}.

 Now, $(w\dot{\delta}_i, \hat{w}\dot{\delta}_i)$ can not be equal to
 $z(v\dot{\delta}_j,
\hat{v}\dot{\delta}_j)$, for any $1\leq i\neq j \leq q$ and $ z\in \br$: We can
not have $z> 0$ since each $\delta_i$ is indecomposable. For $z< 0$, again
$\Gamma(G,\hat{G})$ would
 satisfy two inequalities with opposite signs.

 Also, since each $\dot{\delta}_i \neq 0$, none of the hyperplanes
 $H_{(w, \hat{w})}^{\delta_i}: \lambda(w\dot{\delta}_i)+\hat{\lambda}(\hat{w}
 \dot{\delta}_i)=0$ (given by the (c)-part of Theorem \ref{ress}) is a face
 of the dominant chamber for the group $G\times \hat{G}$.

 We finally show that $H_{(w, \hat{w})}^{\delta_i}\cap \Gamma (G, \G)_\br$ is a
 (codimension one) facet of $\Gamma (G, \G)_\br$ for any $\delta_i\in
 \mathfrak{S}(G, \G)$ and any $(w, \hat{w}) \in W^{P(\delta_i)}\times
\hat{W}^{\hat{P}(\delta_i)}$ as in the  (c)-part of Theorem \ref{ress}:

In the following, we abbreviate $\delta_i$ by $\delta$. Consider
 $\Gamma (C)_\br\subset  \Lambda { (\br)}\times{ \hat{\Lambda}}{(\br)}$, where
 $C=C_\delta(w, \hat{w})$,
 $$\Gamma (C):=\{(\lambda,\hat{\lambda})\in \Lambda \times  \hat{\Lambda}:
 H^0(C, \mathcal{L}(N\lambda\boxtimes
N\hat{\lambda})_{|C})^{G^\delta}\neq 0,\,\,\,\text{for some} \,\,N>0\}$$
and  $\Gamma (C)_\br$ is the cone inside $\Lambda (\br)\times
\hat{\Lambda} (\br)$ generated by $\Gamma (C)$. We show that
\beqn\label{new5001}
\langle \Gamma (C) \rangle = \langle H_{(w, \hat{w})}^{\delta}\cap
\Gamma (G, \G)\rangle,
\eeqn
where $\langle \Gamma (C) \rangle$ (resp.  $\langle H_{(w, \hat{w})}^{\delta}\cap
\Gamma (G, \G)\rangle$) denotes the $\br$-subspace of $\Lambda (\br)
\times{ \hat{\Lambda}} (\br)$ spanned by $\Gamma (C)$ (resp.
$H_{(w, \hat{w})}^{\delta}\cap
\Gamma (G, \G)$). We first show that
\beqn\label{new5002}
H_{(w, \hat{w})}^{\delta}\cap
\Gamma (G, \G)\subset  \Gamma (C).
\eeqn
Take $(\lambda, \hat{\lambda})\in H_{(w, \hat{w})}^{\delta}\cap
\Gamma (G, \G)$. Then, by the proof of Theorem \ref{ress} (specifically the part
``Proof of $(a) \Rightarrow (b)$") there exists a $G$-semistable point
$x=(gB, \hat{g}\hat{B})\in X:=G/B\times \G/\hat{B}$ corresponding to the line
bundle $\mathcal{L}(\lambda\boxtimes
\hat{\lambda})$ such that $x\in C_+:=C_\delta(w, \hat{w})_+$ with
$$\mu^{\mathcal{L}(\lambda\boxtimes
\hat{\lambda})}(x,\delta)=-\lambda(w\dot{\delta})-\hat{\lambda}(\hat{w}
 \dot{\delta})=0.$$
 Since $x$ is a semistable point, there exists $N>0$ and a section
 $\sigma\in H^0(X, \mathcal{L}(N\lambda\boxtimes
N\hat{\lambda}))^{G}$ such that $\sigma(x)\neq 0$. Hence, by Proposition
\ref{propn14} (c), $\sigma$ does not vanish at  $\lim_{t\to 0}\delta(t)x$. Thus,
$(\lambda, \hat{\lambda})\in \Gamma (C)$.

Conversely, take finitely many $(\lambda_j, \hat{\lambda}_j)\in\Gamma (C)$
which $\br$-span $\langle \Gamma (C) \rangle$ . We can assume (replacing
$(\lambda_j, \hat{\lambda}_j)$ by a multiple $(N\lambda_j, N\hat{\lambda}_j)$) that
$\mathcal{L}_j:=\mathcal{L}(\lambda_j\boxtimes
\hat{\lambda}_j)_{|C}$ has a nonzero $L(\delta)$-invariant section $\sigma_j$. We
now show that $\sigma_j$ can be extended to a $G$-invariant rational section
$\hat{\sigma}_j $ of $\mathcal{L}(\lambda_j\boxtimes
\hat{\lambda}_j)$ on $X$:

Extend the action of $L(\delta)$ on $C$ to an action of $P(\delta)$ on $C$ by
demanding that the unipotent radical $U(\delta)$ of $P(\delta)$ acts trivially on $C$
(and hence on $\mathcal{L}_j$). (Observe that the standard action of $P(\delta)$
on $X$ does not keep $C$ stable in general, so this action is a different action of
$P(\delta)$ on $C$.) It is easy to see that the map $\pi_\delta : C_+\to C$
(defined just above Definition \ref{wellcovering}) is $P(\delta)$-equivariant.
Thus, we have a  $G$-equivariant line bundle $G\times^{P(\delta)} \,
\pi_\delta^*(\mathcal{L}_j)\to G\times^{P(\delta)} \,C_+$. Also, we have
a $G$-equivariant line bundle $\eta^*(\mathcal{L}(\lambda_j\boxtimes
\hat{\lambda}_j))$ on
$G\times^{P(\delta)} C_{+}$. These two  $G$-equivariant line bundles
on $G\times^{P(\delta)}\, C_{+}$ are $G$-equivariantly isomorphic: By
[CG, $\S\S$5.2.16 and  5.2.18], it suffices to show that their restrictions to
$C_+$ are isomorphic as $L(\delta)$-equivariant line bundles. But,
$\pi_\delta : C_+\to C$ is a $L(\delta)$-equivariant vector bundle (as proved by
Bialynicki-Birula), and thus by the Thom isomorphism  (cf. [CG, Theorem 5.4.17]),
it suffices to show that these two line bundles restricted to $C$ are
 $L(\delta)$-equivariantly isomorphic. But, the latter is true since both of the line
 bundles are the same restricted to $C$.

 The $L(\delta)$-invariant section $\sigma_j$ of $\mathcal{L}_j$ (which is
 automatically $P(\delta)$-invariant) gives rise to the $G$-invariant section
 $\bar{\sigma}_j$ defined by $[g,x]\mapsto [g,\sigma_j(\pi_\delta (x))]$, for
 $g\in G, x\in C_+$. Since $(C, \delta)$ is a well-covering pair, $\bar{\sigma}_j$
 descends to a $G$-invariant regular section on a $G$-stable open subset $X^o$
 of $X$ such that $X^o\cap C\neq \emptyset$, and thus a $G$-invariant
 rational section $\hat{\sigma}_j$ of the line bundle $\mathcal{L}(\lambda_j\boxtimes
\hat{\lambda}_j)$ on $X$. Let $\{E_p\}_p$ be the irreducible components of
$X\setminus X^o$ of codimension one. Since $G$ is connected, each $E_p$ is
$G$-stable. Consider the line bundle $\mathcal{E}:=\mathcal{O}_X(\sum_p\,a_pE_p)$,
with $a_p\geq 0$ large enough so that each of the rational sections
$\hat{\sigma}_j$ of  $\mathcal{L}(\lambda_j\boxtimes
\hat{\lambda}_j)$ are ($G$-invariant) regular sections $\hat{\lambda}_j^o$ of
the line bundle $\mathcal{L}(\lambda_j\boxtimes
\hat{\lambda}_j)\otimes \mathcal{E}$. Moreover, since no $E_p$ contains $C$,
 $(\hat{\lambda}^o_j)_{|C}\neq 0$. We can easily lift the diagonal $G$-equivariant
 structure on $\mathcal{E}$ to a $G\times \G$-equivariant
 structure by replacing (if needed) $\mathcal{E}$ by $\mathcal{E}^N$  for some
 $N>0$. Let $\mathcal{E}\simeq \mathcal{L}(\mu\boxtimes
\hat{\mu})$. Then, $(\lambda_j+\mu, \hat{\lambda}_j+\hat{\mu})\in
H_{(w, \hat{w})}^{\delta}\cap
\Gamma (G, \G)$, for all $j$. Since both of
$(\lambda_j+\mu, \hat{\lambda}_j+\hat{\mu})$ and $(\lambda_j+2\mu,
\hat{\lambda}_j+2\hat{\mu})$ are in
$H_{(w, \hat{w})}^{\delta}\cap
\Gamma (G, \G),$  we see that each $(\lambda_j, \hat{\lambda}_j)\in
\langle H_{(w, \hat{w})}^{\delta}\cap
\Gamma (G, \G)\rangle$. Thus,
\beqn\label{new5003}
\langle \Gamma (C) \rangle \subset \langle H_{(w, \hat{w})}^{\delta}\cap
\Gamma (G, \G)\rangle.
\eeqn
Combining \eqref{new5002} and \eqref{new5003}, we get \eqref{new5001}.

As a $G^\delta$-variety, $C$ is isomorphic with $G^\delta/B^\delta \times
\G^\delta/\hat{B}^\delta$. Thus, by [MR, Corollaire 1], $\langle \Gamma (C)
\rangle $ is of codimension one in $\Lambda (\br)\times{ \hat{\Lambda}} (\br)$,
since $\delta$ is special. This proves the theorem.
\end{proof}

The following result for any semisimple and connected $G$ is a particular case of
 Theorem \ref{ress2} (cf. Lemma \ref{ress1}). In the case $G=\SL(n)$, the following corollary was earlier
proved by Knutson-Tao-Woodward [KTW].
\begin{corollary} \label{6.4} The set of inequalities provided by the (b)-part of
Theorem  \ref{EVT}
is an irredundant system of inequalities describing the cone  ${\Gamma_s(G)}_\br$
 generated by $\Gamma_s(G)$
inside $\Lambda_+(\br)^s$, i.e.,
the hyperplanes given by the equality in $I^P_{(w_1, \dots, w_s)}$ are precisely those
 facets of the cone ${\Gamma_s(G)}_\br$ which intersect the interior of
 $\Lambda_+(\br)^s$.

 By Theorem \ref{sj}, the same result is true for the cone $\bar{\Gamma}_s(\frg)$, i.e., the inequalities
 given by Corollary \ref{7.5} (b)
 form an irredundant system of inequalities describing the cone  $\bar{\Gamma}_s(\frg)$
inside  $\frh_+^s$ , i.e.,
the hyperplanes given by the equality in $I^P_{(w_1, \dots, w_s)}$ are precisely those
 facets of the cone  $\bar{\Gamma}_s(\frg)$ which intersect the interior of
 $\frh_+^s$.
 \end{corollary}
\begin{remark}  (1) Fix a maximal compact subgroup $\hat{K} \subset \hat{G}$ and $K\subset G$
such that $K\subset \hat{K}$. Define
$$\bar{\Gamma}(\frg, \hat{\frg}):=\{(h,\hat{h})\in \frh_+\times \hat{\frh}_+:
K\cdot (-h)\cap \pi(\hat{K}\cdot \hat{h})\neq \emptyset\},$$
where $\pi: i\hat{\frk}\to i\frk$ is the restriction map obtained from the
identifications (induced from the Killing forms) $i\hat{\frk}\simeq i \hat{\frk}^*$
and $i\frk \simeq i\frk^*$. Then, we get exact analogues of Theorems \ref{ress}
 and \ref{ress2} for $\bar{\Gamma}(\frg, \hat{\frg})$ by using an analogue of Theorem
 \ref{sj} in this setting (just as we got Corollary \ref{7.5} from Theorem \ref{EVT}).

(2) Berenstein-Sjamaar have determined the cone $\bar{\Gamma}(\frg, \hat{\frg})$ for the pairs $(\frh, \frg)$ (for any semisimple $\frg$
and its Cartan subalgebra $\frh$); $(\mathfrak{s}, \frg)$ (for any $sl_2$-triple $\mathfrak{s}$); and
$(G_2, sl(3))$ (cf. [BS, $\S$5]).

(3) Smaller faces of the cone $\Gamma (G, \hat{G})_\br$ are determined by Ressayre
in [R$_1$] and [R$_5$] (also see [Br]).

(4) For any simple $G$ with Lie algebra different from
$\mathfrak{s} \mathfrak{l} (2)$, the cone ${\Gamma_s(G)}_\br$ inside $\Lambda (\br)^s$
has facets precisely those given by the facets of ${\Gamma_s(G)}_\br$ intersecting
 the interior of
 $\Lambda_+(\br)^s$ together with the facets of the dominant chamber
 $\Lambda_+(\br)^s$ inside $\Lambda (\br)^s$ (cf. [KTW, Theorem 4] for
 $G=\SL (n), n\geq 3$, and [MR] for an arbitrary $G$). As observed by Ressayre,
 it is easy to see that this property fails for the pair $(\GL(n), \SL(n+1))$
 embedded as a Levi subgroup.
 \end{remark}

\section{Notational generalities on classical groups}\label{section4}

For a general reference for the material in this section, see, e.g., [BL]. In its
present form it is taken from [BK$_2$].
\subsection{Special Linear Group $\SL(n+1)$.}\label{4.1}
In this case we take  $B$ to be the (standard)
Borel subgroup consisting of upper triangular matrices of determinant $1$
and $H$ to be the subgroup
 consisting of diagonal matrices (of determinant $1$). Then,
 \[\frh=\{{\bf t}=\text{diag}(t_1, \dots , t_{n+1}): \sum t_i=0\}, \]
 and
  \[\frh_+=\{{\bf t}\in \frh: t_i\in \Bbb R\,\text{and}\, t_1\geq \cdots  \geq t_{n+1}\}.\]
  For any $1\leq i\leq n$,
  \[\alpha_i({\bf t})=t_i-t_{i+1}; \alpha_i^\vee=\text{diag}(0, \dots, 0,1,-1,0, \dots,
  0); \omega_i({\bf t})=t_1+\cdots t_i,\]
  where $1$ is placed in the $i$-th place.

  The Weyl group $W$ can be identified with the symmetric group $S_{n+1}$, which acts
  via the permutation of the coordinates of ${\bf t}$. Let $\{r_1, \dots, r_n\} \subset S_{n+1}$
  be the (simple)
  reflections corresponding to the simple roots $\{\alpha_1, \dots, \alpha_n\}$
  respectively. Then,
  \[r_i=(i,i+1).\]
  For any $1\leq m\leq n$, let $P_m\supset B$ be the (standard) maximal parabolic
  subgroup of $\SL(n+1)$ such that its unique Levi subgroup $L_m$ containing $H$
  has for its simple roots $\{\alpha_1, \dots, \hat{\alpha}_m, \dots, \alpha_n\}$.
  Then, $\SL(n+1)/P_m$ can be identified with the Grassmannian $\Gr(m, n+1)
  =\Gr(m, \Bbb C^{n+1})$ of $m$-dimensional subspaces of $\Bbb C^{n+1}$. Moreover, the set
  of minimal coset representatives $W^{P_m}$ of $W/W_{P_m}$ can be identified
with the set of $m$-tuples
\[S(m,n+1)=\{A:=1\leq a_1 < \cdots < a_m \leq n+1 \}.\]
Any such $m$-tuple $A$ represents the permutation
\[v_A=(a_1,\dots, a_m,a_{m+1}, \dots, a_{n+1}),\]
where $\{a_{m+1}< \cdots < a_{n+1}\}=[n+1]\setminus \{a_1, \dots, a_m\}$
and
\[[n+1]:=\{1, \dots, n+1\}.\]

For a complete flag
 $E_{\bull}:0=E_0\subsetneq E_1\subsetneq
\cdots\subsetneq E_{n+1}= \Bbb C^{n+1}$, and $A \in S(m,n+1)$, define the
corresponding
{\it shifted Schubert cell} inside $\Gr(m,n+1)$:
\begin{align*}
\OMM_{A}(E_{\bull})=&\{M\in \Gr(m,n+1): \,\text{for any}\, 0\leq \ell
\leq m \, \text{and any}\, a_\ell\leq b < a_{\ell +1},\\
& \dim M\cap
E_{b}=\ell\},
\end{align*}
 where we set $a_0=0$ and $a_{m+1}=n+1.$ Then,
$\OMM_{A}(E_{\bull})=g(E_{\bull}) C_{v_A}^{P_m}$, where
$g(E_{\bull})$ is an element of $\SL(n+1)$ which takes the standard
flag $E_{\bull}^o$ to the flag $E_{\bull}$. (Observe that
$g(E_{\bull})$ is determined up to the right multiplication by an
element of $B$.) Its closure in $\Gr(m,n+1)$ is denoted by
$\OMC_{A}(E_{\bull})$ and its cycle class in
$H^*(\Gr(m,n+1))$ by $[{\OMC}_{A}]$. (Observe that the cohomology
class $[{\OMC}_{A}]$ does not depend upon the choice of
$E_{\bull}$.) For the standard flag $E_{\bull}=E^o_{\bull}$, we thus
have
 $\OMM_{A}(E_{\bull})=C_{v_A}^{P_m}$.

\subsection{Symplectic Group $\Sp(2n)$.}\label{notate1} Let $V=\Bbb{C}^{2n}$ be equipped with the
nondegenerate symplectic form $\langle \,,\,\rangle$ so that its matrix
$\bigl(\langle e_i,e_j\rangle\bigr)_{1\leq i,j \leq 2n}$ in the
standard basis $\{e_1,\dots, e_{2n}\}$ is given by
\begin{equation*}
E=\left(\begin{array}{cc}
0&J\\
-J&0
\end{array}\right),
\end{equation*}
where $J$ is the anti-diagonal matrix $(1,\dots,1)$ of size $n$. Let
$$\Sp(2n):=\{g\in \SL(2n):
g \,\text{leaves  the form}\, \langle \,,\,\rangle \,\text{invariant}\}$$ be the associated
symplectic group.  Clearly, $\Sp(2n)$ can be realized
as the fixed point subgroup $G^\sigma$ under the involution $\sigma:G\to G$
defined by $\sigma(A)=E(A^t)^{-1}E^{-1}$, where $G=\SL(2n)$. The involution $\sigma$
keeps both of $B$ and $H$ stable, where $B$ and $H$  are as in the $\SL(2n)$ case. Moreover,
$B^\sigma$ (respectively, $H^\sigma$) is a Borel subgroup (respectively, a maximal torus)
of $\Sp(2n)$. We denote $B^\sigma, H^\sigma$ by $B^C=B^{C_n},H^C=H^{C_n}$
respectively and (when confusion is likely) $B,H$ by $B^{A_{2n-1}}, H^{A_{2n-1}}$
respectively (for $\SL(2n)$). Then, the Lie algebra of $H^C$ (the Cartan subalgebra
$\frh^C$)
\[\frh^C=\{\text{diag}(t_1, \dots, t_n, -t_n,\dots, -t_1):t_i\in \Bbb C\}.\]
Let $\Delta^C=\{\beta_1, \dots, \beta_n\}$ be the set of simple roots. Then, for any
$1\leq i\leq n$, $\beta_i={\alpha_i}_{\vert\frh^C},$ where $\{\alpha_1, \dots, \alpha_{2n-1}\}$
are the simple roots of $\SL(2n)$. The corresponding (simple) coroots
$\{\beta_1^\vee, \dots, \beta_n^\vee\}$ are given by
\[\beta^\vee_i=\alpha_i^\vee+ \alpha_{2n-i}^\vee, \,\,\,\text{for}\, 1\leq i <n\]
and
\[\beta_n^\vee=\alpha^\vee_n.\]
Thus,
\[\frh^C_+=\{\text{diag}(t_1, \dots, t_n, -t_n,\dots, -t_1):\text{each $t_i$
is   real and}\, t_1\geq \cdots \geq t_n\geq 0\}.\]
Moreover, $\frh^{A_{2n-1}}_+$ is $\sigma$-stable and
\[ \bigl(\frh^{A_{2n-1}}_+\bigr)^\sigma = \frh^C_+.\]
Let $\{s_1, \dots, s_n\}$ be the (simple) reflections in the Weyl group
$W^C=W^{C_n}$ of $\Sp(2n)$ corresponding to the simple roots $\{\beta_1, \dots, \beta_n\}$
respectively.
Since $H^{A_{2n-1}}$ is $\sigma$-stable, there is an induced action of $\sigma$ on the Weyl group
$S_{2n}$ of $\SL(2n)$.
The Weyl group $W^C$  can be identified with the subgroup of $S_{2n}$
consisting of $\sigma$-invariants:
\[ \{(a_1,\dots, a_{2n})\in S_{2n}: a_{2n+1-i}=2n+1-a_i \, \forall 1\leq i\leq 2n\}.\]
In particular, $w=(a_1,\dots, a_{2n})\in W^C$ is determined from
$(a_1,\dots, a_{n})$.

Under the inclusion  $W^C\subset S_{2n}$, we have
 \begin{align} \label{0}
s_i&=r_ir_{2n-i}, \,\,\,\text{if}\,\,\, 1\leq i\leq n-1\notag\\
 &=r_n, \,\,\,\text{if}\,\,\, i=n.
 \end{align}
Moreover, for any $u,v\in W^C$ such that  $\ell^C(uv)= \ell^C(u)+ \ell^C(v)$,
we have
 \begin{equation} \label{-1}
 \ell^{A_{2n-1}}(uv)= \ell^{A_{2n-1}}(u)+ \ell^{A_{2n-1}}(v),
\end{equation}
where $\ell^C(w)$ denotes the length of $w$ as an element of the
Weyl group $W^C$ of $\Sp(2n)$ and similarly for $\ell^{A_{2n-1}}$.

For  $1\leq r \leq n$, we let $\IG(r,2n)=\IG(r,V)$ to be
the set of $r$-dimensional isotropic
subspaces of $V$ with respect to the form $\langle\,,\,\rangle$, i.e.,
$$\IG(r,2n):=\{M\in \Gr(r,2n): \langle v,v'\rangle=0,\ \forall\,  v,v'\in M\}.$$
 Then, it is the quotient $\Sp(2n)/P_r^C$ of $\Sp(2n)$ by
 the standard maximal parabolic subgroup
$P_r^C$  with $\Delta^C\setminus \{\beta_r\}$ as the set of simple roots of its Levi component
$L_r^C$. (Again we take $L_r^C$ to be the unique Levi subgroup of $P_r^C$
 containing $H^C$.) It can be easily seen that the set
$W^C_r$ of minimal-length coset representatives of $W^C/W_{P_r^C}$ is identified with the set
\[\FS(r,2n)=\{{I}:= 1\leq i_1< \cdots <i_r\leq 2n \,\,{\text and}\,\, I\cap \bar{I}=
\emptyset\},\]
where
\begin{equation}\label{defibari}
\bar{I}:=\{2n+1-i_1, \dots, 2n+1-i_r\}.
\end{equation}
 Any such $I$ represents the
permutation $w_I=(i_1,\dots, i_{n}) \in W^C$ by taking
$\{i_{r+1}<\cdots <i_n\}=[n]\setminus(I\sqcup \bar{I}).$

\subsection{Definition.}\label{woodyallen}
A complete flag $$E_{\bull}:0=E_0\subsetneq E_1\subsetneq
\cdots\subsetneq E_{2n}= V$$
 is called an {\rm isotropic flag} if $E_{a}^{\perp}=E_{2n-a}$, for
 $a=1,\dots,2n$. (In particular,  $E_n$ is a maximal isotropic
subspace of $V$.)

For an isotropic flag $E_{\bull}$ as above, there exists an element
$k(E_{\bull})\in \Sp(2n)$  which takes the standard flag
$E_{\bull}^o$ to the flag $E_{\bull}$. (Observe that $k(E_{\bull})$ is determined
up to the right multiplication by an element of $B^C$.)

For any $I\in \FS(r,2n)$ and any isotropic flag $E_{\bull}$, we have the
corresponding {\it shifted Schubert cell} inside $\IG(r,V)$:
$$\LAM_{I}(E_{\bull})=\{M\in \IG(r,V): \text{for any}\,
0\leq \ell \leq r \,
\text{and any}\, i_\ell\leq a < i_{\ell +1}, \dim M\cap E_{a}=\ell\},$$
where we set $i_0=0$ and $i_{r+1}=2n$. Clearly, set
theoretically,
\begin{equation}\label{2}\LAM_I(E_{\bull})=\OMM_I(E_{\bull})\cap
\IG(r,V);
\end{equation}
this is also a scheme theoretic equality (cf. [BK$_2$, Proposition 36 (4)]).
 Moreover,  $\LAM_I(E_{\bull})=k(E_{\bull})C_{w_I}^{P_r^C}$.
 Denote
the closure of $\LAM_I(E_{\bull})$ inside $\IG(r,V)$  by
${\bar{\LAM}}_I(E_{\bull})$ and its cycle class  in
$H^*(\IG(r,V))$ (which does not depend upon the choice of the
isotropic flag $E_{\bull}$) by $[\bar{\LAM}_{I}]$. For the standard
flag $E_{\bull}=E^o_{\bull}$, we have
 $\LAM_{I}(E_{\bull})=C^{P_r^C}_{w_I}$.

\subsection{Special Orthogonal Group $\SO(2n+1)$.}\label{notate2} Let $V'=\Bbb{C}^{2n+1}$ be equipped with the
nondegenerate symmetric form $\langle \,,\,\rangle$ so that its matrix $E=\
\bigl(\langle e_i,e_j\rangle\bigr)_{1\leq i,j \leq 2n+1}$ (in the standard basis
 $\{e_1,\dots, e_{2n+1}\}$) is  the $(2n+1)\times (2n+1)$
antidiagonal matrix with $1'$s all along the antidiagonal except at the $(n+1, n+1)$-th
place where the entry is $2$. Note that the associated quadratic form on $V'$ is given by
\[Q(\sum t_ie_i)= t_{n+1}^2+\sum_{i=1}^n\,t_it_{2n+2-i}.\]
 Let $$\SO(2n+1):=\{g\in \SL(2n+1):
g \,\text{leaves  the quadratic  form $Q$ invariant}\}$$ be the associated
special orthogonal group.  Clearly, $\SO(2n+1)$ can be realized
as the fixed point subgroup $G^\theta$ under the involution $\theta:G\to G$
defined by $\theta(A)=E^{-1}(A^t)^{-1}E$, where $G=\SL(2n+1)$. The involution $\theta$
keeps both of $B$ and $H$ stable. Moreover,
$B^\theta$ (respectively, $H^\theta$) is a Borel subgroup (respectively, a maximal torus)
of $\SO(2n+1)$. We denote $B^\theta, H^\theta$ by $B^B=B^{B_n},H^B=H^{B_n}$
respectively. Then, the Lie algebra of $H^B$ (the Cartan subalgebra
$\frh^B$)
\[\frh^B=\{\text{diag}(t_1, \dots, t_n, 0,-t_n,\dots, -t_1):t_i\in \Bbb C\}.\]
This allows us to identify $\frh^C$ with $\frh^B$ under the map
\[\text{diag}(t_1, \dots, t_n, -t_n,\dots, -t_1)\mapsto \text{diag}(t_1, \dots,
t_n, 0,-t_n,\dots, -t_1).\]
Let $\Delta^B=\{\delta_1, \dots, \delta_n\}$ be the set of simple roots. Then, for any
$1\leq i\leq n$, $\delta_i={\alpha_i}_{\vert\frh^B},$ where $\{\alpha_1, \dots, \alpha_{2n}\}$
are the simple roots of $\SL(2n+1)$. The corresponding (simple) coroots
$\{\delta_1^\vee, \dots, \delta_n^\vee\}$ are given by
\[\delta^\vee_i=\alpha_i^\vee+ \alpha_{2n+1-i}^\vee, \,\,\,\text{for}\, 1\leq i <n\]
and
\[\delta_n^\vee=2(\alpha^\vee_n+\alpha^\vee_{n+1}).\]
Thus, under the above identification,
\[\frh^B_+=\frh^C_+.\]
Moreover, $\frh^{A_{2n}}_+$ is $\theta$-stable and
\[ \bigl(\frh^{A_{2n}}_+\bigr)^\theta = \frh^B_+.\]
Let $\{s'_1, \dots, s'_n\}$ be the (simple) reflections in the Weyl group
$W^B=W^{B_n}$ of $\SO(2n+1)$ corresponding to the simple roots $\{\delta_1, \dots, \delta_n\}$
respectively.
Since $H^{A_{2n}}$ is $\theta$-stable, there is an induced action of $\theta$ on the Weyl group
$S_{2n+1}$ of $\SL(2n+1)$.
The Weyl group $W^B$  can be identified with the subgroup of $S_{2n+1}$
consisting of $\theta$-invariants:
\[ \{(a_1,\dots, a_{2n+1})\in S_{2n+1}: a_{2n+2-i}=2n+2-a_i \, \forall 1\leq i\leq 2n+1\}.\]
In particular, $w=(a_1,\dots, a_{2n+1})\in W^B$ is determined from
$(a_1,\dots, a_{n})$. (Observe that $a_{n+1}=n+1$.) We can identify
the Weyl groups $W^C\simeq W^B$ under the map $(a_1,\dots, a_{2n})\mapsto (a_1,
\dots,  a_n, n+1, a_{n+1}+1, \dots, a_{2n}+1).$

Under the inclusion  $W^B\subset S_{2n+1}$, we have
 \begin{align}
 s'_i&=r_ir_{2n+1-i}, \,\,\,\text{if}\,\,\, 1\leq i\leq n-1\notag\\
 &=r_nr_{n+1}r_n, \,\,\,\text{if}\,\,\, i=n.
 \end{align}

For  $1\leq r \leq n$, we let $\OG(r,2n+1)=\OG(r,V')$ to be
the set of $r$-dimensional isotropic
subspaces of $V'$ with respect to the quadratic form $Q$, i.e.,
$$\OG(r,2n+1):=\{M\in \Gr(r,V'): Q(v)=0,\ \forall\,  v\in M\}.$$
 Then, it is the quotient $\SO(2n+1)/P_r^B$ of $\SO(2n+1)$ by
 the standard maximal parabolic subgroup
$P_r^B$  with $\Delta^B\setminus \{\delta_r\}$ as the set of simple roots of its Levi component
$L_r^B$. (Again we take $L_r^B$ to be the unique Levi subgroup of $P_r^B$
 containing $H^B$.) It can be easily seen that the set
$W^B_r$ of minimal-length coset representatives of $W^B/W_{P_r^B}$ is identified with the set
\[\FS'(r,2n+1)=\{{J}:= 1\leq j_1< \cdots <j_r\leq 2n+1 ,  j_p\neq n+1 \,
\text{for\, any}\, p \,\,{\text and}\,\, J\cap \bar{J}'=
\emptyset\},\]
where
$$\bar{J}':=\{2n+2-j_1, \dots, 2n+2-j_r\}.$$
Any such $J$ represents the permutation $w'_J=(j_1,\dots,
j_{n})
\in W^B$ by taking
$\{j_{r+1}<\cdots <j_n\}=[n]\setminus(J\sqcup \bar{J}').$

Similar to the Definition \ref{woodyallen} of isotropic flags on $V$, we have
 the notion of isotropic flags on $V'$.
Then, for an isotropic flag $E'_{\bull}$, there exists an element
$k(E'_{\bull})\in \SO(2n+1)$  which takes the standard flag
${E'}_{\bull}^o$ to the flag $E'_{\bull}$. (Observe that $k(E'_{\bull})$ is determined
up to the right multiplication by an element of $B^B$.)

For any $J\in \FS'(r,2n+1)$ and any isotropic flag $E'_{\bull}$, we have the
corresponding {\it shifted Schubert cell} inside $\OG(r,V')$:
$$\Psi_{J}(E'_{\bull})=\{M\in \OG(r,V'): \text{for any}\,
0\leq \ell \leq r \,
\text{and any}\, j_\ell\leq a < j_{\ell +1}, \dim M\cap E'_{a}=\ell\},$$
where we set $j_0=0$ and $j_{r+1}=2n+1$. Clearly, set
theoretically,
\begin{equation}\Psi_J(E'_{\bull})=\OMM_J(E'_{\bull})\cap
\OG(r,V');
\end{equation}
this is also a scheme theoretic equality.
 Moreover,  $\Psi_J(E'_{\bull})=k(E'_{\bull})C_{w'_J}^{P_r^B}$.
 Denote
the closure of $\Psi_J(E'_{\bull})$ inside $\OG(r,V')$  by
$\bar{\Psi}_J(E'_{\bull})$
and its cycle class  in $H^*(\OG(r,V'))$
 (which does not depend upon the choice of the isotropic flag
$E'_{\bull}$) by $[\bar{\Psi}_{J}]$.
For
the standard flag $E'_{\bull}=E^o_{\bull}$, we have
 $\Psi_{J}(E'_{\bull})=C^{P_r^B}_{w'_J}$.

\section{Comparison of the eigencones under diagram automorphisms}
\label{three3}

Fix a positive integer $s$. Let $V=\Bbb C^{2n}$ be equipped with the nondegenerate
 symplectic form $\langle\,,\,\rangle$ as in Section \ref{section4}, and let $1\leq r\leq n$ be
 a positive integer. Let $A^1,\dots,A^s\in S(r,2n)$. The following theorem is a key technical result that
underlies the proof of the comparison of eigencone for $\Sp(2n)$ with that
of $\SL(2n)$. The following results \ref{wilson1} - \ref{woodrow1} are due to
Belkale-Kumar [BK$_2$].

Instead of giving the original proof of the following theorem due to Belkale-Kumar
[BK$_2$], we give a shorter proof observed by Sottile [So] using the work of
Eisenbud-Harris [EH] on rational normal curves.

\begin{theorem}\label{wilson1} Let $E^1_{\bull},\dots,E^s_{\bull}$ be
isotropic flags on $V$  in general position. Then, the  intersection
of subvarieties $\cap_{j=1}^s \bar{\OMM}_{A^j}(E^j_{\bull})$ inside
$\Gr(r,V)$ is
 proper (possibly empty).
\end{theorem}
\begin{proof} Consider the rational normal curve $\gamma: \mathbb{C} \to V=\mathbb{C}^{2n},$
$$\gamma (t)=\bigl(1, t, \frac{t^2}{2!}, \dots, \frac{t^n}{n!}, -\frac{t^{n+1}}{n+1!},
\frac{t^{n+2}}{n+2!}, \dots, (-1)^{n-1}\frac{t^{2n-1}}{2n-1!}\bigr).$$
Defne the corresponding `osculating'  flag
$$E(t)_{\bull}: E(t)_1 \subset \dots \subset E(t)_{2n} $$
by $E(t)_j:=\mathbb{C} \gamma (t)\oplus \mathbb{C} \gamma^{(1)} (t)\oplus \dots \oplus
\mathbb{C}\gamma^{(j-1)} (t),$ where $ \gamma^{(k)} (t)$ is the $k$-th derivative of
$\gamma$ at $t$. Then, it is easy to see that $E(t)_{\bull}$ is an isotropic flag
for any value of $t$.

By a theorem due to Eisenbud-Harris [EH, Theorem 2.3], the intersection
$\cap_{j=1}^s \bar{\OMM}_{A^j}(E(t_j)_{\bull})$ inside
$\Gr(r,V)$ is proper if $t_1, \dots, t_s$ are distinct complex numbers. This
 proves the theorem.
 \end{proof}
 \begin{remark} Even though we do not need it, as observed by Sottile [So]
 using the work of Mukhin-Tarasov-Varchenko [MTV, Corollary 6.3], the intersection of the open
 cells $\cap_{j=1}^s {\OMM}_{A^j}(E(t_j)_{\bull})$ is transverse if
 $t_1, \dots, t_s$ are distinct {\it real}  numbers.
 \end{remark}

The following result follows as an immediate consequence of the above theorem.

\begin{corollary}\label{frasier} Let $1\leq r\leq n$ and let
 $I^1,\dots, I^s\in \FS(r,2n)$  be such that
$$\prod_{j=1}^s[\bar{\LAM}_{{I^j}}]\neq 0\in H^{*}(\IG(r,2n)).$$
 Then,  $\prod_{j=1}^s[\OMC_{{I^j}}]\neq 0\in H^{*}(\Gr(r,2n)).$
\end{corollary}
\begin{proof} Observe that by Proposition \ref{fultonproper},
\begin{equation} \label{3}\prod_{j=1}^s[\bar{\LAM}_{{I^j}}]\neq 0 \,\,\text{if
\,and \,only\, if}\,\, \cap_{j=1}^s
\,{\bar{\LAM}}_{{I^j}}(E^j_{\bull})\neq \emptyset
\end{equation}
for isotropic flags $\{E^j_{\bull}\}$ such that the above intersection is
 proper. Thus, by
assumption, $\cap_{j=1}^s \,{\bar{\LAM}}_{I^j}(E^j_{\bull})\neq \emptyset$
for such flags $\{E^j_{\bull}\}$. By the above theorem and Equation
~\eqref{2}, we conclude that $\cap_{j=1}^s
\,{\bar{\OMM}}_{{I^j}}(E^j_{\bull})\neq \emptyset$ and the intersection is proper
for  isotropic flags $\{E^j_{\bull}\}_{1\leq j\leq s}$ in general position. From this and using
Equation \eqref{3} for $\Gr(r, V)$, the corollary follows.
\end{proof}

We have the following analogue of Theorem  ~\ref{wilson1} for $\SO(2n+1)$
proved similarly by replacing the rational normal curve $\gamma$ by $\eta: \bc
 \to V=\bc^{2n+1}$ given by
$$\eta (t)=\bigl(1, t, \frac{t^2}{2!}, \dots, \frac{t^{n-1}}{n-1!}, \frac{t^n}{n!
\sqrt{2}}, -\frac{t^{n+1}}{n+1!},
\frac{t^{n+2}}{n+2!}, \dots, (-1)^{n}\frac{t^{2n}}{2n!}\bigr).$$
\begin{theorem}\label{newwilson2} Let $1 \leq r\leq n$.
 Let $A^1,\dots,A^s$ be
subsets of $[2n+1]$ each of cardinality $r$. Let
${E'}^1_{\bull}, \dots, {E'}^s_{\bull}$ be isotropic flags on $V'=\Bbb
C^{2n+1}$ in general position. Then, the intersection $\cap_{j=1}^s
\bar{\OMM}_{A^j}({E'}^j_{\bull})$  of subvarieties of $\Gr(r,V')$ is
 proper (possibly empty).
\end{theorem}

The following result follows as an immediate consequence of the above theorem
(just as in the case of $\Sp(2n)$).

\begin{corollary}\label{frasier'} Let $1\leq r\leq n$ and let
 $J^1,\dots, J^s\in \FS'(r,2n+1)$  be such that
$$\prod_{j=1}^s[\bar{\Psi}_{{J^j}}]\neq 0\in H^{*}(\OG(r,2n+1)).$$
 Then,  $\prod_{j=1}^s[\OMC_{{J^j}}]\neq 0\in H^{*}(\Gr(r,2n+1)).$
\end{corollary}

Recall that $\frh^C_+$ (respectively, $\frh^B_+$) is the dominant chamber in
the Cartan subalgebra of $\Sp(2n)$ (respectively, $\SO(2n+1)$) as in Section
\ref{section4}.

The following theorem provides a comparison of the eigencone for $sp(2n)$
 with that of  $sl(2n)$ (and also for  $so(2n+1)$
 with that of  $sl(2n+1)$).
\begin{theorem}\label{woodrow1} For any $s\geq 1$,
\begin{enumerate}
\item[(a)]   $\bar{\Gamma}_s(sp(2n))=
\bar{\Gamma}_s(sl(2n))\cap (\frh_{+}^{C})^s.$
\item[(b)] $\bar{\Gamma}_s(so(2n+1))=
\bar{\Gamma}_s(sl(2n+1))\cap (\frh_{+}^{B})^s.$
\end{enumerate}

(Observe that by Section \ref{section4}, $\frh_{+}^{C} \subset\frh_{+}^{A_{2n-1}}$
and $\frh_{+}^{B} \subset\frh_{+}^{A_{2n}}$.)
\end{theorem}
\begin{proof} Clearly, $\bar{\Gamma}_s(sp(2n))\subset \bar{\Gamma}_s(sl(2n))$.
Conversely, we
need to  show that if ${\bf h}=(h_1,\dots,h_s)\in (\frh_{+}^{C})^s$
is such that ${\bf h}\in  \bar{\Gamma}_s(sl(2n))$, then  ${\bf h}\in
\bar{\Gamma}_s(sp(2n))$. Take any
$1\leq r\leq n$ and any $I^1, \dots, I^s \in \FS(r,2n)$ such that
$$[\bar{\LAM}_{I^1}] \dots  [\bar{\LAM}_{I^s}]=
d[\bar{\LAM}_e] \in
 H^*(\IG(r,2n))\,  \,\text{for some nonzero}\, d .$$
By Corollary ~\ref{frasier},
$$[\OMC_{I^1}] \dots  [\OMC_{I^s}]\neq 0\in
H^{*}(\Gr(r,2n)).$$
In particular, by Corollary ~\ref{eigen} (rather Remark \ref{2.9} (a)) applied to $sl(2n)$,
\[\omega_r(\sum_{j=1}^sv_{I^j}^{-1}h_j)\leq 0,\]
where $\omega_r$ is the $r$-th fundamental weight of $\SL(2n)$ and
$v_{I^j}\in S_{2n}$ is the element associated to $I^j$ as in Subsection \ref{4.1}.
It is easy to see that the  $r$-th fundamental weight $\omega_r^C$ of $\Sp(2n)$
is the restriction of $\omega_r$ to $\frh^C$. Moreover, even though the elements
$v_{I^j}\in S_{2n}$ and $w_{I^j}\in W^C$ are, in general, different, we still have
$$\omega_r(v_{I^j}^{-1}h_j)=\omega_r^C(w_{I^j}^{-1}h_j).$$

Applying Corollary \ref{eigen} for $sp(2n)$, we get the (a)-part of
the theorem.

The proof for $so(2n+1)$ is similar. (Apply Corollary ~\ref{frasier'} instead of Corollary ~\ref{frasier}.)
\end{proof}
\begin{remark} (1) Belkale-Kumar have given a set of necessary and sufficient
conditions to determine the non-vanishing of any product of Schubert classes
$[X^P_w]$ in $\bigl(H^*(G/P), \odot_0\bigr)$ (under the deformed product) for any maximal parabolic
subgroup $P$ and any $G$ of type $B_n$ or $C_n$
(cf. [BK$_2$, Theorems 30,41 and Remarks 31,42]).

(2) For any $G$ of type $B_n$ or $C_n$, and any maximal parabolic subgroup
$P$, Ressayre has determined the triples $(w_1, w_2, w_3)\in (W^P)^3$ such that
$[X^P_{w_1}]\odot_0 [X^P_{w_2}]\odot_0 [X^P_{w_3}]= 1 [X^P_{e}]$ in terms of the corresponding
result for certain associated Schubert varieties in Grassmannians (cf. [R$_3$,
Theorems 14 and 15]).
\end{remark}

Let $\frg$ be a simple simply-laced Lie algebra and let $\sigma:\frg \to \frg$ be a
diagram automorphism with fixed subalgebra $\frk$ (which is necessarily a simple
Lie algebra again). Let $\frb$ (resp. $\frh$) be a  Borel (resp. Cartan) subalgebra
of $\frg$ such that they are stable under $\sigma$. Then, $\frb^\frk:=\frb^\sigma$
 (resp. $\frh^\frk:=\frh^\sigma$) is a  Borel (resp. Cartan) subalgebra
of $\frk$. Let $\frh_+$ and $\frh^\frk_+$ be the dominant chambers in $\frh$ and
$\frh^\frk$ respectively. Then,
$$ \frh^\frk_+ = \frh_+ \cap \frk.$$
We have the following generalization of Theorem \ref{woodrow1} conjectured by
Belkale-Kumar. (In fact, they have made a stronger conjecture, cf. Conjecture
\ref{conj}.)
\begin{theorem}\label{5.6} For any $s\geq 1$,
$$\bar{\Gamma}_s(\frk)= \bar{\Gamma}_s(\frg)\cap (\frh^\frk_+)^s.$$
(In the cases (d) and (e) as below, the theorem is proved only for $s=3$, though
it must be true for any $s$.)
\end{theorem}
\begin{proof} Unfortunately, the proof is case by case. Following is the complete
list of  $(\frg, \frk)$ coming from the diagram automorphisms of simple Lie
algebras $\frg$.

(a) $(sl(2n), sp(2n)), \, n\geq 2$

(b) $(sl(2n+1), so(2n+1)), \, n\geq 2$

(c) $(so(2n), so(2n-1)), \, n\geq 4$

(d) $(so(8), G_2)$

(e) $(E_6, F_4)$.

In the cases (a) and (b), the theorem is nothing but Theorem \ref{woodrow1}.

In the case (c), it was proved by E. Braley in her thesis [Bra]. Similar to the proof
of Theorem \ref{woodrow1}, her proof relies on the comparison
between the intersection theory of
the partial flag varieties $G/P$ of $G$ (corresponding to the maximal parabolic
subgroups $P$ of $G$)
with that of the   partial flag varieties $K/Q$ of $K$ (corresponding to the maximal
parabolic subgroups $Q$ of $K$). But her proof  uses
 the  deformed product in the cohomology of $K/Q$ and Corollary \ref{7.5},
 whereas she needs
 to use the standard cup product in the cohomology of $G/P$ and Corollary \ref{eigen}.

 The theorem for the cases (d) and (e) was proved by B. Lee in his thesis [Le].
 Lee  used the comparison
between the deformed product in the cohomology of $G/P$
 corresponding to the maximal parabolic
subgroups $P$ of $G$
with that of the  deformed product in the cohomology of $K/Q$
 corresponding to the maximal
parabolic subgroups $Q$ of $K$ (and Corollary \ref{7.5}). Lee used the recipe of Duan (cf. [D$_1$], [D$_2$])
to develop a program which
allowed him to explicitly calculate the deformed product
in the cohomology of the relevant flag varieties.
\end{proof}
\section{Saturation Problem}\label{secsaturation}

We continue to follow the notation
and assumptions from Section \ref{sec1}; in particular,
 $G$ is a  semisimple  connected complex
algebraic group.
In Section \ref{section2}, we defined the  saturated tensor
 semigroup $\Gamma_s(G)$ (for any integer $s\geq 1$)
  and  determined it by describing its
  facets (cf. Theorems \ref{thm1} and \ref{EVT}).

  Define the {\it tensor semigroup} for $G$:
  $$\hat{\Gamma}_{s}(G)=\left\{(\lambda_{1},\ldots,\lambda_{s})\in \Lambda_+^{s}:
      \left[V(\lambda_1)\otimes\cdots\otimes
        V(\lambda_{s})\right]^{G}\neq 0\right\}.
$$
It is indeed a semigroup by [K$_3$, Lemma 3.9]. The {\it saturation problem}
aims at comparing these two semigroups. We first prove that $\hat{\Gamma}_{s}(G)$ is
a finitely generated semigroup. More generally, we have the following result
(cf. [Br, Th\'eor\`eme 2.1]).
\begin{lemma} \label{lemme9.1} Let $S$ be a reductive subgroup of a connected semisimple group $G$. Let
$$\mathcal{D}_S= \{\lambda \in \Lambda_+: [V(\lambda)]^S\neq 0\},$$
where $\Lambda_+$ is the set of dominant characters of $G$. Then, $\mathcal{D}_S$
is a finitely generated semigroup.
\end{lemma}
\begin{proof} Since $S$ is reductive, by Matsushima's theorem, $G/S$ is an affine variety.
In particular, the affine coordinate ring $\bc[G/S]$ is a finitely generated
$\bc$-algebra. Now, by the Frobenius reciprocity,
\begin{align} \label{13n}
\bc[G/S] &\simeq \oplus_{\lambda \in \Lambda_+}\,V(\lambda)\otimes [V(\lambda)^*]^S,\,\,
\text{as $G$-modules, where $G$ acts only on the first factor} \notag\\
& = \oplus_{\lambda \in \mathcal{D}_S}\,V(\lambda)\otimes [V(\lambda)^*]^S.
\end{align}

Of course, $\bc[G/S] \hookrightarrow \bc[G]$. Consider the map $\Delta^*:
\bc[G]\otimes \bc[G] \to \bc[G]$ induced from the diagonal map
$\Delta: G\to G \times G$. Then, for the $G\times G$-isotypic component
$V(\lambda)\otimes V(\lambda)^*$ of $\bc[G]$, we have
$$\Delta^*\bigl((V(\lambda)\otimes V(\lambda)^*)\otimes (V(\mu)\otimes V(\mu)^*)\bigr)
\subset V(\lambda+\mu)\otimes V(\lambda+\mu)^*.$$
Take a finite set of algebra generators $f_1, \dots, f_N$ of $\bc[G/S]$ so that,
under the above decomposition \eqref{13n}, $f_p\in V(\lambda_p)\otimes
 [V(\lambda_p)^*]^S$ for some $\lambda_p\in \mathcal{D}_S$. Then, it is easy to see that
 these $\{\lambda_p\}$ generate the semigroup $\mathcal{D}_S$.
 \end{proof}
 As an easy consequence of the above lemma, we get the following.
 \begin{corollary} \label{saturationnew} There exists a uniform integer $d>0$ (depending only upon $s$ and $G$)
 such that for any
 $\lambda= (\lambda_{1},\ldots,\lambda_{s})\in {\Gamma}_{s}(G)$, $d\lambda=
 (d\lambda_{1},\ldots,d\lambda_{s})\in \hat{\Gamma}_{s}(G)$.
 \end{corollary}
 \begin{proof} Take a finite set of semigroup generators
 $\lambda^p= (\lambda^p_{1},\ldots,\lambda^p_{s})$ of ${\Gamma}_{s}(G)$, which exists
 by Theorem \ref{sj}, since $\bar{\Gamma}_{s}(\frg)$ is a rational polyhedral cone.
 Also, choose a finite set of semigroup generators
 $\mu^k= (\mu^k_{1},\ldots,\mu^k_{s})$ of $\hat{\Gamma}_{s}(G)$
 (cf. Lemma \ref{lemme9.1}). We can of course write
 $$\lambda^p=\sum_k\,a^k_p\mu^k,\,\,\text{for some non-negative rational numbers}
 \, a^k_p.$$
 Now, take $d>0$ large enough  so that $da^k_p\in \bz_+$ for all $a^k_p$. Take any
 $\gamma=(\gamma_{1},\ldots,\gamma_{s})\in {\Gamma}_{s}(G)$ and write
 \begin{align*}\gamma &=\sum n_p\lambda^p,\,\,\,\text{for some }\,n_p\in \bz_+\\
 &=\sum_{k,p}\, n_pa^k_p\mu^k\\
 &=\sum_k\bigl(\sum_p n_pa^k_p\bigr)\mu^k.
 \end{align*}
 This implies that
 $$d\gamma=\sum_k\bigl(\sum_p n_pda^k_p\bigr)\mu^k\in \hat{\Gamma}_{s}(G).$$
 \end{proof}

 We now begin with the following definition. We take $s=3$ as this is the most
 relevant case to the tensor product decomposition.
 \begin{definition} \label{sfactor} An integer $d\geq 1$ is called a {\it saturation factor} for $G$,
 if for any $(\lam,\mu,\nu)\in \Gamma_3(G)$ such that $\lam+\mu+\nu \in Q$, we have
 $(d\lam, d\mu,d\nu)\in \hat{\Gamma}_3(G)$, where $Q$ is the root lattice of $G$.
 Of course, if $d$ is a saturation
 factor then so is its any multiple. If $d=1$ is a saturation factor for $G$, we say that
 the {\it saturation property holds for $G$}.
 \end{definition}
The {\it saturation theorem} of Knutson-Tao (cf. Theorem \ref{knutson-tao}) asserts
 that the saturation property holds for $G=\SL(n)$.

The following general result (though not optimal) on saturation factor is obtained
 by Kapovich-Millson
[KM$_2$] by using the geometry of geodesics in Euclidean buildings and Littelmann's
path model (see the Appendix). A weaker form of the following theorem was conjectured by Kumar
in a private communication to J. Millson (also see [KT, Conjecture]).
\begin{theorem} \label{KM} For any connected simple $G$, $d=k_\frg^2$ is a saturated factor,
where $k_\frg$ is the least common multiple of the coefficients of the highest
root $\theta$ of the Lie algebra $\frg$ of $G$ written in terms of the simple roots
$\{\al_1,\dots, \al_\ell\}$.

Observe that the value of $k_\frg$ is $1$ for $\frg$ of type $A_\ell (\ell\geq 1)$;
it is $2$ for $\frg$ of type $B_\ell (\ell\geq 2), C_\ell (\ell\geq 3), D_\ell
 (\ell\geq 4)$; and it is 6, 12, 60, 12, 6 for $\frg$ of type $E_6,E_7,E_8,F_4,G_2$
 respectively.
 \end{theorem}

Kapovich-Millson determined $\hat{\Gamma}_3(G)$ explicitly for $G=\Sp(4)$ and $G_2$
(cf. [KM$_1$, Theorems 5.3, 6.1]). In particular, from their description,
 the following theorem follows easily.
\begin{theorem} The saturation property does not hold for either $G=\Sp(4)$ or $G_2$.
Moreover, $2$ is a saturation factor (and no odd integer $d$ is a saturation factor)
 for $\Sp(4)$, whereas both of $2,3$ are
saturation factors for  $G_2$ (and hence any integer $d>1$ is a saturation factor for
$G_2$).
\end{theorem}
It was known earlier that the saturation property fails for $G$ of type $B_\ell$ (cf. [E]).

 Kapovich-Millson [KM$_1$] made the following very interesting conjecture:
 \begin{conjecture} If $G$ is simply-laced, then the saturation
 property holds for $G$.
 \end{conjecture}

 Apart from  $G=\SL(n)$, the only other simply-connected, simple, simply-laced group
 $G$ for which the above conjecture is known so far is $G=\Spin (8)$, proved by
 Kapovich-Kumar-Millson [KKM, Theorem 5.3] by explicit calculation using Theorem ~\ref{EVT}.
 \begin{theorem} The above conjecture is true for $G=\Spin (8)$.
 \end{theorem}

Finally, we have the following improvement of Theorem ~\ref{KM} for the classical
groups $\SO(n)$ and $\Sp(2\ell)$. It was proved by  Belkale-Kumar [BK$_2$,
Theorems 25 and 26]  for the groups $\SO(2\ell+1)$ and $\Sp(2\ell)$ by using
geometric techniques.  Sam [S] proved it for $\SO(2\ell)$ (and also for
$\SO(2\ell+1)$ and $\Sp(2\ell)$) via the quiver approach (following the proof
by Derksen-Weyman [DW] for $G=\SL(n)$).
\begin{theorem} For the groups $\SO(n)\,(n\geq 7)$ and $\Sp(2\ell) \,(\ell \geq 2)$,
$2$ is a saturation factor.
\end{theorem}
The Belkale-Kumar proof of the above theorem for $\SO(2\ell+1)$ and $\Sp(2\ell)$
relies on the following theorem [BK$_2$, Theorem 23].

\begin{theorem}\label{clef'1}
Let $(\lambda^1, \dots , \lambda^s)\in \hat{\Gamma}_s(\SL(2\ell))$. Then,
$(\lambda_C^1, \dots ,\lambda_C^s)\in \hat{\Gamma}_s(\Sp(2\ell))$,
 where $\lambda_C^j$ is the restriction of
$\lambda^j$ to the maximal torus  of $\Sp(2\ell)$.

A similar result is true for $\Sp(2\ell)$ replaced by $\SO(2\ell+1).$
\end{theorem}

Belkale-Kumar [BK$_2$, Conjecture 29] conjectured the following generalization of Theorem
~\ref{clef'1}.
Let $G$ be a simply-connected, semisimple complex
algebraic group and let $\sigma$ be a diagram automorphism of $G$ (in particular,
$G$ is simply-laced)
with fixed subgroup $K$.

  \begin{conjecture}\label{conj}
  Let $(\lambda^1, \dots , \lambda^s)\in \hat{\Gamma}_s(G)$. Then,
$(\lambda_K^1, \dots ,\lambda_K^s)\in \hat{\Gamma}_s(K)$,
 where $\lambda_K^j$ is the restriction of
$\lambda^j$ to the maximal torus  of $K$.

(Observe that $\lam_K$ is dominant for $K$ for any dominant character
$\lam$ for $G$ with respect to the Borel subgroup $B^K := B^{\sigma}$
of $K$.)
  \end{conjecture}
\begin{remark} Lee showed in his thesis [Le] that the above conjecture is true for
the pair $($Spin $(8), G_2)$.
\end{remark}
We generalize Definition \ref{sfactor} in the following.
\begin{definition} \label{sfactor2} Let $G \subset \hat{G}$ be connected
reductive groups with the choice of $H, \hat{H}, B, \hat{B}$ as in the beginning
of Section \ref{section8}. An integer $d\geq 1$ is called a {\it saturation factor} for
the pair $(G, \hat{G})$,
 if for any $\hat{\lam} \in \hat{\Lambda}_+$ and $\lam \in \Lambda_+$
 such that

 (a) for all $t\in \hat{Z}\cap G$, $\lambda (t)\cdot \hat{\lam}(t)=1$, where
 $\hat{Z}$ is the center of $\hat{G}$, and

 (b) there exists $N>0$ such that $[V(N\lambda)\otimes
 \hat{V}(N\hat{\lambda})]^G\neq 0$,

 then we have $[V(d\lambda)\otimes
 \hat{V}(d\hat{\lambda})]^G\neq 0$.

 If we can take $d=1$, we say that
 the {\it saturation property holds for the pair $(G, \hat{G})$}.
 \end{definition}

 As proved by Pasquier-Ressayre [PS, Theorem 5], the pairs $(\Spin(2n-1), \Spin(2n));
 (\SL(3), G_2); (G_2, \Spin(7)); (\Spin(9), F_4); (F_4, E_6); (\Sp(2n), \SL(2n))$
 for any
 $2\leq n\leq 5$ have the saturation property.

 The following result is due to
 Haines-Kapovich-Millson [HKM, Corollary 3.4], though we give  a different proof (observed by  A. Berenstein)  reducing the problem to that of the saturation factor for 
$\hat{G}$.
 \begin{theorem}\label{hkm} Let $\hat{G}$ be any connected simple group
 and $G$ any Levi subgroup. Then, if  $d$ is a saturation factor for $\hat{G}$, then $d$ is also a saturation factor
 for the pair $(G, \hat{G})$. 

In particular, $k_{\hat{\frg}}^2$ is a saturation factor
 for the pair $(G, \hat{G})$,   where $\hat{\frg}$ is the Lie algebra of
 $\hat{G}$ and $k_{\hat{\frg}}$ is as defined in
 Theorem \ref{KM}.
 \end{theorem}
\begin{proof} Let $\omega_G:=\sum_{\alpha_i\in \hat{\Delta}\setminus\Delta}\,\omega_i$.
We first show that for any $\lambda\in \Lambda_+$ and $\hat{\lambda}\in \hat{\Lambda}_+$,
\beqn\label{neweq10.13}
\dim \bigl([V(\lambda)\otimes
 \hat{V}(\hat{\lambda})]^G\bigr)=\dim \bigl([\hat{V}(\hat{\lambda})\otimes \hat{V}(m\omega_G)\otimes 
\hat{V}(-\hat{w}_o(m\omega_G-{w}_o \lambda))]^{\hat{G}}\bigr),\eeqn
where $\hat{w}_o$ (resp. $w_o$) is the longest element of the Weyl group $\hat{W}$ of $\hat{G}$ (resp. $W$ of $G$) and $m=m_{\lambda,\hat{\lambda}}$ is 
any positive integer such that 
$e_i^{m+1}\cdot x=0$, for all $x\in \hat{V}(\hat{\lambda})$ and $\alpha_i\in \hat{\Delta}\setminus\Delta$ and such that 
$m\omega_G-w_o\lambda\in \hat{\Lambda}_+$ (where $e_i$ is the  root vector corresponding to the simple root $\alpha_i$) . 
To prove this, observe that (since $V(\lambda)^*\simeq 
V(-{w}_o\lambda)$)
\begin{align*} [V(\lambda)\otimes
 \hat{V}(\hat{\lambda})]^G &\simeq \Hom_G\bigl(V(-w_o\lambda), \hat{V}(\hat{\lambda})\bigr)\\
&\simeq \{v\in \hat{V}(\hat{\lambda})_{-w_o\lambda}:e_i\cdot v=0 \,\,\,\text{for all}\,\, \alpha_i\in \Delta\} \\
&\simeq \{v\in \hat{V}(\hat{\lambda})_{-w_o\lambda}:e_i^{m\langle \omega_G, \alpha_i^\vee\rangle +1} \cdot v=0 \,\,\,\text{for all}\,\, \alpha_i\in \hat{\Delta}\}.
\end{align*}

The last space has the same dimension as that of $[\hat{V}(\hat{\lambda})\otimes \hat{V}(m\omega_G)\otimes 
\hat{V}(-\hat{w}_o(m\omega_G-{w}_o \lambda))]^{\hat{G}}$ from [K$_3$, Theorem 3.7] .
This proves the identity \eqref{neweq10.13}. From the identity \eqref{neweq10.13},  the theorem follows easily by observing that 
$m_{N\lambda, N\hat{\lambda}}$ can be taken to be $Nm_{\lambda, \hat{\lambda}}$.
\end{proof}

\begin{remark} As shown by Roth [Ro] (also by Ressayre [R$_7$]), for a pair
 $(\lambda, \hat{\lambda})$
in any regular face (i.e., a face which intersects
$\Lambda_{++}\times \hat{\Lambda}_{++}$) of the cone $\Gamma (G, \hat{G})_\br$ (cf. Theorem \ref{ress2} for
the definition of  $\Gamma (G, \hat{G})_\br$), the dimension of the invariant
subspace $[V(\lambda)\otimes
 \hat{V}(\hat{\lambda})]^G$ is equal to a similar dimension for representations
 of Levi subgroups of $G$ and $\hat{G}$.
 \end{remark}

 We also recall the following `rigidity' result conjectured by Fulton and proved by
 Knutson-Tao-Woodward [KTW]. (Subsequently, geometric proofs were given by
Belkale [B$_4$] and Ressayre [R$_2$].)

\begin{theorem} \label{ktwfulton} Let
 $L=\GL(r)$ and let $\lambda,\mu,\nu\in \Lambda(H)_+$. Then,
if the dimension $\dim \bigl([V(\lambda)\otimes V(\mu)\otimes
V(\nu)]^{\SL(r)}\bigr) =1$, we have $\dim\bigl([V(n\lambda)\otimes
V(n\mu)\otimes V(n\nu)]^{\SL(r)}\bigr)=1$,
 for  every positive integer $n$.
\end{theorem}

The direct generalization of the above theorem for an arbitrary connected reductive
group $L$ is false. However, Belkale-Kumar-Ressayre [BKR] proved the following
generalization using the deformed product.

\begin{theorem} \label{ktwfulton1} Let $G$ be any connected reductive group and let $P$ be any
standard parabolic subgroup with the Levi subgroup $L$ containing $H$. Then, for
any $w_1, \dots, w_s\in W^P$ such that
\beqn \label{bkre1}[X_{w_1}^P]\odot_0\dots \odot_0 [X_{w_s}^P]=[X_e^P]\in
H^{*}(G/P, \odot_0),
\eeqn
we have, for every positive integer $n$,
\beqn \label{bkre2} \dim \bigl(\bigl[V_L(n\chi_{w_1})\otimes \dots \otimes
V_L(n\chi_{w_s})\bigr]^{L^{ss}}\bigr)=1,
\eeqn
where $L^{ss}$ denotes the semisimple part $[L,L]$ of $L$, $V_L(\lambda)$
is the irreducible representation of $L$ with highest weight
 $\lambda$ and $\chi_w :=\rho -2\rho^L+w^{-1}\rho$ ($\rho$ and $\rho^L$ being
 the half sum of positive roots of $G$ and $L$ respectively).
\end{theorem}

\section{Deformed product and Lie algebra cohomology}

We continue to follow the same notation and assumptions from Section \ref{sec1}.
We relate the cohomology algebra $H^*(G/P)$ under the product $\odot_0$
with the Lie algebra cohomology of the nil-radical $\fru_P$ of the parabolic subalgebra
$\frp$.

For any Lie algebra $\fs$ and a subalgebra $\ft$, let $H^*(\fs ,\ft )$
be the Lie algebra cohomology of the pair $(\fs ,\ft )$ with
trivial coefficients.  Recall (cf. [K$_1$, Section 3.1]) that this is the cohomology of the cochain
complex
  \begin{align*}
C\u. (\fs ,\ft ) &= \{ C^p(\fs ,\ft )\}_{p\geq 0}, \quad\text{ where}\\
C^p(\fs ,\ft ) &:= \Hom_{\ft}\bigl( \wedge^p(\fs /\ft ), \bc \bigr).
  \end{align*}

 For any (positive) root $\beta\in
R^+$, let $y_{\beta}\in \fg_{\beta}$ be a nonzero
root vector  and let $y_{-\beta}\in\fg _{-\beta}$ be
the vector such that $\ip<y_{\beta}, y_{-\beta}> =1$ under the Killing
form. For any $w\in W^P$, let $\Phi_w := w^{-1}R^-\cap R^+ \subset R(\fu_P)$.  Then,
as it is well known,
\begin{equation}\sum_{\beta\in\Phi_w} \beta = \rho -w^{-1}\rho.
\end{equation}
In particular,  $\Phi_v =
\Phi_w$ iff $v=w$.
Let $\Phi_w = \{ \beta_1,\dots ,\beta_p\} \subset
R(\fu_P)$.  Set $y_w := y_{\beta_1}\wedge \cdots\wedge y_{\beta_p}\in
\wedge^p(\fu_P)$, determined up to a nonzero scalar multiple.  Then, up to scalar multiples,
$y_w$ is the unique weight vector of $\wedge (\fu_P)$ with weight $\rho
-w^{-1}\rho$ (cf. [Ko, Lemma 5.12]).  Similarly, we can define $y^-_w
 := y_{-\beta_1}\wedge \cdots\wedge y_{-\beta_p}\in
\wedge^p(\fu^-_P)$ of
weight $w^{-1}\rho -\rho$.

\vskip1ex
We recall the following fundamental result due to Kostant  [Ko].

\begin{theorem}\label{Th9.A}  For any standard parabolic subgroup $P$ of
$G$,
  \[
H^p (\fu_P) = \bigoplus_{\substack{ w\in W^P:\\ \ell (w)=p}} M_w,
  \]
 as $\fl_P$-modules, where $M_w$ is the unique irreducible $\fl_P$-submodule of $H^p(\fu_P)$ with
highest weight $w^{-1}\rho -\rho$ (which is $\fl_P$-dominant for any
$w\in W^P$). This has  a highest weight vector $\phi_w\in\wedge^p(\fu_P)^*$ defined
by $\phi_w(y_w)= 1$ and $\phi_w(y)=0$ for any weight
vector of $\wedge^p(\fu_P)$  of weight $\neq \rho-w^{-1}\rho$.

Similarly, for the opposite nil-radical $\fu^-_P$,
  \[
H^p(\fu^-_P) = \bigoplus_{\substack{ w\in W^P:\\ \ell (w)=p}} N_w,
  \]
 as $\fl_P$-modules, where $N_w$ is the unique  irreducible $\fl_P$-submodule of
$H^p(\fu_P^-)$ isomorphic
with the dual $M^*_w$ and it has a lowest weight vector
$\phi^-_w\in\wedge^p(\fu^-_P)^*$ defined by $\phi_w^-(y_w^-)= 1$ and $\phi_w^-(y)=0$ for any weight
vector of $\wedge^p(\fu_P^-)$  of weight $\neq w^{-1}\rho -\rho$.

Thus,
  \begin{align*}
[H^p(\fu_P)\otimes H^q(\fu^-_P)]^{\fl_P} &=0, \quad\text{unless $p=q$, and}\\
[H^p(\fu_P)\otimes H^p(\fu^-_P)]^{\fl_P} &\simeq \bigoplus_{\substack{ w\in
W^P:\\ \ell (w)=p}} \Bbb C\xi^w,
  \end{align*}
where $\xi^w\in [M_w\otimes N_w]^{\fl_P}$ is the unique element whose
$H$-equivariant projection to $(M_w)_{w^{-1}\rho -\rho} \otimes N_w$ is the
element $\phi_w\otimes\phi_w^-$, $(M_w)_{w^{-1}\rho -\rho}$ being the weight
space of $M_w$ corresponding to the weight $w^{-1}\rho -\rho$.  (Observe that
the ambiguity in the choice of $y_w$ disappears in the definition of
$\xi^w$ giving rise to a completely unique element.)
\end{theorem}
The following theorem is due to Belkale-Kumar (cf. [BK$_1$, Theorem 43] for a proof).
  \begin{theorem}\label{liecohomology}  For any standard parabolic subgroup $P$ of $G$, there
is a graded algebra isomorphism
 $$\phi : \bigl( H^*(G/P, \bc ), \odot_0\bigr) \simeq \bigl[
H^*(\fu_P)\otimes H^*(\fu_P^-)\bigr]^{\fl_P}$$
such that, for any $w\in W^P$,
 \begin{equation}\label{eqn10.0} \phi\bigl(\epsilon_w^P\bigr) =
(-1)^{\frac{p(p-1)}{2}}\Bigl( \frac{i}{2\pi}\Bigr)^{p
}\ip<\rho
,\Phi_{w^{-1}}> \, \xi^{w} ,
\end{equation}
where $p:=\ell(w), \ip<\rho,\Phi_{w^{-1}}> :=\prod_{\al\in wR^-\cap R^+} \ip<\rho
,\al>$ (for any $w\in W$),  and we take the tensor product algebra structure on
the right side.
  \end{theorem}

A proof of the following corollary due to Belkale-Kumar can be found in
[BK$_1$, Corollary 44].
\begin{corollary}  \label{leviprod} The product in $(H^*(G/B), \odot_0)$
is given by
\begin{align*}
\epsilon^B_u \odot_0 \epsilon_v^B &= 0,\,\text{if \,}\,\Phi_u\cap\Phi_v \neq \emptyset\\
&=0, \text{if\,}\,\Phi_u\cap\Phi_v =
\emptyset \,\text{and} \, \not\exists \,w\in W \,\text{with}\,\, \Phi_w = \Phi_u\sqcup\Phi_v \\
&=\frac {\ip<\rho, \Phi_{u^{-1}}> \ip<\rho, \Phi_{v^{-1}}>}{\ip<\rho, \Phi_{w^{-1}}>}\,
\epsilon_w^B,  \text{if \,$\Phi_u\cap\Phi_v =
\emptyset$ and  $ \exists \,w\in W$ with  $\Phi_w = \Phi_u\sqcup\Phi_v$}.
  \end{align*}
  \end{corollary}
  As shown by Dimitrov-Roth [DR$_1$, Theorem 9.1.2], for any classical $G$ or
  $G=G_2$, and any $u,v,w\in W$ such that $\Phi_w = \Phi_u\sqcup\Phi_v$, the
  structure constant
  $$\frac {\ip<\rho, \Phi_{u^{-1}}> \ip<\rho, \Phi_{v^{-1}}>}
  {\ip<\rho, \Phi_{w^{-1}}>}=1.$$

  \begin{remark} (a) The above result Theorem \ref{liecohomology} identifying
$H^*(G/P)$ under the deformed product with the Lie algebra cohomology
has crucially been used  (though for affine $G$) by Kumar in the solution of the Cachazo-Douglas-Seiberg-Witten conjecture
(cf. [K$_2$]).

(b)
   Evens-Graham
have realized the algebra   $\bigl(H^*(G/P), \odot_t\bigr) $ (for any value of
$t=(t_1, \dots, t_m) \in \Bbb C^m$, where $m:=|\Delta|-|\Delta(P)|$) as
the relative cohomology algebra $H^*(\mathfrak g_t, \mathfrak l_\Delta)$ for certain
Lie subalgbebras $\mathfrak g_t \supset \mathfrak l_\Delta$ of $\mathfrak g \times
\mathfrak g$ (cf. [EG$_1$]).

Let $J_t:=\{\alpha_q,
1\leq q\leq \ell: \alpha_q\in \Delta\setminus \Delta (P) \,\,\text{and}\,
t_q\neq 0\}$, $D_t:=\Delta (P) \cup J_t$ and
let $P_t\supset P$ be the parabolic subgroup of $G$ such that its Levi subgroup has
$D_t$ for its set of simple roots.

Now, Evens-Graham prove that the standard singular
 cohomology algebra $H^*(P_t/P)$ is isomorphic, as an algebra, to a certain
 graded subalgebra $A_t$ of $\bigl(H^*(G/P), \odot_t\bigr). $ Moreover, the algebra
 $\bigl(H^*(G/P_t), \odot_0\bigr) $ is isomorphic, as an algbera, with
 $\bigl(H^*(G/P), \odot_t\bigr)/I_+ $, where $I_+$ is the graded ideal of
$\bigl(H^*(G/P), \odot_t\bigr)$ generated by the positive degree elements in $A_t$
(cf. [EG$_2$]).

\end{remark}

\section{A restricted functoriality of the deformed product and a product formula}
Let the notation and assumptions be as in the beginning of Section \ref{section8}.
Take a $G$-dominant OPS
$\delta \in O(H)$. Thus, $P(\delta)$ is a standard parabolic since $\delta$ is
dominant for $G$. Moreover, the choice of the Borel subgroup $\hat{B}$ is made
so that $B \subset \hat{B}\supset \hat{H}$ and $\hat{B} \subset \hat{P}(\delta)$.
(Such a $\hat{B}$ depends upon the choice of $\delta$.) We have the embedding
$\iota: G/P(\delta) \to \G/\hat{P}(\delta)$. Define a $\bz[\tau]$-linear
product  $\odot^\delta$  (with single indeterminate $\tau$) in $H^*( G/P, \bz)\otimes_\bz\,\bz[\tau]$ by
$$\label{newEE'}
[X^P_u] \odot^\delta [X^P_v]=
\sum_{w\in W^{P}}
\tau^{(w^{-1}\rho-u^{-1}\rho -v^{-1}\rho -\rho)(\dot{\delta})}
\bigr)
c^w_{u,v} [X^{P}_w],
$$
where $P:=P(\delta)$ and $ c^w_{u,v}$ is given by
$$
[X^P_u]\cdot[X^P_v]=\sum_{w\in W^P} c^w_{u,v}[X^P_w].
$$
By  Corollary \ref{product}, the exponent $(w^{-1}\rho-u^{-1}\rho -v^{-1}\rho -\rho)(\dot{\delta})\geq 0$, whenever
$c^w_{u,v}\neq 0$,  since
$\dot{\delta}\in \sum_{\alpha_i\in \Delta \setminus \Delta(P)}\, \bz_+x_i.$

 Define a similar product, again denoted by  $\odot^\delta$, in
$H^*(\G/\hat{P}, \bz)\otimes_\bz\,\bz[\tau]$, where $\hat{P}:=\hat{P}(\delta)$.

In particular, we can specialize $\tau=0$ in the above product $ \odot^\delta$. Since
$\alpha_i(\dot{\delta}) >0$ for any $\alpha_i\in \Delta \setminus \Delta(P)$, it is
easy to see from Corollary \ref{product} that
\begin{equation}\label{eqn105} \bigl([X^P_u] \odot^\delta [X^P_v]\bigr)_{\tau=0}=
[X^P_u] \odot_0 [X^P_v].
\end{equation}
A similar result is true for the product $\odot^\delta$ in $H^*(\G/\hat{P}, \bz)$.
Let $\iota^*: H^*(\G/\hat{P}, \bz)\to  H^*( G/P, \bz)$ be the standard pull-back
map in cohomology.  Write
$$\iota^*([\hat{X}^{\hat{P}}_{\hat{w}}])= \sum_{w\in W^P}\,d_w^{\hat{w}}[X^P_w] .$$
Now, define a  $\bz[\tau]$-linear map
$$\theta^\delta : H^*(\G/\hat{P}, \bz)\otimes_\bz\,\bz[\tau]\to
H^*( G/P, \bz)\otimes_\bz\,\bz[\tau]$$
by $$\theta^\delta ([\hat{X}^{\hat{P}}_{\hat{w}}])=
\sum_{w\in W^P}\,\tau^{\chi_{w}(\dot{\delta})-\hat{\chi}_{\hat{w}}(\dot{\delta})}
\,d_w^{\hat{w}}[X^P_w] ,$$
where $\chi_w$ is given by the identity \eqref{eqn5}. By an argument similar to the
proof of Corollary \ref{product}, we can see that  if $d_w^{\hat{w}}\neq 0$, then
$\chi_{w}(\dot{\delta})-\hat{\chi}_{\hat{w}}(\dot{\delta}) \geq 0$. Thus, the map
$\theta^\delta$ is well defined.

Let $\theta^\delta_0:  H^*(\G/\hat{P}, \bz)\to  H^*( G/P, \bz)$ be the map obtained
by setting $\tau=0$ in the definition of  $\theta^\delta$. Let us express
$$\theta^\delta_0([\hat{X}^{\hat{P}}_{\hat{w}}])= \sum_{w\in W^P}\,c_w^{\hat{w}}[X^P_w] .$$
We have the following result due to Ressayre-Richmond [ReR, Theorem 1.1].
\begin{theorem} \label{rr}The map   $\theta^\delta_0: H^*(\G/\hat{P}, \bz)
\to H^*( G/P, \bz)$ is a graded algebra homomorphism with respect to the deformed
products on both the domain and the range. Moreover, it satisfies
$$c_w^{\hat{w}}=d_w^{\hat{w}}, \,\,\,\text{if}\,\,c_w^{\hat{w}}\neq 0.$$
\end{theorem}
\begin{proof} It is easy to see, by an explicit calculation, that $\theta^\delta$ is a graded
 $\bz[\tau]$-algebra homomorphism  with respect to the
products $\odot^\delta$ on both the domain and the range. From this and the
identity \eqref{eqn105}, the theorem follows immediately.
\end{proof}
\begin{remark}
 (1) As observed by Ressayre-Richmond [ReR, Lemma 3.3], it is easy to see that if
$G/P$ is cominuscule, then  $\theta^\delta_0 = \iota^*$. (Use the identity
\eqref{eqn5.2},
the definition of $\hat{\chi}_{\hat{w}}$ as in the identity \eqref{eqn5new} and
the nonnegativity of
$\chi_{w}(\dot{\delta})-\hat{\chi}_{\hat{w}}(\dot{\delta})$ if
$d_w^{\hat{w}}\neq 0$.)

(2) The map $\theta_0^\delta$  is partially computed for the pairs $(\SL(2), \SL(n)),
(\SL(n)\times \SL(n), \SL(n^2))$ and $(\SO(2n+1), \SL(2n+1))$ by Ressayre-Richmond
[ReR, $\S$4].

(3) Clearly, the conditions $(c_1)$ and $(c_2)$ in Theorem \ref{ress} can be replaced
by the condition
$$[X_w^{P(\delta_i)}]\cdot
\theta^{\delta_i}_0([\hat{X}^{\hat{P}(\delta_i)}_{\hat{w}}])=[X_e^{P(\delta_i)}]
\in H^*(G/P(\delta_i), \mathbb{Z}),$$
cf. [ReR, Theorem 5.1].
\end{remark}

{\it We follow the following notation and assumptions till the end of this section.}

Let $G\subset \G$ be connected reductive groups. Let ${B}\subset
G$ and $\hat{B}\subset \G$ be Borel subgroups, and
${H}\subset {B}$ and $\hat{H}\subset \hat{B}$ be maximal
tori. We assume that
$
{H}\subset \hat{H}
$
and there exists $x\in N(\hat{H})$  such that
$
{B}=x \hat{B} {x}^{-1}\cap G$, where $N(\hat{H})$ is the normalizer of $\hat{H}$ in
$\G$.

Let $\hat{B}\subset \hat{P}\subset \hat{Q}$ be (standard) parabolic subgroups in
$\G$. Define the standard parabolic subgroups in $G$:
$
{P} = x \hat{P} x^{-1}\cap G,\,\,
{Q} = x \hat{Q} x^{-1}\cap G.$

Define an embedding of the flag verieties
$$
f_2:G/{P}\hookrightarrow \G/\hat{P},\,\,
{g}{P}\mapsto x^{-1}{g}x\hat{P}
$$
and similarly
$f: G/{Q}\hookrightarrow
\G/\hat{Q}$. Then, we have a commutative diagram
\[
\xymatrix{
{Q}/{P}\ar@{^{(}->}[r]^{f_1}\ar[d] & \hat{Q}/\hat{P}\ar[d]\\
G/{P}\ar@{^{(}->}[r]^{f}\ar[d] & \G/\hat{P}\ar[d]\\
G/{Q}\ar@{^{(}->}[r]^{f_2} & \G/\hat{Q}
}
\]
where the vertical maps are the standard maps.
The Weyl group for $\G$ is denoted by $\hat{W}$ and similarly ${W}$ for
$G$.
Let ${\hat{w}}\in \hat{W}^{\hat{P}}$ be such that
\begin{equation}\label{rr1}
 \dim \G/\hat{P} - \ell({\hat{w}}) =\dim G/{P},\quad\text{and}
\,\,\dim \G/\hat{Q} - \ell({\hat{u}})=\dim G/{Q},
\end{equation}
where ${\hat{w}}={\hat{u}}{\hat{v}}$ is the unique decomposition with ${\hat{u}}\in
\hat{W}^{\hat{Q}}$ and ${\hat{v}}\in
\hat{W}^{\hat{P}}\cap \hat{W}_{\hat{Q}}$. Thus, we automatically get
\begin{equation}\label{rr1'}
\dim \hat{Q}/\hat{P}-\ell({\hat{v}})=\dim {Q}/{P}.
\end{equation}
Recall from Section \ref{section6} that ${\hat{\Phi}}^{\hat{P}}_{{\hat{w}}}$ is the shifted cell
${\hat{w}}^{-1}\hat{B}{\hat{w}}\hat{P}/\hat{P}\subset \G/\hat{P}$.
\begin{lemma} \label{rrlem1} For any $g=q\dot{{\hat{u}}}^{-1}$, with $q\in \hat{Q}$
and a representative
$\dot{{\hat{u}}}$ of ${\hat{u}}$,
$$g{\hat{C}}_{\hat{w}}^{\hat{P}}\cap \hat{Q}/\hat{P}= q{\hat{C}}_{\hat{v}}^{\hat{P}}.$$
\end{lemma}
\begin{proof} Let $\hat{R}_{\hat{w}}:={\hat{R}}^+\cap {\hat{w}}^{-1}{\hat{R}}^-$, where ${\hat{R}}^+$ (resp. ${\hat{R}}^-$) is the set
of positive (resp. negative) roots of $\G$. Let ${\hat{U}}_{{\hat{R}}_{\hat{w}}}$
(resp.  ${\hat{U}}^-_{{\hat{R}}_{\hat{w}}}$)
be the unipotent subgroup
of the unipotent radical of $\hat{B}$ (resp. ${\hat{B}}^-$) such that its Lie
algebra has roots
 ${\hat{R}}_{\hat{w}}$ (resp. $-{\hat{R}}_{\hat{w}}$). Then, it is easy to see from [K$_1$, Lemma 1.3.14]
 that
 \begin{equation}\label{rr0}{\hat{\Phi}}^{\hat{P}}_{{\hat{w}}}={\hat{v}}^{-1}{\hat{U}}^-_{{\hat{R}}_{\hat{u}}}
 {\hat{v}}{\hat{U}}^-_{{\hat{R}}_{\hat{v}}}\hat{P}/\hat{P}.
 \end{equation}
 Also, it is easy to see that
 \begin{equation} \label{rreq0}
 {\hat{U}}^-_{{\hat{R}}_{\hat{u}}}\cap \hat{Q}=(1).
\end{equation}
 Thus, by the identities \eqref{rr0} and \eqref{rreq0},
 \begin{equation}\label{req1}
 g{\hat{C}}_{\hat{w}}^{\hat{P}}\cap \hat{Q}/\hat{P}=q{\hat{v}}{\hat{U}}^-_{{\hat{R}}_{\hat{v}}}
 \hat{P}/\hat{P}= q{\hat{C}}_{\hat{v}}^{\hat{P}}.
 \end{equation}
 This proves the lemma.
 \end{proof}

\begin{definition}\label{rrdefi1}
Define a subset $\mathscr{X}=\mathscr{X}_{{\hat{u}}}$ by
$$
\mathscr{X}=\{(\overline{{g}}, \overline{h})\in
G/{Q}\times \G/\hat{B}: f_2(\overline{{g}})\in h{\hat{C}}_{{\hat{u}}}^{\hat{Q}}\},
$$
where ${\hat{C}}_{{\hat{u}}}^{\hat{Q}}$ is the Schubert cell $\hat{B}{\hat{u}}\hat{Q}/\hat{Q}$ in
$\G/\hat{Q}$, and $\overline{{g}}$ denotes $gQ$ etc.
\end{definition}

Let ${Q}$ act on $\hat{Q}{\hat{u}}^{-1}\hat{B}/\hat{B}$ via
$
{q}\odot z = (x^{-1}{q}x)\cdot z.
$
\begin{lemma}\label{rrlem2}
There is an isomorphism
$$
  \mu:G\displaystyle{\times^{{Q}}}(\hat{Q}{\hat{u}}^{-1}\hat{B}/\hat{B})
  \xrightarrow{\sim}\mathscr{X},\,\,\,
 \mu
     [{g},z]=(\overline{{g}},(x^{-1}{g}x)\cdot z),\,\,\, \text{for}\,
     {g}\in G,\,\, z\in \hat{Q}{\hat{u}}^{-1}\hat{B}/\hat{B}.
$$
Thus, $\mathscr{X}$ is an irreducible smooth variety.
\end{lemma}
\begin{proof}
For $\overline{h} \in \G/\hat{B}$, $(\overline{1},\overline{h})\in
\mathscr{X}\Leftrightarrow 1\in h \  \hat{B}{\hat{u}} \ \hat{Q}/\hat{Q}\Leftrightarrow h\in
\hat{Q}{\hat{u}}^{-1}\hat{B}$.
Moreover, $(\overline{{g}},\overline{h})\in
\mathscr{X}\Leftrightarrow (\bar{1}, x^{-1}{g}^{-1}x\overline{h})\in
\mathscr{X}$.
From this it is easy to see that $\mu$ is an isomorphism.
\end{proof}

\begin{definition}\label{rrdefi3} Let $\xi_{{\hat{u}}}:\hat{Q}{\hat{u}}^{-1}\hat{B}/\hat{B}
\to \hat{Q}/\hat{B}$ be
the map $q{\hat{u}}^{-1}\hat{B}\mapsto q\hat{B}$, for $q\in \hat{Q}$. This is well
defined since $({\hat{u}}^{-1}\hat{B}{\hat{u}})\cap \hat{Q}= ({\hat{u}}^{-1}\hat{B}{\hat{u}})\cap \hat{B}$
(and clearly $\hat{Q}$-equivariant).

Define a subset $\mathscr{X}^o\subset \mathscr{X}$ consisting of
$
(\overline{{g}},\overline{h})\in \mathscr{X}$ satisfying:

(a) $(x^{-1}{g}^{-1}xh{\hat{C}}_{\hat{w}}^{\hat{P}})\cap \hat{Q}/\hat{P}$ intersects
  $f_1({Q}/{P})$ in $\hat{Q}/\hat{P}$ transversally at every point of
  the intersection, and

(b)  $(\xi_{{\hat{u}}}(x^{-1}{g}^{-1}x\bar{h}){\hat{C}}_{\hat{v}}^{\hat{P}})\cap
  f_1({Q}/{P})=(\xi_{{\hat{u}}}(x^{-1}{g}^{-1}
  x\bar{h}){{\hat{X}}}_{{\hat{v}}}^{\hat{P}})\cap f_1({Q}/{P})$.

(Recall that since $(\overline{{g}},\overline{h})\in
\mathscr{X}$, we have
$
x^{-1}{g}^{-1}xh\in \hat{Q}{\hat{u}}^{-1}\hat{B}
$
by the proof of Lemma \ref{rrlem2}.
Moreover, by Lemma \ref{rrlem1},
   $(x^{-1} {g}^{-1} xh{\hat{C}}_{\hat{w}}^{\hat{P}})\cap
  \hat{Q}/\hat{P}=\xi_{{\hat{u}}}(x^{-1} {{g}^{-1}} x \bar{h}){\hat{C}}_{{\hat{v}}}^{\hat{P}}$;
  in particular, it is smooth.
\end{definition}

\begin{definition}\label{rrdefi4}Let ${Q}$ act on $\hat{Q}/\hat{B}$ via
${q}\odot z=(x^{-1}{q}x)\cdot z$.
Define
$
\xi:\mathscr{X}\to \mathcal{Z}:=
G\displaystyle{\times^{{Q}}}\hat{Q}/\hat{B}
$
by
$$
\xi (\mu[{g},z])=[{g},\xi_{{\hat{u}}}(z)],\quad\text{for}\quad
  {g}\in G, z\in \hat{Q}{\hat{u}}^{-1}\hat{B}/\hat{B}.
$$
\end{definition}

\begin{proposition}\label{rrprop5}
The subset $\mathscr{X}^{o}$ contains a dense open subset of
$\mathscr{X}$.
\end{proposition}

\begin{proof} By Theorem \ref{kleiman} and the identity \eqref{rr1'},
there exists a dense open subset  $V\subset \hat{Q}/\hat{B}$, which is
stable under the left multiplication by
$x^{-1}{Q}x$,  such that  for any
$\overline{q}\in V$

(a) $q{\hat{C}}_{\hat{v}}^{\hat{P}}\cap f_1(Q/{P})$ is a transverse
  intersection in $\hat{Q}/\hat{P}$ (at any point of the intersection), and

(b) $q{\hat{C}}_{\hat{v}}^{\hat{P}}\cap
  f_1({Q}/{P})=q{\hat{X}}_{{\hat{v}}}^{\hat{P}}\cap
  f_1({Q}/{P})$.

Now, $(\overline{{g}},\overline{h})\in \mathscr{X}$ belongs to
$\mathscr{X}^{o}$ if $\xi_{{\hat{u}}}(x^{-1}{g}^{-1}x\overline{h})\in
V$.
Thus, $\mathscr{X}^{o}$ contains a dense open subset of $\mathscr{X}$.
\end{proof}

Let $\sigma:\mathscr{X}\to \G/\hat{B}$ be the projection on the second factor.

\begin{lemma}\label{rrlem6}
Assume that $f_2^{*}[{{\hat{X}}}_{{\hat{u}}}^{\hat{Q}}]\neq 0 \in H^*(G/Q)$. Then, $\sigma$ is a
dominant morphism. Moreover,
$$
\dim \mathscr{X}=\dim \G/\hat{B}.
$$
\end{lemma}

\begin{proof}
Since $\codim {\hat{C}}_{{\hat{u}}}^{\hat{Q}}=\dim G/{Q}$ (by assumption \eqref{rr1}) and
 $f_2^{*}([{{\hat{X}}}_{{\hat{u}}}^{\hat{Q}}])\neq 0$,
we get that   $h {\hat{C}}_{{\hat{u}}}^{\hat{Q}}\cap f(G/{Q})$ is a finite nonempty
subset for general  $h\in
\G$. Thus, the map $\sigma$ is dominant and on a dense open subset
of $\G/\hat{B}$, $\sigma$ has finite fibres. Thus,
$
\dim \G/\hat{B}=\dim \mathscr{X}.
$
\end{proof}

The following result, as well as Theorem \ref{rrthm2}, is due to Richmond [Ri$_2$]
(and also due to Ressayre [R$_6$]).
\begin{theorem}\label{rrthm7}
Let ${\hat{w}}\in \hat{W}^{\hat{P}}$ be such that it satisfies the condition \eqref{rr1}.
Write
$$
 f^{*}([{{\hat{X}}}_{{\hat{w}}}^{\hat{P}}])=d[pt] \in
  H^{*}(G/{P}),\,\,\,
f_2^{*}([{{\hat{X}}}_{{\hat{u}}}^{\hat{Q}}])=d_{1}[pt]\in
  H^{*}(G/{Q}),\,\,\,
 f_1^{*}([{{\hat{X}}}_{{\hat{v}}}^{\hat{P}}])=d_{2}[pt]\in
  H^{*}({Q}/{P}),
$$
for some integers $d,d_1,d_2$.
Then, $d=d_{1}d_{2}$.

(Observe that since ${\hat{v}}\in \hat{W}^{\hat{P}}\cap \hat{W}_{\hat{Q}}$,
${{\hat{X}}}_{{\hat{v}}}^{\hat{P}}\subset \hat{Q}/\hat{P}$.)
 \end{theorem}

\begin{proof}
Assume first that $d_1\neq 0$.
Choose general elements $y=\overline{h}\in \G/\hat{B}$ such that
\begin{itemize}
\item[(a)] $h {\hat{C}}_{{\hat{u}}}^{\hat{Q}}\cap f_2(G/{Q})$ is a transverse
  intersection

\item[(b)] $h{\hat{C}}_{{\hat{u}}}^{\hat{Q}}\cap
  f_2(G/{Q})=h {\hat{X}}_{{\hat{u}}}^{\hat{Q}}\cap
  f_2(G/{Q})$

\item[(c)] $h{\hat{C}}_{\hat{w}}^{\hat{P}}\cap f(G/{P})$ is a transverse
  intersection

\item[(d)] $h{\hat{C}}_{\hat{w}}^{\hat{P}}\cap
  f(G/{P})=h{\hat{X}}_{{\hat{w}}}^{\hat{P}}\cap
  f(G/{P})$

\item[(e)] for all $\overline{{g}}\in G/{Q}$ such
  that
$
(\overline{{g}}, \overline{h})\in \mathscr{X},
  (\overline{{g}},\overline{h})\in \mathscr{X}^{o}.
$
\end{itemize}

We now show the existence of such $y$'s. Let $\mathfrak{V}$ be an open dense
subset of $\G/\hat{B}$ satisfying (a)-(d), which exists by Theorem \ref{kleiman}
and the assumption \eqref{rr1}.
 Take a dense open subset
$\hat{\mathscr{X}}^{o}$ of $\mathscr{X}$ contained in
$\mathscr{X}^{o}\cap {\sigma}^{-1}(\mathfrak{V})$. Then,
$$
\dim \bigl(\overline{\sigma(\mathscr{X}\backslash \hat{\mathscr{X}}^o)}\bigr)
\leq \dim
(\mathscr{X}-\hat{\mathscr{X}}^o)<\dim \mathscr{X}=\dim \G/\hat{B},
$$
where the last equality follows by Lemma \ref{rrlem6}.
Thus,
$
{\sigma}^{-1}\left(\overline{\sigma(\mathscr{X}\backslash \hat{\mathscr{X}}^o)}\right)
$
is a proper closed subset of $\mathscr{X}$. Take any $y'\in
\mathscr{X}\backslash\sigma^{-1}\left(\overline{\sigma(\mathscr{X}\backslash
  \hat{\mathscr{X}}^{o})}\right)$.
Then, clearly $y'\in \mathscr{X}^{o}$ and
$y=\sigma(y')\in \mathfrak{V}$. This proves the existence of $y$ satisfying
(a)-(e).

For any  $y=\overline{h}\in \G/\hat{B}$ satisfying the conditions (a) -(e),
$$
 d_{1}=|(h{\hat{C}}_{{\hat{u}}}^{\hat{Q}})\cap f_2(G/{Q})|,\,\,\,
 d=|(h{\hat{C}}_{{\hat{w}}}^{\hat{P}})\cap f(G/{P})|.$$
Moreover, under the projection map
$ (h{\hat{C}}_{\hat{w}}^{\hat{P}})\cap f(G/{P})\xrightarrow{\pi} (h{\hat{C}}_{\hat{u}}^{\hat{Q}})\cap f_2(G/{Q}),
$ for any $\overline{{g}}\in {f_2}^{-1}(h{\hat{C}}_{\hat{u}}^{\hat{Q}}\cap
f_2(G/{Q}))$,
$$
{\pi}^{-1}\left(f_2(\overline{{g}})\right) \simeq
(x^{-1}{g}^{-1}xh{\hat{C}}_{\hat{w}}^{\hat{P}} \cap \hat{Q}/\hat{P})\cap
f_1({Q}/{P})
= (\xi_{{\hat{u}}}(x^{-1}{g}^{-1}x\overline{h}){\hat{C}}_{\hat{v}}^{\hat{P}})\cap
f_1({Q}/{P}),
$$
where the last equality follows from Definition \ref{rrdefi3} and the
condition (e).  The last intersection is a transverse intersection in
 $\hat{Q}/\hat{P}$ and
$$
(\xi_{{\hat{u}}}({x}^{-1}{{g}}^{-1}x\overline{h}){\hat{X}}_{\hat{v}}^{\hat{P}})\cap
f_1({Q}/{P})=(\xi_{{\hat{u}}}(x^{-1}{g}^{-1}x\overline{h}){\hat{C}}_{\hat{v}}^{\hat{P}})\cap
f_1(Q/P),
$$
by the definition of ${\mathscr{X}}^{o}$. Thus,
$$
|\pi^{-1}(f_2(\overline{{g}}))|=d_{2}.
$$
This gives $d=d_{1}d_{2}$, proving the Theorem for the case
$d_{1}\neq 0$.

We finally show that if $d_{1}=0$, then $d=0$. For, if not, take general
$g\in \G$ such that
$ g{\hat{C}}_{\hat{w}}^{\hat{P}}\cap f(G/{P})$ is
    nonempty. Then,
$g{\hat{C}}_{\hat{u}}^{\hat{Q}}\cap f(G/{Q})$ is nonempty too since
$\pi(g{\hat{C}}_{\hat{w}}^{\hat{P}}\cap f(G/{P}))\subset g{\hat{C}}_{\hat{u}}^{\hat{Q}}\cap
  f(G/{Q}).$
This proves the theorem completely.
\end{proof}
A particular case of the Definition \ref{bknewdef} is the following.
\begin{definition}\label{rrdef5} Let ${\hat{w}}\in \hat{W}^{\hat{P}}$ be such that codim
${\hat{\Phi}}^{\hat{P}}_{{\hat{w}}}= \dim G/P$. Then, ${\hat{\Phi}}^{\hat{P}}_{{\hat{w}}}$ is said to be
{\it $L$-movable for the embedding} $f:G/P \to \G/\hat{P}$ if
$$T_{e}(G/{P})
\xrightarrow{{(df)_e}}\dfrac{T_{{\hat{e}}}(\G/\hat{P})}{T_{{\hat{e}}}({\hat{l}}{\hat{\Phi}}^{\hat{P}}_{{\hat{w}}})}$$
is an isomorphism for some ${\hat{l}}\in {\hat{L}}_{\hat{P}}$, where ${\hat{L}}_{\hat{P}}$ is
 the Levi
subgroup of $\hat{P}$ containing
$\hat{H}$.
\end{definition}
Recall the definition of the elements $x_i\in \frh$ from the equation \eqref{eqn0}.
\begin{theorem}\label{rrthm2}
Let ${\hat{w}}\in \hat{W}^{\hat{P}}$ be such that ${\hat{\Phi}}^{\hat{P}}_{{\hat{w}}}\subset \G/\hat{P}$ is
$L$-movable for the embedding $f:G/{P}\to \G/\hat{P}$ (in particular,
 $\codim {\hat{\Phi}}^{\hat{P}}_{{\hat{w}}}=\dim
G/{P}$). Assume further that
there exists a dominant regular element $y_{o}\in \mathfrak{z}(L_{{Q}})$
(i.e., an element of the form $y_{o}=\sum\limits_{\alpha_{i}\in
  \Delta (G)\backslash \Delta (Q)}\,r_{i}x_{i}$, \,
$r_{i}>0$) such that $\Ad(x^{-1})\cdot y_{o}\in \mathfrak{z}({\hat{L}}_{\hat{Q}})$, where
$\mathfrak{z}(L_{{Q}})$ denotes the Lie algebra of the center of the Levi
subgroup $L_{{Q}}$ of ${Q}$. Then, ${\hat{\Phi}}^{\hat{P}}_{{\hat{u}}}\subset \G/\hat{Q}$
is Levi movable for the embedding $f_2:G/{Q}\to \G/\hat{Q}$ and
${\hat{\Phi}}^{\hat{P}}_{{\hat{v}}}\subset \hat{Q}/\hat{P}$ is Levi movable for the embedding
$f_1:{Q}/{P}\to \hat{Q}/\hat{P}$.

In particular, the assumption (and hence the conclusion) of Theorem \ref{rrthm7}
is automatically satisfied under the above assumptions.
\end{theorem}

\begin{proof}
Let $C$ be the group $x Z({\hat{L}}_{\hat{Q}})x^{-1}\cap {H}$, where $Z({\hat{L}}_{\hat{Q}})$
is the center of the Levi subgroup ${\hat{L}}_{\hat{Q}}$ of $\hat{Q}$. Observe that
$C\subset Z(L_{{Q}})$.

Take ${\hat{l}}\in {\hat{L}}_{\hat{P}}$
such that the  map $\varphi$ in the following big diagram is an isomorphism,
which is possible since
${\hat{\Phi}}^{\hat{P}}_{{\hat{w}}}$ is
$L$-movable for the embedding $f$. Define an action of $C$ on $\G/\hat{P}$ via
 $t\odot y=(x^{-1}tx) y$, for $y\in \G/\hat{P}$ and $t\in C$.
 The left multiplication map $\G/\hat{P}\to
\G/\hat{P}$, $y\mapsto {\hat{l}}y$ commutes with
the action of $C$ and hence we have a $C$-module isomorphism
$$
\dfrac{T_{{\hat{e}}}(\G/\hat{P})}{T_{{\hat{e}}}({\hat{\Phi}}^{\hat{P}}_{{\hat{w}}})}
\xrightarrow{\sim}\dfrac{T_{{\hat{e}}}(\G/\hat{P})}{T_{{\hat{e}}}({\hat{l}}{\hat{\Phi}}^{\hat{P}}_{{\hat{w}}})}.
$$
We have the following commutative diagram, where the maps $\varphi_1, \varphi,
 \varphi_2$ are  induced by the embeddings $f_1, f, f_2$ respectively. (By [K$_1$,
 Lemma 1.3.14], it is easy to see that ${\hat{\Phi}}^{\hat{P}}_{{\hat{v}}}\subset
 {\hat{\Phi}}^{\hat{P}}_{{\hat{w}}}$ and ${\hat{\Phi}}^{\hat{P}}_{{\hat{w}}}$ maps to
 ${{\hat{v}}}^{-1}{\hat{\Phi}}^{\hat{Q}}_{{\hat{u}}}$
 under the projection $\G/\hat{P}\to \G/\hat{Q}$.)
Moreover,  in the following diagram, all the modules are $C$-modules and all the
maps are $C$-module
maps, where the action of $C$ on the left vertical side of the diagram is induced from
the action of $C$ on $G/P$ via the left multiplication.
\[
\xymatrix{
T_{e}({Q}/{P})\ar@{^{(}->}[d]\ar@{^{(}->}[r]^{\varphi_{1}} &
\dfrac{T_{{\hat{e}}}(\hat{Q}/\hat{P})}{T_{{\hat{e}}}({\hat{l}}{\hat{\Phi}}^{\hat{P}}_{{\hat{v}}})}\ar[d]\ar@{}[r]|{\simeq} &
\dfrac{T_{{\hat{e}}}(\hat{Q}/\hat{P})}{T_{{\hat{e}}}({\hat{\Phi}}^{\hat{P}}_{{\hat{v}}})}\\
T_{e}(G/{P})\ar[d]\ar[r]^{\displaystyle{\mathop{\sim}^{\varphi}}}
& \dfrac{T_{{\hat{e}}}(\G/\hat{P})}{T_{{\hat{e}}}({\hat{l}}{\hat{\Phi}}^{\hat{P}}_{{\hat{w}}})}\ar@{->>}[d]\ar@{}[r]|{\simeq}
& \dfrac{T_{{\hat{e}}}(\G/\hat{P})}{T_{{\hat{e}}}({\hat{\Phi}}^{\hat{P}}_{{\hat{w}}})}\\
T_{e}(G/{Q})\ar@{->>}[r]^{\varphi_{2}} &
\dfrac{T_{{\hat{e}}}(\G/\hat{Q})}{T_{{\hat{e}}}({\hat{l}}{{\hat{v}}}^{-1}{\hat{\Phi}}^{\hat{Q}}_{{\hat{u}}})}\ar@{}[r]|{\simeq}
& \dfrac{T_{{\hat{e}}}(\G/\hat{Q})}{T_{{\hat{e}}}({\hat{\Phi}}^{\hat{Q}}_{{\hat{u}}})}
}
\]
By the identity \eqref{eqn5new},
the sum of the set of roots in $T_{{\hat{e}}}(\G/\hat{P})=-\chi^{\hat{P}}_{1}$ and the sum
of the set of roots in
$\dfrac{T_{{\hat{e}}}(\G/\hat{P})}{T_{{\hat{e}}}({\hat{\Phi}}^{\hat{P}}_{{\hat{w}}})}=-\chi^{\hat{P}}_{{\hat{w}}}$. Thus, from the
isomorphism $\varphi$, we get:
\begin{equation}\label{rreq2}
-{\chi^P_1}_{|\mathfrak{c}}=-{\hat{\chi}}^{\hat{P}}_{{\hat{w}}_{| \mathfrak{c}}},
\end{equation}
where $\hat{\chi}^{\hat{P}}_{{\hat{w}}_{|{\mathfrak{c}}}}$    refers to the twisted action
$\odot$, i.e., $\hat{\chi}^{\hat{P}}_{{\hat{w}}}(y)=\chi^{\hat{P}}_{{\hat{w}}}(\Ad x^{-1}\cdot y)$,
for
$y\in \mathfrak{c}$, where $\mathfrak{c}:=$ Lie $C$.

Let $M$ be the kernel of $\varphi_{2}$ and let $\beta$ be the sum of roots
in $M$. Then, from the surjective map $\varphi_{2}$, we get
\begin{equation}\label{rreq3}
-\chi_{1}^{{Q}}-\beta=-\hat{\chi}^{\hat{Q}}_{{\hat{u}}}\quad\text{restricted
  to ~ }\mathfrak{c}.
\end{equation}
But, it is easy to see that
\begin{equation}\label{rreq4}
\chi^{{P}}_{1|\mathfrak{c}} = \chi^{{Q}}_{1|\mathfrak{c}},
\quad\text{and}\,\,\,
\hat{\chi}^{\hat{P}}_{{\hat{w}}|\mathfrak{c}}=\hat{\chi}^{\hat{Q}}_{{\hat{u}}|\mathfrak{c}}.
\end{equation}
Thus, combining the equations \eqref{rreq2} - \eqref{rreq4}, we get $\beta_{|\mathfrak{c}}
\equiv 0$. In particular, $\beta(y_o)=0$. But since $y_o$ is a dominant regular
element of $\mathfrak{z}(L_{{Q}})$
and $\beta$ is a positive sum of roots in
$R^{-}_{G}\backslash R^{-}_{L_{Q}}$, this is possible only if $\beta=0$, i.e., $M$ is zero
dimensional. This shows that $M=0$ and hence $\varphi_{2}$ is an
isomorphism.
Since $\varphi_{1}$ is injective, by dimensional consideration,
$\varphi_{1}$ is an isomorphism as well. This proves the theorem.
\end{proof}
\begin{remark} As in [ReR, Lemma 3.4], by virtue of Theorem \ref{rrthm2}, the
multiplicative formula for the decomposition of structure constants as in
Theorem \ref{rrthm7} applies to all the structure constants associated to the
 homomorphism $\theta_0^\delta$ of Theorem \ref{rr}.
 \end{remark}
 Specializing Theorems \ref{rrthm7} and \ref{rrthm2} to the diagonal case, we immediately
 get the following. This result was obtained by Richmond for type $A$ flag varieties
 in [Ri$_1$, Theorem 3], for type $C$ flag varieties in his Ph. D. thesis, and in general
 in [Ri$_2$].
 \begin{corollary} \label{12.14} Let $G,B,H$ be as in Section \ref{sec1}; in particular, $G$ is
 a connected semisimple group. Let $B\subset P \subset Q$ be (standard) parabolic
 subgroups. Let $\{w_j\}_{1\leq j \leq s} \subset W^P$ be such that
 $\sum_{j=1}^s\,\codim X_{w_j}^P
 = \dim G/P$. Assume further that
 \begin{equation} \label{rr-1}
 \sum_{j=1}^s\,\codim X_{u_j}^Q
 = \dim G/Q.
 \end{equation}
 (and hence $\sum_{j=1}^s\,\codim_{Q/P} X_{v_j}^P
 = \dim Q/P$), where $w_j=u_jv_j$ is the unique decomposition with $u_j\in
 {W}^{{Q}}$ and $v_j\in
W^P\cap {W}_{Q}$.
 Write
 \begin{align*}
[X^P_{w_1}]\dots  [X^P_{w_s}] &= d[pt] \in  H^*(G/P),\\
[X^Q_{u_1}]\dots  [X^Q_{u_s}] &= d_1[pt] \in H^*(G/Q),\,\,\,\text{and}\\
[X^P_{v_1}]\dots  [X^P_{v_s}] &= d_2[pt] \in  H^*(Q/P).
\end{align*}
Then, $d=d_1d_2$.

If $\{w_j\}$ is $L_P$-movable for $G/P$, then $\{u_j\}$ (resp. $\{v_j\}$)  is
 $L_Q$-movable for $G/Q$ (resp.  $L_P$-movable for $Q/P$). In particular,
 \eqref{rr-1} is automatically satisfied.
 \end{corollary}
\begin{remark}
 Knutson-Purbhoo [KP] have shown that for a $(d-1)$-step flag
variety  $G/P$ (for $G=\SL(n)$), any structure
constant for the deformed product $\odot_0$ is a product of $d(d-1)/2$
Littlewood-Richardson numbers. This refines the factorization into
$d-1$ factors as in Corollary \ref{12.14}.
\end{remark}

\section{Tables of the Deformed Product $\odot$ for the Groups of type $B_2$, $G_2$,
$B_3$ and $C_3$}
\label{exemples}

We give below  the multiplication tables under the deformed product
$\odot$  for $G/P$ for the  complex simple groups of type $B_2$, $G_2$,
$B_3$ and $C_3$ and maximal
parabolic subgroups $P$.  Since we are only considering maximal parabolics, we
have only one indeterminate, which we denote by ${\defpar}$.  We
let $r$,$s$ and $t$ be the simple reflections of any group of rank $3$
(and $r,s$ for the simple groups of rank $2$) associated to the
nodes from left to right of the Dynkin diagram following the
Bourbaki [Bo, Planche I - IX] convention (so $t$ corresponds to the long simple root in
the case of $C_3$ and the short simple root in the case of $B_3$).

These tables for rank $3$ groups are taken from [BK$_1$, $\S$ 10] and [KuLM, $\S$4].

{\bf Example 1. $G=B_2, P=P_1:$} In the following $b_1=\epsilon^P_r, b_2=
\epsilon^P_{sr},  b_3=\epsilon^P_{rsr}.$
\begin{center}
\begin{tabular} {|c|c|c|c|c|c|}
\hline

$H^*(G/P_1)$ & $b_1$ &  $b_2$ &  $b_3$   \\
\hline $b_1$ & $ 2b_2$ & $ b_3$ & $0$ \\ \hline
\end{tabular}
\end{center}

{\bf Example 2. $G=B_2, P=P_2:$} In the following $b_1=\epsilon^P_s, b_2=
\epsilon^P_{rs},  b_3=\epsilon^P_{srs}.$
\begin{center}
\begin{tabular} {|c|c|c|c|c|c|}
\hline

$H^*(G/P_2)$ & $b_1$ &  $b_2$ &  $b_3$   \\
\hline $b_1$ & $ {\defpar} b_2$ & $ b_3$ & $0$ \\ \hline
\end{tabular}
\end{center}

{\bf Example 3. $G=G_2, P=P_1:$} In the following $a_1=\epsilon^P_r, a_2=
\epsilon^P_{sr},  a_3=\epsilon^P_{rsr},  a_4=\epsilon^P_{srsr},  a_5=
\epsilon^P_{rsrsr}.$

\begin{center}
\begin{tabular} {|c|c|c|c|c|c|}
\hline

$H^*(G/P_1)$ & $a_1$ &  $a_2$ &  $a_3$ &  $a_4$ & $a_5$   \\
\hline $a_1$ & $ {\defpar}^2 a_2$ & $5{\defpar} a_3$ & ${\defpar}^2 a_4$ &  $a_5$
& $0$\\ \hline
$a_2$ & & $5 {\defpar}a_4$ & $a_5$ &  $0$ & $0$ \\ \hline

\end{tabular}
\end{center}

{\bf Example 4. $G=G_2, P=P_2:$} In the following $a_1=\epsilon^P_s, a_2=
\epsilon^P_{rs},  a_3=\epsilon^P_{srs},  a_4=\epsilon^P_{rsrs},  a_5=
\epsilon^P_{srsrs}.$

\begin{center}
\begin{tabular} {|c|c|c|c|c|c|}
\hline

$H^*(G/P_2)$ & $a_1$ &  $a_2$ &  $a_3$ &  $a_4$ & $a_5$   \\
\hline $a_1$ & $ 3 a_2$ & $2{\defpar} a_3$ & $3 a_4$ &  $a_5$
& $0$\\ \hline
$a_2$ & & $2 {\defpar}a_4$ & $a_5$ &  $0$ & $0$ \\ \hline

\end{tabular}
\end{center}

{\bf Example 5. $G=B_3, P=P_1:$} In the following $b_1=\epsilon^P_r, b_2=
\epsilon^P_{sr},  b_3=\epsilon^P_{tsr},  b_4=\epsilon^P_{stsr},  b_5=
\epsilon^P_{rstsr}.$
\begin{center}
\begin{tabular} {|c|c|c|c|c|c|}
\hline

$H^*(G/P_1)$ & $b_1$ &  $b_2$ &  $b_3$ &  $b_4$ & $b_5$   \\
\hline $b_1$ & $ b_2$ & $2 b_3$ & $b_4$ &  $b_5$ &  $0$ \\ \hline
$b_2$ & & $2b_4$ & $b_5$ &  $0$ & $0$ \\ \hline

\end{tabular}
\end{center}

{\bf Example 6. $G=B_3, P=P_2:$} In the following $b_1=\epsilon^P_s, b^{\prime}_2=
\epsilon^P_{rs},  b^{\prime \prime}_2=\epsilon^P_{ts},  b^{\prime}_3=\epsilon^P_{rts},  b^{\prime \prime}_3=
\epsilon^P_{sts}, b^{\prime}_4=\epsilon^P_{srts}, b^{\prime \prime}_4=\epsilon^P_{rsts},
b^{\prime}_5=\epsilon^P_{tsrts}, b^{\prime \prime}_5=\epsilon^P_{rstrs}, b_6=\epsilon^P_{rtsrts},
b_7=\epsilon^P_{srtsrts}.$
\begin{center}
\begin{tabular} {|c|c|c|c|c|c|}
\hline $H^*(G/P_2)$ & $b_1$ & $b_2^{\prime}$ & $b_2^{\prime
\prime}$ & $b_3^{\prime}$ & $b_3^{\prime \prime}$ \\ \hline

$b_1$ & $b_2^{\prime} + 2 b_2^{\prime \prime}$ & $2 b_3^{\prime}$
& $b_3^{\prime} + b_3^{\prime \prime}$ & $ 2{\defpar} b_4^{\prime} +{\defpar}
b_4^{\prime \prime}$ & $ {\defpar}b_4^{\prime} + 2{\defpar} b_4^{\prime \prime}$  \\
\hline

$b_2^{\prime}$ & & $2 {\defpar}b_4^{\prime}$ & ${\defpar}b_4^{\prime} + {\defpar}b_4^{\prime
\prime}$ & $2{\defpar} b_5^{\prime} +{\defpar} b_5^{\prime \prime}$
& $ {\defpar}b_5^{\prime \prime}$ \\
\hline

$b_2^{\prime \prime}$ & & & $ {\defpar}b_4^{\prime} + {\defpar}b_4^{\prime \prime}$
& ${\defpar}b_5^{\prime} + {\defpar}b_5^{\prime \prime}$ & ${\defpar} b_5^{\prime} + {\defpar}
b_5^{\prime \prime} $  \\ \hline

$b_3^{\prime }$ & & & & $ 2 {\defpar}b_6$ & ${\defpar}b_6$ \\ \hline

$b_3^{\prime \prime}$ & & & & & $ 2 {\defpar}b_6$ \\ \hline

\end{tabular}
\end{center}

\begin{center}
\begin{tabular} {|c|c|c|c|c|c|c|}
\hline $H^*(G/P_2)$ & $b_4^{\prime}$ & $b_4^{\prime \prime}$ &
$b_5^{\prime}$ & $b_5^{\prime \prime}$ & $b_6$ & $b_7$ \\ \hline

$b_1$ & $2 b_5^{\prime}+b_5^{\prime \prime}$ & $b_5^{\prime
\prime}$ & $b_6$ & $2 b_6$ & $b_7$ & $0$ \\ \hline

$b_2^{\prime}$ & $2 b_6$ & $0$ & $b_7$ & $0$ & $0$ & $0$ \\
\hline

$b_2^{\prime \prime}$ & $b_6$ & $b_6$ & $0$ & $b_7$ & $0$ & $0$ \\
\hline

$b_3^{\prime }$ & $b_7$ & $0$ & $0$ & $0$ & $0$ & $0$\\ \hline

$b_3^{\prime \prime}$ &$0$ & $b_7$ & $0$ & $0$ & $0$ & $0$
\\ \hline
\end{tabular}
\end{center}

{\bf Example 7. $G=B_3, P=P_3:$}
In the following $b_1=\epsilon^P_t, b_2=
\epsilon^P_{st},  b^{\prime}_3=\epsilon^P_{rst},  b^{\prime\prime}_3=\epsilon^P_{tst},
 b_4=
\epsilon^P_{trst}, b_5=\epsilon^P_{strst},  b_6=\epsilon^P_{tstrst}.$

\begin{center}
\begin{tabular} {|c|c|c|c|c|c|c|c|}
\hline

$H^*(G/P_3)$ & $b_1$ &  $b_2$ &  $b_3^{\prime}$ & $b_3^{\prime
\prime}$ & $b_4$ & $b_5$ &  $b_6$   \\ \hline $b_1$ & $ {\defpar}b_2$ & $
{\defpar}b_3^{\prime} + b_3^{\prime \prime}$ & $b_4$ & ${\defpar}b_4$ & ${\defpar}b_5$ &
$b_6$ & $0$ \\ \hline $b_2$ & & $2b_4$ & $b_5$ & ${\defpar}b_5$ & $b_6$ &
$0$ & $0$ \\ \hline $b_3^{\prime}$ & & & $0$ & $b_6$ & $0$ & $0$ &
$0$ \\ \hline $b_3^{\prime \prime}$ & & & & $0$ & $0$ & $0$ & $0$
\\ \hline

\end{tabular}
\end{center}

{\bf Example 8. $G=C_3, P=P_1:$} In the following $a_1=\epsilon^P_r, a_2=
\epsilon^P_{sr},  a_3=\epsilon^P_{tsr},  a_4=\epsilon^P_{stsr},  a_5=
\epsilon^P_{rstsr}.$

\begin{center}
\begin{tabular} {|c|c|c|c|c|c|}
\hline

$H^*(G/P_1)$ & $a_1$ &  $a_2$ &  $a_3$ &  $a_4$ & $a_5$   \\
\hline $a_1$ & $ a_2$ & ${\defpar} a_3$ & $a_4$ &  $a_5$ &  $0$ \\ \hline
$a_2$ & & ${\defpar}a_4$ & $a_5$ &  $0$ & $0$ \\ \hline

\end{tabular}
\end{center}
{\bf Example 9. $G=C_3, P=P_2:$} In the following $a_1=\epsilon^P_s, a^{\prime}_2=
\epsilon^P_{rs},  a^{\prime \prime}_2=\epsilon^P_{ts},  a^{\prime}_3=\epsilon^P_{rts},
 a^{\prime \prime}_3=
\epsilon^P_{sts}, a^{\prime}_4=\epsilon^P_{srts}, a^{\prime \prime}_4=\epsilon^P_{rsts},
a^{\prime}_5=\epsilon^P_{tsrts}, a^{\prime \prime}_5=\epsilon^P_{rstrs},
a_6=\epsilon^P_{rtsrts},
a_7=\epsilon^P_{srtsrts}.$
\begin{center}
\begin{tabular} {|c|c|c|c|c|c|}
\hline $H^*(G/P_2)$ & $a_1$ & $a_2^{\prime}$ & $a_2^{\prime
\prime}$ & $a_3^{\prime}$ & $a_3^{\prime \prime}$  \\ \hline

$a_1$ & $a_2^{\prime} + {\defpar}a_2^{\prime \prime}$ & ${\defpar}a_3^{\prime}$ &
$a_3^{\prime} + a_3^{\prime \prime}$ & $ 2{\defpar} a_4^{\prime} +{\defpar}
a_4^{\prime \prime}$ & $ {\defpar}a_4^{\prime} + 2{\defpar} a_4^{\prime \prime}$ \\
\hline

$a_2^{\prime}$ & & ${\defpar}^2a_4^{\prime}$ & ${\defpar}a_4^{\prime} + {\defpar}a_4^{\prime
\prime}$ & $ {\defpar}^2a_5^{\prime} + {\defpar}a_5^{\prime \prime}$ & $ {\defpar}a_5^{\prime
\prime}$  \\ \hline

$a_2^{\prime \prime}$ & & & $ 2 a_4^{\prime} + 2 a_4^{\prime
\prime}$ & ${\defpar}a_5^{\prime} + 2 a_5^{\prime \prime}$ & $ {\defpar}a_5^{\prime}
+ 2 a_5^{\prime \prime} $  \\ \hline

$a_3^{\prime }$ & & & & $ 2 {\defpar}a_6$ & ${\defpar}a_6$ \\ \hline

$a_3^{\prime \prime}$ & & & & & $ 2{\defpar} a_6$ \\ \hline

\end{tabular}
\end{center}

\begin{center}
\begin{tabular} {|c|c|c|c|c|c|c|}
\hline $H^*(G/P_2)$ & $a_4^{\prime}$ & $a_4^{\prime \prime}$ &
$a_5^{\prime}$ & $a_5^{\prime \prime}$ & $a_6$ & $a_7$ \\ \hline

$a_1$ & ${\defpar}a_5^{\prime}+a_5^{\prime \prime}$ & $a_5^{\prime \prime}$
& $a_6$ & ${\defpar}a_6$ & $a_7$ & $0$ \\ \hline

$a_2^{\prime}$ & ${\defpar}a_6$ & $0$ & $a_7$ & $0$ & $0$ & $0$ \\ \hline

$a_2^{\prime \prime}$ & $a_6$ & $a_6$ & $0$ & $a_7$ & $0$ & $0$ \\
\hline

$a_3^{\prime }$ & $a_7$ & $0$ & $0$ & $0$ & $0$ & $0$\\ \hline

$a_3^{\prime \prime}$ &  $0$ & $a_7$ & $0$ & $0$ & $0$ & $0$
\\ \hline
\end{tabular}
\end{center}

{\bf Example 10. $G=C_3, P=P_3:$}
In the following $a_1=\epsilon^P_t, a_2=
\epsilon^P_{st},  a^{\prime}_3=\epsilon^P_{rst},  a^{\prime\prime}_3=\epsilon^P_{tst},
 a_4=
\epsilon^P_{trst}, a_5=\epsilon^P_{strst},  a_6=\epsilon^P_{tstrst}.$
\begin{center}
\begin{tabular} {|c|c|c|c|c|c|c|c|}
\hline

$H^*(G/P_3)$ & $a_1$ &  $a_2$ &  $a_3^{\prime}$ & $a_3^{\prime
\prime}$ & $a_4$ & $a_5$ &  $a_6$   \\ \hline $a_1$ & $2 a_2$ & $2
a_3^{\prime} + a_3^{\prime \prime}$ & $a_4$ & $2a_4$ & $2a_5$ &
$a_6$ & $0$ \\ \hline $a_2$ & & $2a_4$ & $a_5$ & $2a_5$ & $a_6$ &
$0$ & $0$ \\ \hline $a_3^{\prime}$ & & & $0$ & $a_6$ & $0$ & $0$ &
$0$ \\ \hline $a_3^{\prime \prime}$ & & & & $0$ & $0$ & $0$ & $0$
\\ \hline

\end{tabular}
\end{center}
\begin{remark} (a) The deformed product $\odot_0$ for $\SO(8)/P$ for all the maximal
parabolic subgroups $P$ of $\SO(8)$ is determined in [KKM].

(b) The deformed product $\odot_0$ for $F_4/P$ and $E_6/P$ for all the maximal
parabolic subgroups $P$ is determined by B. Lee (cf. [Le])
\end{remark}

\section{An explicit determination of the eigencone  for the ranks 2 and 3 simple Lie
algebras}\label{sec10}

The results in this section are taken from [KLM$_1$, $\S$ 7] for the rank $2$ root systems
and [KuLM]  for the rank $3$ root systems.

In this section we describe the irredundant set of inequalities, given by Corollary
\ref{7.5}, determining the
eigencone $\bar{\Gamma}_3$ inside $\frh_+^3$ for the ranks $2$ and $3$ root
systems $A_2$, $B_2$, $G_2$, $A_3$, $B_3$ and $C_3$. Thus the following inequalities correspond to the
facets of  $\bar{\Gamma}_3$ which intersect the interior of  $\frh_+^3$.
In each of the three rank $2$ (resp. rank $3$) cases, there are $2$ (resp. $3$) standard maximal
parabolics, hence the system breaks up into two (resp. three) subsystems.

\subsection{The inequalities for $A_2$}
 The Weyl
chamber $\frh_+$ is given by
$$\frh_+ = \{(x,y,z) : x + y + z  = 0, x\geq y \geq z\}.$$

  We give below the inequalities  in terms of the triples
$(v_1,v_2,v_3) \in \frh_+^3$ with $v_j= (x_j,y_j,z_j), j
=1,2,3$.  We only give a set of inequalities up to the action of $S_3$. Thus,
to get the full set of inequalities,  we need to
symmetrize these with respect to the action of $S_3$ diagonally
permuting the variables $x_1, x_2,x_3; y_1, y_2,y_3; z_1, z_2,z_3$.
\subsubsection{The subsystem associated to $H^*(G/P_1)$ (before symmetrization)}
\begin{align*}
x_1 + z_2 + z_3 \leq 0 \\
y_1 + y_2 + z_3 \leq 0 .
\end{align*}
 These constitute $6$ inequalities after symmetrization.
\subsubsection{The subsystem associated to $H^*(G/P_2)$ (before symmetrization)}
\begin{align*}
z_1 + x_2 + x_3 \geq 0 \\
y_1 + y_2 + x_3 \geq 0 .
\end{align*}
 These constitute $6$ inequalities after symmetrization.

To summarize,  for $A_2$, these provide an irredundant set of  altogether
 $12$ inequalities.
 \subsection{The inequalities for $B_2$}
The Weyl chamber $\frh_+$ is given by the pairs $(x,y)$
of real numbers satisfying
$x\geq y  \geq 0.$

The inequalities will now be in terms of $(v_1,v_2,v_3) \in
\frh_+^3$ with $v_j = (x_j,y_j),j=1,2,3$. We will need to symmetrize the
inequalities
 with respect to the action of $S_3$ diagonally permuting the variables
$x_1, x_2,x_3; y_1, y_2,y_3$.
\subsubsection{The subsystem associated to $H^*(G/P_1)$ (before symmetrization)}
\begin{align*}
 x_1 \leq x_2 + x_3  \\
 y_1 \leq y_2 + x_3 .
 \end{align*}
After symmetrizing, we get $9$ inequalities.
\subsubsection{The subsystem associated to $H^*(G/P_2)$ (before symmetrization)}
\begin{align*}
 x_1 +y_1\leq x_2 + y_2+ x_3 +y_3 \\
 x_1+y_2 \leq y_1+ x_2+ x_3 + y_3 .
 \end{align*}
After symmetrizing, we get $9$ inequalities.

To summarize,  for $B_2$, these provide an irredundant set of  altogether
 $18$ inequalities.
 \subsection{The inequalities for $G_2$}
The Weyl chamber $\frh_+$ is given by the pairs $(x,y)$
of real numbers satisfying
$x\geq 0, y  \geq 0.$

The inequalities will now be in terms of $(v_1,v_2,v_3) \in
\frh_+^3$ with $v_j = (x_j,y_j),j=1,2,3$. We will need to symmetrize the
inequalities
 with respect to the action of $S_3$ diagonally permuting the variables
$x_1, x_2,x_3; y_1, y_2,y_3$.
\subsubsection{The subsystem associated to $H^*(G/P_1)$ (before symmetrization)}
\begin{align*}
 2x_1 +y_1\leq 2x_2 +y_2+ 2x_3+y_3  \\
 x_1+y_1 \leq x_2+y_2 + 2x_3 +y_3\\
 x_1\leq x_2 +2x_3 +y_3 .
 \end{align*}
After symmetrizing, we get $15$ inequalities.
\subsubsection{The subsystem associated to $H^*(G/P_2)$ (before symmetrization)}
\begin{align*}
 3x_1 +2y_1\leq 3x_2 + 2y_2+ 3x_3 +2y_3 \\
 3x_1+y_1 \leq 3x_2 + y_2+ 3x_3 +2y_3\\
 y_1\leq y_2+ 3x_3 +2y_3 .
 \end{align*}
After symmetrizing, we get $15$ inequalities.

To summarize,  for $G_2$, these provide an irredundant set of  altogether
 $30$ inequalities.
\subsection{The inequalities for $A_3$}
 The Weyl
chamber $\frh_+$ is given by
$$\frh_+ = \{(x,y,z,w) : x + y + z + w = 0, x\geq y \geq z \geq w \}.$$

  We give below the inequalities  in terms of the triples
$(v_1,v_2,v_3) \in \frh_+^3$ with $v_j= (x_j,y_j,z_j,w_j), j
=1,2,3$.  We only give a set of inequalities up to the action of $S_3$. Thus,
to get the full set of inequalities,  we need to
symmetrize these with respect to the action of $S_3$ diagonally
permuting the variables $x_1, x_2,x_3; y_1, y_2,y_3; z_1, z_2,z_3;
w_1, w_2,w_3$.
\subsubsection{The subsystem associated to $H^*(G/P_1)$ (before symmetrization)}
\begin{align*}
x_1 + w_2 + w_3 \leq 0 \\
y_1 + z_2 + w_3 \leq 0 \\
z_1 + z_2 + z_3 \leq 0 .
\end{align*}
 These constitute $10$ inequalities after symmetrization.
\subsubsection{The subsystem associated to $H^*(G/P_2)$ (before symmetrization)}
\begin{align*}
x_1 + y_1 + z_2 + w_2 + z_3 + w_3 \leq 0 \\
x_1 + z_1 + y_2 + w_2 + z_3 + w_3 \leq 0 \\
x_1 + w_1 + x_2 + w_2 + z_3 + w_3 \leq 0 \\
y_1 + z_1 + y_2 + z_2 + z_3 + w_3 \leq 0 \\
x_1 + w_1 + y_2 + w_2 + y_3 + w_3 \leq 0 \\
y_1 + z_1 + y_2 + w_2 + y_3 + w_3 \leq 0 .
\end{align*}
These constitute $21$ inequalities after symmetrization.
\subsubsection{The subsystem associated to $H^*(G/P_3)$
(before
symmetrization)}
\begin{align*}
x_1 + y_1 + z_1 + y_2 + z_2 + w_2 + y_3 + z_3 + w_3 \leq 0 \\
x_1 + y_1 + w_1 + x_2 + z_2 + w_2 + y_3 + z_3 + w_3 \leq 0 \\
x_1 + z_1 + w_1 + x_2 + z_2 + w_2 + x_3 + z_3 + w_3 \leq 0 .
\end{align*}
These constitute $10$ inequalities after
symmetrization.

These $41$ inequalities form an irredundant set to define
$\bar{\Gamma}_3(A_3)$ inside $\mathfrak{h}_+^3$.
\subsection{The inequalities for $B_3$}
The Weyl chamber $\frh_+$ is given by the triples $(x,y,z)$
of real numbers satisfying
$x\geq y \geq z \geq 0.$

The inequalities will now be in terms of $(v_1,v_2,v_3) \in
\frh_+^3$ with $v_j = (x_j,y_j,z_j),j=1,2,3$. We will need to symmetrize the
inequalities
 with respect to the action of $S_3$ diagonally permuting the variables
$x_1, x_2,x_3; y_1, y_2,y_3; z_1, z_2,z_3$.
\subsubsection{The subsystem associated to $H^*(G/P_1)$ (before symmetrization)}
\begin{align*}
 x_1 \leq x_2 + x_3  \\
 y_1 \leq y_2 + x_3 \\
 z_1 \leq z_2 + x_3  \\
 z_1 \leq y_2 + y_3 .
 \end{align*}
After symmetrizing, we get $18$ inequalities.
\subsubsection{The subsystem associated to $H^*(G/P_2)$ (before symmetrization)}
\begin{align*}
x_1 + y_1 & \leq x_2 + y_2 + x_3 + y_3 \\
x_1 + z_1 & \leq x_2 + z_2 + x_3 + y_3  \\
y_1 + z_1 & \leq y_2 + z_2 + x_3 + y_3  \\
x_1 + z_2 & \leq z_1 + x_2 + x_3 + y_3 \\
y_1 + z_2 & \leq z_1 + y_2 + x_3 + y_3\\
x_1 + y_2 & \leq y_1 + x_2 + x_3 + y_3 \\
y_1 + z_1 & \leq x_2 + z_2 + x_3 + z_3 \\
y_1 + z_2 & \leq z_1 + x_2 + x_3 + z_3  \\
x_1 + z_2 & \leq y_1 + x_2 + x_3 + z_3 .
\end{align*}
 We get $48$ inequalities after symmetrizing.
\subsubsection{The subsystem associated to $H^*(G/P_3)$ (before symmetrization)}
\begin{align*}
x_1 + y_1 + z_1 & \leq x_2 + y_2 + z_2 + x_3 + y_3 + z_3 \\
x_1 + y_1 + z_2 & \leq z_1 + x_2 + y_2 + x_3 + y_3 + z_3 \\
x_1 + z_1 + y_2 & \leq y_1 + x_2 + z_2 + x_3 + y_3 + z_3 \\
x_1 + y_2 + z_2 & \leq y_1 + z_1 + x_2 + x_3 + y_3 + z_3 \\
x_1 + y_2 + z_3 & \leq y_1 + z_1 + x_2 + z_2 + x_3 + y_3 .
\end{align*}
After symmetrizing we get $27$ inequalities.

To summarize,  for $B_3$, these provide an irredundant set of  altogether
 $93$ inequalities.
\subsection{The inequalities for $C_3$}

In this case the Weyl chamber $\frh_+$ is given by the triples
$(x,y,z)$ of real numbers satisfying
$x\geq y \geq z \geq 0.$
Here  $x,y,z$ are the coordinates relative to the standard basis
$\epsilon_1,\epsilon_2,\epsilon_3$ in the notation of
[Bo, pg. 254 - 255]. The inequalities will now be in
terms of $(v_1,v_2,v_3) \in \frh_+^3$ with $v_j =
(x_j,y_j,z_j),j=1,2,3$. We will need to symmetrize the
inequalities
 with respect to the action of $S_3$ diagonally permuting the variables
$x_1, x_2,x_3; y_1, y_2,y_3; z_1, z_2,z_3$.
\subsubsection{The subsystem associated to $H^*(G/P_1)$ (before symmetrization)}
\begin{align*}
 x_1 \leq x_2 + x_3  \\
 y_1 \leq y_2 + x_3  \\
 z_1 \leq z_2 + x_3 \\
 z_1 \leq y_2 + y_3 .
 \end{align*}
These give $18$ inequalities after
symmetrization.
\subsubsection{The subsystem associated to $H^*(G/P_2)$ (before symmetrization)}
\begin{align*}
x_1 + y_1 & \leq x_2 + y_2 + x_3 + y_3  \\
x_1 + z_1 & \leq x_2 + z_2 + x_3 + y_3  \\
y_1 + z_1 & \leq y_2 + z_2 + x_3 + y_3  \\
x_1 + z_2 & \leq z_1 + x_2 + x_3 + y_3 \\
y_1 + z_2 & \leq z_1 + y_2 + x_3 + y_3 \\
x_1 + y_2 & \leq y_1 + x_2 + x_3 + y_3  \\
y_1 + z_1 & \leq x_2 + z_2 + x_3 + z_3 \\
y_1 + z_2 & \leq z_1 + x_2 + x_3 + z_3  \\
x_1 + z_2 & \leq y_1 + x_2 + x_3 + z_3 .
\end{align*}
This subsystem  after
symmetrization consists of $48$ inequalities.
\subsubsection{The subsystem associated to $H^*(G/P_3)$ (before symmetrization)}
\begin{align*}
x_1 + y_1 + z_1 \leq x_2 + y_2 + z_2 + x_3 + y_3 + z_3 \\
x_1 + y_1 + z_2 \leq z_1 + x_2 + y_2 + x_3 + y_3 + z_3  \\
x_1 + z_1 + y_2 \leq y_1 + x_2 + z_2 + x_3 + y_3 + z_3  \\
x_1 + y_2 + z_2 \leq y_1 + z_1 + x_2 + x_3 + y_3 + z_3 \\
x_1 + y_2 + z_3 \leq y_1 + z_1 + x_2 + z_2 + x_3 + y_3 .
\end{align*}
This gives   $27$ inequalities after symmetrization.

 The $27$ inequalities above can be rewritten  in a very simple way.
 Let $S = \sum_{j=1}^3 x_j + y_j +z_j$. Then the $27$ inequalities
 are just the inequalities
 $$x_i + y_j + z_k \leq \frac{S}{2}, \  i,j,k = 1,2,3.$$

These $93$
inequalities form an irredundant set to define
$\bar{\Gamma}_3(C_3)\subset \frh_+^3$.

\begin{remark} The irredundant set of inequalities to define
$\bar{\Gamma}_3(D_4)\subset \frh_+^3$ is explicitly determined in [KKM, $\S$ 5].
\end{remark}

\newpage

\appendix

\section{Buildings and tensor product multiplicities\\ (by Michael Kapovich$^0$)}
\footnotetext{Partial  financial support for this work was provided by the NSF grants DMS-09-05802 and DMS-12-05312.}

The goal of this appendix is to explain connections between {\em metric geometry} (driven by notions such as {\em distance} and {\em curvature}) and the representation theory of complex semisimple Lie groups. The connections run through the theory of buildings. We will give sketches of proofs of results established in a sequence of papers 
\cite{KapovichLeebMillson1, KapovichLeebMillson2, KapovichLeebMillson3, KapovichMillson1, KapovichMillson2}  of the author and his collaborators: B. Leeb and J. J. Millson. (The results were further extended in the papers \cite{KKM}, \cite{HKM} and \cite{Berenstein-Kapovich}.)  We also refer the reader to the survey \cite{K2006} for a different take on these results  and the discussion of symmetric spaces and eigenvalue problems which we did not discuss here. Some of this theory should generalize in the context of Kac--Moody groups; we refer the reader to \cite{GR} for the first step in this direction. We refer the reader to the papers \cite{Fontaine-Kamnitzer, FKK, Kam, MV} for other developments connecting algebraic geometry of buildings and representation theory. 

\subsection{Notation}

Throughout, we let $F$ be a local field with discrete valuation and $\O\subset \K$ be the corresponding ring of integers; 
the reader can think of $F=\Q_p, \O=\Z_p$. Let $q$ denote the cardinality of the residue field of $F$ and let $\pi\in F$ be a uniformizer. 
Let ${\sf G}$ be a split semisimple algebraic group-scheme over $\Z$, let ${\sf G}^\vee$ be the Langlands' dual group scheme and set
$$
G={\sf G}(F), \quad G^\vee:= {\sf G}^\vee(\C).
$$

We will also fix a dual maximal tori $T\subset G, T^\vee \subset G^\vee$ and Borel subgroups $B, B^\vee$ normalized by these tori. These choices will allow us to talk about (dominant) weights of the group $G^\vee$ (more precisely, weights of $T^\vee$ positive with respect to $B^\vee$) etc. Let ${\sf U}\subset {\sf G}$ be the unipotent radical, set $U:={\sf U}(F)$.  

We let $\X_*(\sf{T})$,  $\X^*(\sf{T})$ denote the groups of cocharacters and characters of ${\sf T}$. The subgroup $K={\sf G}(\O)$ is a maximal compact subgroup of $G$. Lastly, let $W$ be the Weyl group of ${\sf G}$ corresponding to ${\sf T}$. 

The general theme of this appendix is that the representation theory of the group $G^\vee$ is governed by the geometry 
of the group $G$: This geometry will manifest itself through geometry (both metric and algebraic) 
of the Bruhat--Tits buildings associated with the group $G$.

Given the root system $R$ of the group $G$ (of rank $\ell$) we define the constant $k_R$ to the least common multiple of the numbers $a_1,\ldots, a_\ell$, where
$$
\theta=\sum_{i=1}^\ell a_i \al_i,
$$
with simple roots $\al_i$ of $R$ and the highest root $\theta$ of $R$. For instance, if $R$ is an irreducible root system, then $R$ has type $A$ if and only if $k_R=1$ and the largest value $k_R=60$ occurs for $R=E_8$.

\subsection{Buildings}\label{sec:buildings}

In this section we discuss axioms of (discrete) Euclidean (affine) buildings.  

\subsubsection{Coxeter complexes}

Let $A$ be the Euclidean $\ell$-dimensional space and let $W_{af}$ be a Coxeter group acting properly discontinuously, isometrically and faithfully on $A$, 
so that Coxeter generators of $W_{af}$ act as reflections on $A$. Note that the stabilizer $W_x$ of a point $x\in A$ in the group $W_{af}$ is a finite reflection group. We will assume that the group $W_{af}$ acts cocompactly on $A$, i.e., that $W_{af}$ is crystallographic. The pair $(A,W_{af})$ is a {\em Euclidean Coxeter complex}; the space $A$ is called {\em model apartment}. Fixed-point sets of reflections in $W_{af}$ are called {\em walls} in $A$. Let ${\mathcal W}$ denote the union of walls in $A$. Closures $a$ of components of the complement 
$$
A\setminus {\mathcal W}, 
$$
are {\em alcoves} in $A$; they are fundamental domains for the action of $W_{af}$ on $A$. 

A {\em half-apartment} in $A$ is a half-space bounded by a wall. The group $W_{af}$ splits as the semidirect product $\La \rtimes W$, where $\La$ is a group of translations in $\R^\ell$. Since $W_{af}$ was assumed to be crystallographic, it is associated to a root system $R$. Then, $\La$ is a lattice in $A$, the 
coroot lattice $Q(R^\vee)$; the finite reflection group $W$ is the stabilizer of a point $o\in A$. The normalizer of $W_{af}$ in the full group of translations of $A$ is the coweight lattice $P(R^\vee)$. 

We will fix a fundamental domain $\Delta$ (a {\em positive Weyl chamber}) for the action of $W$ on $A$, so that $\Delta$ is bounded by walls. We let $\Delta^*\subset \R^\ell$ denote the dual cone of $\Delta$:
$$
\Delta^*=\{v\in \R^\ell: \<v, u\>\ge 0, \forall u\in \Delta\}. 
$$
The cone $\Delta^*$ is spanned by the positive root vectors of $R$. 

 We define a partial order 
$$
u\le_{\Delta^*} v
$$
on $\R^\ell$ by requiring that
$$
v-u\in \Delta^*. 
$$

The Coxeter complex has natural structure of a regular cell complex, where facets are  alcoves and vertices are zero-dimensional intersections of walls. 
By abusing the terminology, we will refer to this cell complex as a Coxeter complex as well.

A vertex of the model Euclidean apartment $A$ is called {\em special} if its stabilizer in $W_{af}$ is isomorphic to $W$, i.e., is maximal possible. The root system $R$ is the product of root systems of type $A$, if and only if every vertex is special. The numbers $k_R$ are defined so that $k_R$ is the least natural number $n$ so that the image of  every vertex $v\in A$ under the scaling $x\mapsto nx$, is a special vertex of $A$. For instance, for the root system of type $A$, every vertex is special, so $k_R=1$.  

\subsubsection{Affine buildings}

A {\em space modeled on a Coxeter complex} $(A,W_{af})$ is a metric space $X$ together with a collection of isometric embeddings (``charts'') $\phi: A\to X$, so that transition maps between charts are restrictions of elements of $W_{af}$. The number $\ell$ is the {\em rank} of $X$. 
Images of charts are called {\em apartments} in $X$. Note that apartments are (almost never) open in $X$. By taking images of vertices, walls and half-apartments in $A$ under charts, we define vertices, walls and half-apartments in $A$. An isometry $g: X\to X$ is an {\em automorphism} of the building $X$ if for every pair of charts $\phi, \psi$, the composition $\psi^{-1}\circ g\circ \phi$ is the restriction of an element of $W_{af}$. Our definition of affine buildings follows \cite{Kleiner-Leeb}; equivalence of this definition to the more combinatorial one (which could be found e.g. in \cite{Ronan}) was established in \cite{Parreau}. Note that definition that we give below extends (by adding an extra axiom) to the case of non-discrete buildings, see \cite{Kleiner-Leeb} and \cite{Parreau}.

\begin{defn}
A (thick) {\em Euclidean (affine) building}  is a space modeled on a Euclidean Coxeter complex and satisfying three axioms listed below: 
\end{defn}

\medskip
{\bf A1 (Connectivity).} Every pair of points in $X$ belongs to a common apartment. 

\medskip
{\bf A2 (Curvature bounds).} We require $X$ to be a $CAT(0)$-metric space. (We will explain the definitions below.)

\medskip 
{\bf A3 (Thickness).} Every wall in $X$ is the intersection of (at least) three half-apartments. 

\medskip
This definition parallels the one of the symmetric space $G/K$ of a (connected) semisimple Lie group. The $CAT(0)$ condition is the analogue of the fact that symmetric spaces of noncompact type have sectional curvature $\le 0$.

The $CAT(0)$ condition was  first introduced by A.D.Alexandrov in 1950s. Informally, this condition means that {\em geodesic triangles in $X$ are thinner than the geodesic triangles in the Euclidean plane}. Below is the precise definition. 

A {\em geodesic segment} ${xy}$ in $X$ is an isometric (i.e., distance-preserving) embedding of an interval $[a,b]\subset \R$ into $X$; the points $x$, $y$ are the images of $a, b$ under this isometric embedding. We will orient the geodesic segment $xy$ from $x$ to $y$. Similarly, one defines geodesic rays as isometric maps $[0,\infty)\to X$. 

An (oriented) {\em geodesic triangle} $\tau=xyz$ in $X$ is a concatenation of three oriented geodesic segments ${xy}, {yz}, {zx}$, the {\em edges} of $\tau$; the points $x, y, z$ are the {\em vertices} of $\tau$. A {\em disoriented} geodesic triangle is formed by the segments $xy, yz, xz$. 
The side-lengths of a triangle $\tau$ are the lengths of its edges; they are denoted $|xy|, |yz|, |zx|$. Then the side-lengths of $\tau$ satisfy the triangle inequalities
$$
|xy|\le |yz|+|zx|. 
$$
The triangle inequalities above are necessary and sufficient for existence of a geodesic triangle in $\R^2$ with the given side-lengths. Then, in $\R^2$ there exists a {\em comparison triangle} for the triangle $\tau\subset X$, namely, a geodesic triangle $\tilde{\tau}$ with vertices $\tilde{x}, \tilde{y}, \tilde{z}$ whose side-lengths are the same as the side-lengths of $\tau$. Given any pair of points $p\in {xy}, q\in {yz}$, one defines the corresponding {\em comparison points} $\tilde{p}\in {\tilde{x} \tilde{y}}$ and  
$\tilde{q}\in {\tilde{y} \tilde{z}}$, so that
$$
|xp|= |\tilde{x} \tilde{p}|, \quad |yq|= |\tilde{y} \tilde{q}|. 
$$
Then the space $X$ is said to be $CAT(0)$ if 
for every geodesic triangle in $X$ we have:
$$
|pq|\le |\tilde{p} \tilde{q}|. 
$$
We refer the reader to \cite{Ballmann} for further details on $CAT(0)$ geometry.

\begin{ex}
Suppose that $X$ is a 1-dimensional Euclidean building. Then $X$ is a connected graph, whose vertices are the images of the walls in $A$ and all edges have the same length, which is the minimal distance between the walls in $A$. Thickness axiom is equivalent to the requirement that every vertex of $X$ has valence $\ge 3$. The curvature restriction is that the graph $X$ contains no circuits, i.e., $X$ is a tree.  
\end{ex}

Note that $X$ has natural structure of a cell complex, where 
cells are images of cells in $(A, W_{af})$. However, it is important to note that we consider {\em all} points of $X$, not just its vertices. 

\subsubsection{Chamber-valued distance function} 

Let $X$ be a Euclidean building. Our next goal is to define a 2-point invariant 
$d_\Delta(x,y)$ in $X$, taking values in the Weyl chamber $\Delta$. We first define 
$d_\Delta(x,y)$ for $x, y\in A$. We identify the affine space $A$ with the vector space $\R^\ell$ by declaring $o$ to the origin. Next, identify the directed segment $\ov{xy}$  with a vector $v$ in $\R^\ell$, then project $v$ to a vector $\bar{v}\in \Delta$ using the action of the group $W\subset W_{af}$. We declare $\bar{v}\in \Delta$ to be the {\em $\Delta$-valued distance} $d_\Delta(x,y)$ between the points $x, y\in A$. It is clear that $d_\Delta$ is $W_{af}$-invariant. 
Now, for a chart $\phi: A\to X$ we set 
$$
d_\Delta(\phi(x),\phi(y)):=d_\Delta(x,y). 
$$
Since every two points in $X$ belong to a common apartment and transition maps between charts in $X$ are restrictions of elements of $W_{af}$, it follows that we obtain a well-defined function $d_\Delta: X\times X\to \Delta$. Furthermore, if $g$ is an automorphism of $X$, then $g$ preserves $d_\Delta$.  

Note that, in general, the function $d_\Delta$ is not symmetric, however, 
$$d_\Delta(x,y)= -w_0 d_\Delta(y,x),$$ where 
$w_0\in W$ is the longest element of this finite Coxeter group. For an oriented geodesic segment ${xy}$ in $X$, we regard $d_\Delta(x,y)$ as the $\De$-valued lengths of ${xy}$. More generally, given a piecewise-geodesic path $p$ in $X$ (i.e., a concatenation of geodesic paths $p_i, i=1,\ldots,m$), 
we define the $\De$-length of $p$, denoted $length_\De(p)$ to be the sum  
$$
\sum_{i=1}^m length_\De(p_i) \in \Delta. 
$$

The ``metric space'' $(X, d_{\Delta})$ has interesting geometry. For instance, the {\em generalized triangle inequalities} for $X$ are necessary and sufficient conditions for existence of an oriented triangle in $X$ with the given $\De$-side lengths $(\la, \mu, \nu)\in \Delta^3$. A priori, it is far from clear why such consitions are given by linear inequalities. It was proven in \cite{KapovichLeebMillson1} that these necessary and sufficient conditions are exactly the inequalities defining the eigencone $C(H)$ for any (complex or real) semisimple Lie group $H$ whose Weyl group is isomorphic to $W$; below we will explain why this is true for the group $H=G^\vee$. 

Not much is known about this ``geometry'' beyond the generalized triangle inequalities. For instance, one can ask, to which extent, this geometry is ``nonpositively curved.'' Below is a partial result in this direction (reminiscent of the fact that the ordinary distance between geodesics in a $CAT(0)$ space is a convex function):

\begin{thm}
The $\Delta$-distance function between geodesics in $X$ is $\Delta^*$-convex. More precisely: Let $\ga_1(t), \ga_2(t)$ be geodesics in $X$ parameterized with the constant speed. Define the function 
$$
\varphi(t)=d_\Delta(\gamma_1(t), \gamma_2(t)). 
$$
Then for all $a, b$, and $t\in [0,1]$, 
$$
\varphi((1-t)a+ tb)\le_{\Delta^*}  (1-t) \varphi(a) + t \varphi(b). 
$$
\end{thm}

We now continue with definitions. A vertex in the  Euclidean building $X$ modeled on $(A,W_{af})$ is {\em special} if it is the image of a special vertex of the model apartment $A$ under a chart.  A triangle $\tau$ in $X$ is called {\em special} if its vertices are special vertices of $X$ and $\Delta$-side lengths are elements of $\Delta \cap P(R^\vee)$. 

We define 
$$
\widehat{\T}_{\la,\mu,\nu} 
$$
to be the space of oriented triangles in $X$ with the $\Delta$-valued side-lengths $\la, \mu, \nu$. Note that we do not require vertices of these triangles to be vertices of $X$. Similarly, we define
$$
\widehat{\T}^{sp}_{\la,\mu,\nu}\subset \widehat{\T}_{\la,\mu,\nu}
$$
to be the subset consisting of special triangles. 

\subsubsection{Spherical buildings at infinity} 

Spherical or {\em Tits} buildings are defined via axioms similar to the ones for Euclidean buildings, except the model space is no longer a Euclidean spaces equipped with an action of a reflection group, but a sphere equipped with an isometric action of a finite reflection group. A spherical building $Y$ is a cell complex 
whose faces are isometric to faces of the {\em spherical Weyl chamber} $\De_{sph}\subset S^{\ell-1}$, a fundamental domain for $W$. (The Weyl chamber $\De$ of $W$ is the Euclidean cone over $\De_{sph}$.) Then $Y$ admits a simplicial projection 
$\theta: Y\to \Delta_{sph}$. The image $\theta(\xi)$ is called the {\em type} of a point $\xi\in Y$. What will be most important for us is that 
every Euclidean building $X$ has the {\em ideal boundary} $Y=\tits X$ which has natural structure of a spherical building modeled on the sphere $S^{\ell -1}$ equipped with the action of the finite Weyl group $W$. Every geodesic ray $\rho: [0,\infty)\to X$ in $X$ determines a {\em point at infinity} $\rho(\infty)\in Y$. Two rays 
$\rho_1, \rho_2$ determine the same 
point if they are {\em asymptotic}, i.e., the distance function
$$
d(t)=|\rho_1(t) \rho_2(t)|
$$
is bounded. One of the key geometric property of $X$ that we need is that for two asymptotic rays, the distance function $d(t)$ is {\em non-increasing}. 
This monotonicity property is implied by the $CAT(0)$ property of $X$. 

The angular (Tits) metric on $Y$ is denoted $\angle$. For instance, if $X$ is a tree, then $Y$ has the discrete metric which takes only the values $0$ and $\pi$. If $X$ is the rank 2 affine building associated with the group $SL(3, F)$ ($W$ is the permutation group $S_3$ in this case), then 
$Y$ a metric bipartite graph where every edge has length $\pi/3$. Vertices of $Y$ are points and lines in the projective plane $P^2(F)$. Two vertices are connected by an edge iff they are incident. 

Similar construction works for nonpositively curved symmetric spaces $X$: Every such space admits the ideal boundary $\tits X$ which has natural structure of a spherical building. 

\subsection{Weighted configurations and stability} 

Suppose that $\xi_1,\ldots,\xi_n$ are points in a spherical building $Y$ equipped with {\em masses} $m_1,\ldots,m_n$, which are nonnegative numbers. Given such {\em weighted configuration} $\psi$ in $Y$  we define (see \cite{KapovichLeebMillson1}) the {\em slope function}
$$
\slope_\psi(\eta)= -\sum_{i=1}^n m_i \cos(\angle(\xi_i, \eta)). 
$$

\begin{defn} 
A weighted configuration $\psi$ is called {\em (metrically) semistable} if $\slope_\psi(\eta)\ge 0$ for every $\eta$.
\end{defn}

This condition is introduced in \cite{KapovichLeebMillson1} in order to characterize properness of certain functions on $X$, namely, {\em weighted Busemann functions} associated to $\psi$. These functions can be defined for more general finite measures on ideal boundaries of $CAT(0)$ spaces and they play important role in complex analysis (they were first introduced by Douady and Earle in the context of Teichm\"uller theory) and Riemannian geometry.

Note that a positive multiple of a semistable configuration is again semistable. For instance, if $X$ is a tree then $\psi$ is semistable if and only if the total mass of any point in $Y$ does not exceed half of the total mass
$$
\sum_{i=1}^n m_i
$$
of $\psi$. (If some points $\xi_i$ coincide, their masses, of course, add.) 

It turns out (see \cite{KapovichLeebMillson1}) that the metric notion of semistability is essentially equivalent to Mumford's definition, once we introduce an algebraic group acting on $Y$. 

\medskip 
\subsubsection{Gauss correspondence} 

Let $\Pi=x_1 x_2 x_3 \ldots x_n$ denote an {\em oriented geodesic polygon} in $X$ with vertices $x_i$ and edges $x_i x_{i+1}$, $i$ is taken modulo $n$. (The reader can assume that $n=3$ since we are primarily interested in triangles in $X$.) 
We then extend every edge $x_i x_{i+1}$ of $\Pi$ to a geodesic ray $\rho_i$ starting at $x_i$ and representing a point $\xi_i=\rho_i(\infty)\in Y$. We assign the weight $m_i=|x_i x_{i+1}|$ to $\xi_i$. The ray $\rho_i$ is non-unique, but this will not concern us; what's important is that the type $\theta(\xi_i)$ is well-defined (unless $m_i=0$). We note that $\theta(\xi_i)$ is the unit vector which has the same direction as  
$$
\la_i=d_\Delta(x_i, x_{i+1})\in \Delta. 
$$
Thus,
$$
\la_i= m_i \theta(\xi_i). 
$$

The multivalued map $Gauss: \Pi \mapsto \psi$ is the {\em Gauss correspondence}. (The picture defining this correspondence first appears in the letter from Gauss to Bolyai, see \cite{Gauss}.) The following is the key result relating polygons in $X$ and weighted configurations in 
$Y$. 

\begin{thm}
[\cite{KapovichLeebMillson1}]\label{thm:semistable} 
Every weighted configuration $\psi\in Gauss(\Pi)$ is semistable. 
Conversely, for every semistable weighted configuration $\psi$ in $Y$, there exists a polygon in $X$ (with the metric side-lengths $m_i$) so that $\psi\in Gauss(\Pi)$.  
\end{thm}

Below we describe what is involved in proving the {\em hard} direction in this theorem, namely, the converse implication. (It is very instructive to see why the theorem holds in the rank 1 case.) 

Our goal is to ``invert Gauss maps'', i.e., given a semistable weighted configuration
$\psi$, we  would like to find a closed geodesic $n$-gon $\Pi$
so that $\psi\in Gauss(\Pi)$. The polygons $\Pi$ correspond to the fixed points
of a certain dynamical system on $X$ that we describe below.
For $\xi\in Y=\tits X$ and $t\ge 0$, define the map $\phi:= \phi_{\xi,t}: X\to X$
by sending $x$ to the point at distance $t$ from $x$
on the geodesic ray $x\xi$ starting from $x$ and asymptotic to $\xi$. Since $X$ is $CAT(0)$, the map $\phi$ is
1-Lipschitz  (recall that the distance function between asymptotic rays is decreasing). Then, given a weighted configuration
$\psi$ with non-zero total mass, define the map
\[
\Phi=\Phi_{\psi}: X\to X
\]
as the composition
$$
\phi_{\xi_n,m_n}\circ\dots\circ\phi_{\xi_1,m_1}.
$$
The fixed points of $\Phi$ are the first vertices of closed polygons $\Pi= x_1\ldots x_n$
so that $\psi\in Gauss(\Pi)$. Since the map $\Phi$ is 1-Lipschitz, and the space $X$ is complete 
$CAT(0)$, the map $\Phi$ has a fixed point if and only if the dynamical system $(\Phi^i)_{i\in \N}$ has a bounded orbit, see \cite{KapovichLeebMillson2}.
Of course, in general, there is no reason to expect that $(\Phi^i)_{i\in \N}$ has a bounded orbit: For instance, if the
configuration $\psi$ is supported at a single point, all orbits are unbounded. The following theorem was proven for locally compact buildings
in the original version of \cite{KapovichLeebMillson2} and by Andreas Balser \cite{Balser} in the general case:

\begin{thm}
\label{stabletofixed}
Suppose that $X$ is a Euclidean building. Then $\psi$ is semistable if and only if
$(\Phi^i)_{i\in \N}$ has a bounded orbit. \end{thm}

An algebraically inclined reader can ignore all the material in this section except for the following corollary, 
whose only known proof goes through the stability theory and Gauss correspondence described above: 

\begin{cor}
Let $X$ be a Euclidean buildng. Let $\la,\mu,\nu$ be dominant weights such that
$$
\la+\mu+\nu\in Q(R^\vee).
$$
Take $N\in {\mathbb N}$ and set
$$
(\la',\mu',\nu')=N(\la,\mu,\nu). 
$$
Suppose that the space $\widehat{T}_{\la',\mu',\nu'}\ne \emptyset$. 
Then there exists a triangle $\tau\in \widehat{T}_{\la,\mu,\nu}$ so that vertices of $\tau$ are vertices of $X$. 
\end{cor}

\subsection{Buildings and algebraic groups} 

The (so far, purely geometric) theory of buildings connects to the theory of algebraic groups as follows. 
Given a group $G={\sf G}(F)$ as above, Bruhat and Tits \cite{BT} associate with $G$ 
a Euclidean (Bruhat-Tits) building $X$, so that $G$ acts by automorphisms on $X$. 
the action $G\acts X$ is transitive on the set of apartments in $X$. For each 
apartment $A\subset X$ we define $G_A$ to be the stabilizer of $A$ in $G$. Then, the image of $G_A$ in $Aut(A)$ (under the 
restriction map) is equivariantly isomorphic to the group $L\rtimes W$, so that 
$$
Q(R^\vee) \subset L\rtimes W \subset W_{af}\subset P(R^\vee)\rtimes W  
$$  
and $L$ is the cocharacter lattice of ${\sf G}$. The maximal split torus $T$ in $G$ preserves one of the apartments $A\subset X$ and acts on $A$ as the lattice  $L$. 
Furthermore, the group $K$ stabilizes a certain special vertex $o\in X$; the $G$-orbit $Gr_G:= G\cdot o= G/K$ is called the 
{\em affine Grassmannian} of $G$. 

The spherical building $Y=\tits X$ 
is recovered from $G$ as follows: Stabilizers of faces of $X$ are parabolic subgroups of $G$; stabilizers of facets are conjugates of the Borel subgroup $B$. 
Thus, the space of facets in $Y$ is naturally identified with the quotient $G/B$, the set of $F$-points in the complete flag variety ${\sf G}/{\sf B}$. 

The group $G$ has the Cartan decomposition 
$$
G=K T_+ K,
$$
where $T_+$ is a subsemigroup in the torus $T$ consisting of elements of the form
$$
\chi(\pi), \quad \chi\in L_+,
$$
where $\pi\in F$ is a uniformizer (note that $\chi$'s are cocharacters of ${\sf T}$). Thus, we obtain the {\em Cartan projection}
$$
c: G\to L_+, \quad c(g)=\chi(\pi). 
$$
In the case $G=SL(\ell+1)$, the Cartan decomposition is just another interpretation of the Smith normal form for elements of $G$; thus, we will think of the vector 
$c(g)$ as the set of {\em invariant factors} of $g$.

We can now give a GIT interpretation of semistability for weighted configurations in $Y$. For simplicity, we assume that each point $\xi_i\in Y$ is a {\em regular point}, i.e., 
it belongs to a unique facet $\si_i$ of $Y$; thus, $\si_i\in G/B$. We will also assume that each vector 
$$
\la_i=m_i\theta(\xi_i) \in \Delta
$$
belongs to the lattice $L$. In particular, $\la_i$ determines a line bundle $\L_{\la_i}$ on $G/B$.  Then 

{\em A weighted configuration $\psi$ is metrically semistable if and only if the $n$-tuple of chambers 
$$
(\si_1,\ldots,\si_n)\in (G/B)^n$$
is semistable in the GIT sense with respect to the diagonal action of the $G$, where we use the tensor product of the line bundles  $\L_{\la_i}$ to define the polarization. 
}

A similar statement holds for general weighted configurations in $Y$, except we have to use product of (possibly) partial flag varieties. 

We now can give the first algebraic interpretation of oriented polygons $\Pi=x_1 \ldots x_n$ (with $x_1=o$) in $X$ whose vertices belong to $Gr_G$: Every  
polygon $\Pi$ with $\De$-side lengths $\la_1, \ldots, \la_n$ determines a tuple of elements $g_1, \ldots, g_n\in G$ such that 
$$
c(g_i)=\la_i, i=1,\ldots n
$$
and
$$
g_1 \cdots g_n=1.  
$$
Conversely, every tuple $(g_1,\ldots,g_n)$ determines a polygon as above. It turns out that instead of constructing polygons with vertices in $Gr_G$, it suffices to construct special polygons  $X$. We refer to \cite{KapovichLeebMillson2} and \cite{KapovichLeebMillson3} for the details. 

We then define two sets
$$
Hecke(G)\subset Tri(X),
$$
where $Hecke(G)$ consists of triples $(\la_1,\la_2,\la_3)\in L_+^3$ such that there exists a special oriented triangle in $X$ with $\De$-side lengths $\la_1, \la_2, \la_3$, while $Tri(X)$ consists of triples $(\la_1,\la_2,\la_3)\in \De^3$ such that there exists an oriented triangle in $X$ with $\De$-side lengths $\la_1, \la_2, \la_3$. (In the next section we will see why the latter set has the name {\em Hecke}.) For now, we just record the (easy) fact that 
$$
Hecke(G)\subset \{(\la_1, \la_2,\la_3): \la_1+ \la_2 +\la_3\in Q(R^\vee)\},
$$
see \cite{KapovichLeebMillson3} for two different proofs (geometric and algebraic). Observe also that, by considering disoriented special triangles in $X$, we can interpret the set $Hecke(G)$ as answering  the following algebraic problem:

\begin{itemize}
\item Given two sets of invariant factors $\la_1, \la_2 \in L_+$, describe possible invariant factors of the products $g_1 g_2$, where $c(g_1)=\la_1, c(g_2)=\la_2$. 
\end{itemize}

\medskip 
In this survey we discuss two ways in which special triangles in Euclidean buildings $X$ connect to the representation theory of the group $G^\vee$:

\begin{itemize}

\item Satake correspondence. 

\item Littelmann path model. 

\end{itemize}

\subsection{Hecke ringss, Satake transform and triangles in buildings}\label{sec:satake}

In this section we describe the Satake transform from the (spherical) 
Hecke ring of $G$ to the representation ring of $G^\vee$. We refer the reader to \cite{Gross} for more details.
(There are more general notions of Satake transform which apply to other discrete valued fields, like $\C((t))$; these generalizations require one to work with sheaves.)  
 
\subsubsection{Satake transform} 

Below we describe an integral transform $S$, the Satake transform,
from a ring ${\mathcal H}_G$ (spherical Hecke ring) of compactly supported, $K$-biinvariant functions on $G$ to the ring of 
left $K$-invariant, right $U$-invariant functions on $G$. The space of functions ${\mathcal H}_G$ is equipped with the convolution product
$$
f \star g (z) = \int_G f(x) \cdot g(x^{-1}z)dx$$
where $dx$ is the Haar measure on $G$ giving $K$ volume $1$. Then $({\mathcal H}_G, \star)$ 
is a commutative and associative ring. 

Let $\delta: B \to \mathbb{R}^*_+$ be the modular function of $B$, 
We may regard $\delta$ as a left $K$-invariant, right $U$-invariant function on $G$. 
By the Iwasawa decomposition for $G$, any such function is determined by its restriction to $T$.
We normalize the Haar measure $du$ on $U$ so that the open subgroup $K \cap U$
has measure $1$. For a compactly supported $K$-biinvariant function $f$ on $G$ we define its
{\em Satake transform} as a function $Sf(g)$ on $G$ given by
$$
Sf(g)= \delta(g)^{1/2} \cdot \int_U f(gu)du.
$$
(The reader can think of $S$ as a generalization of the Fourier transform.) 
Then $Sf$ is a left $K$-invariant, right $U$-invariant
function on $G$ with values in $\Z[q^{1/2}, q^{-1/2}]$; this function
is determined by its restriction to $T/T\cap K \cong \X_*(\underline{T})$.
Let $R(G^{\vee})\cong (\Z[\X_*({\sf T})])^W $ be the representation ring of $G^\vee$. 
Then: 
 
\begin{thm}
The image of $S$ lies in the subring
$$
(\Z[\X_*({\sf T})])^W \otimes \Z[q^{1/2},q^{-1/2}]$$
and $S$ defines a ring isomorphism
$$
S: \mathcal{H}_G \otimes \Z[q^{1/2},q^{-1/2}] \to (\Z[\X_*({\sf T})])^W \otimes \Z[q^{1/2},q^{-1/2}]\cong 
R(G^{\vee})\otimes \Z[q^{1/2},q^{-1/2}].
$$
\end{thm}

\subsubsection{Connection to geometry}

How does Hecke ring relate to geometry of buildings? Functions on $G$ which are right-invariant under $K$ are the same thing as functions on the affine Grassmannian $Gr_G$, while $K$-biinvariant functions are the same thing as functions on the cone of dominant cocharacters $L^+\subset L=\X_*({\sf T})$, i.e., functions on the set
$$
\Delta\cap L\cdot o\subset A, 
$$
where $A$ is a model apartment of $X$ and $\Delta$ is the positive chamber corresponding to our choice of the Borel subgroup 
${\sf B}$. We let $x_\mu\in A$ denote the vertex corresponding to the image of $o$ under the translation of $A$ given by the cocharacter $\mu$. Thus, we can identify $K$-orbits in $Gr_G$ as ``spheres with fixed $\Delta$-radius'' ${\mathbb S}_\mu(o)$:  
$$
K\cdot x_\mu = {\mathbb S}_\mu(o)=\{x\in X^0: d_\Delta(o,x)= \mu\}, 
$$
where $X^0$ is the vertex set of $X$. More generally, for a vertex $u\in Gr_G$ we set 
$$
{\mathbb S}_\mu(u):= \{x\in X^0: d_\Delta(u,x)= \mu\}. 
$$
The affine Grassmannian $Gr_G$ has structure of $F_q$-points of an ind-scheme, 
where the ``spheres'' ${\mathbb S}_\mu(u)$ are algebraic subvarieties (see \cite{Haines}). The closures of these subraieties also have geometrically appealing interpretation as ``closed metric balls''
$$
\ol{{\mathbb S}_\mu(u)}= {\mathbb B}_\mu(u)=\{x\in Gr_G: d_\Delta(u,x)\le_{\De^*} \mu\}.  
$$

The cocharacters $\mu\in L$ define functions $c_\mu \in {\mathcal H}_G$, where $c_\mu|\Delta$ is the characteristic function of the singleton $\{x_\mu\}$. Since ${\mathcal H}_G$ consists of compactly supported functions, the functions $c_\mu$ form a basis in ${\mathcal H}_G$. In particular, we get the Hecke structure constants
$$
m_{\la,\mu}^\eta\in \Z_+,
$$
$$
c_\la \star c_\mu= \sum_{\eta} m_{\la,\mu}^\eta c_\eta. 
$$
Of course, these constants completely determine the ring ${\mathcal H}_G$. We next interpret the constants 
$m_{\la,\mu}^\eta$   in terms of geometry of the building $X$. 

\subsubsection{Spaces of special triangles} 

Fix vectors $\la, \mu, \eta \in L_+$. Then every {\em disoriented} special triangle $abc$ with vertices in $Gr_G\subset X$ and $\De$-side-lengths $d_\Delta(a,b)=\la, d_{\Delta}(b,c)=\mu, d_{\Delta}(ac)=\eta$ can be transformed (via an element of $G$) to a disoriented triangle of the form
$$
o y x_\eta,
$$
where the first and the last vertices are fixed and the vertex $y$ is variable.  We let 
$$
{\mathcal T}_{\la,\mu}^\eta(F_q)
$$
denote the space of such triangles. Similarly, for $\nu=\eta^*$ we define the space of {\em oriented} triangles
$$
{\mathcal T}_{\la,\mu,\nu}(F_q)
$$
of the form $o y x_{\eta}$ with the $\De$-side-lengths $\la, \mu, \nu$. Let $f(q)=m_{\la,\mu,\nu}(q)$ denote the  
cardinality of the latter set. 

Then (see \cite{KapovichLeebMillson3}): 

\begin{lem}
$m_{\la,\mu}^\eta$ is the cardinality of the set ${\mathcal T}_{\la,\mu}^\eta(F_q)$. 
\end{lem}

Thus, 
$$
Hecke(G):= \{(\la,\mu,\nu): m_{\la,\mu}^{\nu^*}\ne 0\},
$$
which explains the name: This set answer the problem of describing the weights $\eta=\nu^*$ which appear with nonzero coefficient in the expansion of the 
product $c_\la \star c_\mu$ in terms of the basis $\{c_\eta\}$ of the Hecke ring ${\mathcal H}$. 

Even though, $S$ is an isomorphism, the relation between the structure constants of Hecke ring of $G$ and the character ring of $G^\vee$ is somewhat indirect. Define 
$$
n_{\la,\mu,\nu}= \dim (V_\la\otimes V_\mu \otimes V_\nu)^{G^\vee}. 
$$
Let $\rho$ denote the half-sum of positive roots in $R^\vee$. 

\begin{thm}
[\cite{KapovichLeebMillson3}]\label{thm:rep->hecke}
$f(q)$ is a polynomial function of $q$ of degree $\le q^{\<\rho, \la+\mu+\nu\>}$
so that
$$
f(q)= n_{\la,\mu,\nu} q^{\<\rho, \la+\mu+\nu\>} +\hbox{~lower order terms}.
$$
In particular, if $n_{\la,\mu,\nu}\ne 0$ then $m_{\la,\mu,\nu}(q)\ne 0$ and, hence, 
${\mathcal T}_{\la,\mu,\nu}\ne \emptyset$. In other words, if
$$
V_\eta\subset (V_\la \otimes V_\mu)^{G^\vee}
$$
then there exists a triangles in $X$ with special vertices and $\Delta$-side lengths 
$\la, \mu, \nu=\eta^*$. 
\end{thm}

Furthermore, if we replace the finite field $F_q$ with the algebraically closed field $\bar{F}_q$, then 
the  space of triangles ${\mathcal T}_{\la,\mu,\nu}(\bar{F}_q)$ in the corresponding affine Grassmannian becomes an algebraic variety of dimension 
$\le q^{\<\rho, \la+\mu+\nu\>}$ and the number $n_{\la,\mu,\nu}$ is the number of components of  ${\mathcal T}_{\la,\mu,\nu}(\bar{F}_q)$ which have the dimension 
$q^{\<\rho, \la+\mu+\nu\>}$ (see \cite{Haines}).

The reverse relation between triangles in $X$ and tensor product decomposition is more subtle:
For all simple complex Lie groups $G^\vee$ of non-simply laced type, there are examples where 
${\mathcal T}_{\la,\mu,\nu}$ is nonempty while $n_{\la,\mu,\nu}=0$ (see Theorem \ref{thm:examples}). 

We define the semigroup
$$
Rep(G^\vee):=  \{(\la,\mu,\nu): n_{\la,\mu,\nu}\ne 0\}. 
$$
(Note that $Hecke(G)$ need not be a semigroup, see \cite{KapovichMillson1}.) Thus, we have the inclusions
$$
Rep(G^\vee)\subset Hecke(G)\subset (L_+)^3 \cap \{\la+\mu+\nu\in Q(R^\vee)\}. 
$$

\subsection{Littelmann path model}

Let $A$, $W_{af}$, $R$, $G^\vee$, etc. be as in section \ref{sec:buildings}. 
Littelmann, in the series of papers \cite{littelmann1, littelmann2, littelmann3} defined a {\em path model} for the representation ring of the group $G^\vee$. The key to this model is the notion of {\em LS paths} in $A$. Below, we will give a definition of LS paths following \cite{KapovichMillson2}. This definition is essentially equivalent to Littelmann's definition (one difference is that we do not insist on the end-points of the path being in the coroot lattice); however, we do not explain the action of root operators on LS paths.  
 
 \subsubsection{Hecke and LS paths} 
 
Every LS path is a piecewise-linear path $p: [0,1]\to A$ in $A$ satisfying several conditions. For every $t\in [0,1]$ we define two derivatives $p'_-(t), p'_+(t)$: These are the derivatives on the left and on then right respectively. We will assume that 
$p(t)$ has constant speed, i.e., the Euclidean norm $|p'(t)|$ is constant; in particular, 
$$
|p'_-(t)|=|p'_+(t)|, \quad \forall t. 
$$ 
For every $t$ we also have the finite Coxeter group $W_{p(t)}$, the stabilizer of $p(t)$ in $W_{af}$.  

\medskip 
{\bf Axiom 1 (``Billiardness'').} The path $p$ is a {\em billiard path}: For every $t\in [0,1]$, the vectors $p'_-(t), p'_+(t)$ belong to the same $W_{p(t)}$-orbit, i.e., there exists $w\in W_{p(t)}$ so that $w(p'_-(t))= p'_+(t)$. 

For instance, if $w$ is a single reflection in a wall $H$ passing through $p(t)$, then the above condition simply says that the path $p$ bends at the point $t$ according to the rules of  optics (i.e., by the reflection in the wall $H$, the ``mirror'' or the ``side of the billiard table''). 

\medskip 
Since $W_{p(t)}$ is generated by reflections, Axiom 1 implies that  $w\in W_{p(t)}$ can be factored as a product of affine reflections
$$
w=\tau_{k}\circ \ldots \circ \tau_{1}
$$
in the group $W_{p(t)}$, where the derivative of each $\tau_{i}$ is a reflection $\tau_{\beta_i}\in W$ corresponding to a positive root $\beta_i\in R$. In particular, we obtain a {\em chain} of vectors
$$
u_0:= p_-'(t), u_1:= \tau_{\beta_1}(u_0), \ldots , u_k= p'_+(t)= \tau_{\beta_k}(u_{k-1}).  
$$

\begin{defn}
Let $W'\subset W$ be a reflection subgroup. A finite sequence of vectors $u_0,\ldots,u_k$ in $\R^\ell$ is called a {\em positive $W'$-chain} from $u_0$ to $u_k$ 
if for each $i\ge 1$ there exists a reflection $\tau_{\beta_{i}}\in W'$ (corresponding to a positive root $\beta_{i}$) such that $\tau_{\beta_{i}}(u_{i-1})=u_{i}$ and 
$$
u_{i}\ge_{\Delta^*} u_{i-1},
$$  
i.e., $u_{i}-u_{i-1}$ is a positive multiple of $\beta_{i}$. In particular, 
$$
u_k\ge_{\Delta^*} u_0.
$$
A positive $W'$-chain is called $W$-maximal if it cannot be refined to a larger positive $W$-chain from 
$u_0$ to $u_k$. 
\end{defn}

Clearly, every positive $W$-chain can be refined to a positive $W$-chain which is maximal. However, this is not the case for arbitrary positive $W'$-chains where $W'\ne W$. For the group $W_{p(t)}$ we let $W_{p(t)}'$ denote the subgroup of $W$ consisting of derivatives (i.e., linear parts) of elements of $W_{p(t)}$.

\medskip 
{\bf Axiom 2 (``Positivity'').} A billiard path $p$ is called {\em positively folded} (or {\em Hecke}) path, if for every $t$ there exists a positive $W_{p(t)}'$-chain 
from $p'_-(t)$ to $p'_+(t)$.

\medskip
Geometrically speaking, positivity of the path $p$ means that at each break-point, the derivative $p'_+(t)\in T_{p(t)}(A)\cong \R^\ell$ is obtained from $p'_-(t)\in T_{p(t)}(A)$ by applying a sequence of reflections fixing $p(t)$, so that each reflection moves the corresponding vectors $u_{i-1}$ further towards the positive chamber $\De$. 

\medskip 
{\bf Axiom 3 (``Maximality'').} A positive $W_{p(t)}'$-chain in Axiom 2 can be found, which is  $W$-maximal. 

\medskip
From the geometric viewpoint, this is a strange axiom: It is defined in terms of inability to further refine positive $W_{p(t)}'$-chains  
{\em even if we are allowed to use reflections in $W$ which need not be reflections fixing $p(t)$ and, hence, have nothing to do with the fold made by the path $p$ at the point $p(t)$.}  

Note that Axiom 3 is satisfied automatically at each point $p(t)$ which is a special vertex of the apartment $A$.

\begin{defn}
A piecewise-linear path in $A$ is called an {\em LS-path} if it satisfies Axioms 1, 2 and 3. A path satisfying Axioms 1 and 2 is called a {\em Hecke path}. 
\end{defn}

\subsubsection{Littelmann's path model for tensor product multiplicities} 

Given points $x, y\in A$ and a vector $\mu\in \De$, 
one considers the collection $LS_{x,y,\mu}$ of LS paths $p$ in $A$ connecting $x$ to $y$, so that $length_\De( p )=\mu$. Similarly, one defines the set of Hecke paths $Hecke_{x,y,\mu}$ of Hecke paths connecting $x$ to $y$. 

If $x, y$ belong to $\De$, we consider the subset $LS^+_{x,y,\mu}\subset LS_{x,y,\mu}$ consisting of {\em positive paths}, i.e., paths 
whose image is contained in $\De$. For a weight $\eta\in L$ we let $x_\eta\in A$ denote the point so that $\ov{o x_\eta}=\be$. 
Given a weight $\ga\in L_+$ we let $V_\ga$ denote the (finite-dimensional) irreducible representation of $G^\vee$ with the highest weight $\ga$. 

\begin{thm}
[P. Littelmann, \cite{littelmann2}] \label{LT1}
Let $\la, \mu\in L_+\subset P(R^\vee)$ be weights for the group $G^\vee$. Then
$$
V_\la \otimes V_\mu = \bigoplus_{\eta\in L_+} n_{\la,\mu}^\eta V_\eta, 
$$
where the multiplicity $n_{\la,\mu}^\eta$ equals the cardinality of $LS^+_{x_\la, x_\eta, \mu}$.  
\end{thm}

We will call the ``broken triangle'' in $\De$ which is the concatenation of the geodesic segment 
${o x_\la}$,  piecewise-geodesic path $p\in LS^+_{x_\la, x_\nu, \mu}$ and the 
geodesic segment ${x_\eta o}$, a {\em Littelmann triangle}.  Similarly, we define {\em Hecke triangles} by replacing LS paths with Hecke paths. 

\begin{rem}
The spaces of Littelmann and Hecke triangles are invariant under scaling by natural numbers.  
\end{rem}

\subsubsection{1-skeleton paths} 

In \cite{KapovichMillson2}, in order to get a better connection between the Littelmann path model and triangles in buildings, we had to 
modify slightly the concept of LS paths. (This modification is actually a special case of a more general class of paths defined by Littelmann earlier in terms of root operators.) 
Namely, we will have to relax Axiom 1 in the definition (and, accordingly modify Axiom 2). 
Let $\varpi_1,\ldots,\varpi_\ell$ denote the fundamental weights of $R$. Then every 
positive weight $\mu$ of $G$ is the sum
$$
\mu= \sum_{i=1}^\ell \mu_i, \quad \mu_i=c_i \varpi_i, \quad c_i\in \Z_+,\quad i=1,\ldots,\ell. 
$$ 
We then obtain a {\em model path} $p_\mu$ in $\De$ as the concatenation of the geodesic segments $p_i$ connecting $o$ to 
$x_\mu$, where each $p_i$ is a translate of the geodesic segment $o x_{\mu_i}$. Thus, $p_\mu$ is contained in the 1-skeleton of the Coxeter cell complex. Using paths $p_\mu$ as a model, one defines {\em generalized Hecke and LS paths}: These are paths $p$ in the 1-skeleton of the Coxeter complex,  where each $p$ is a concatenation of Hecke (resp. LS) paths $p_1,\ldots p_\ell$, so that
$$
length_\De(p_i)= \mu_i=c_i \varpi_i. 
$$ 
In addition, generalized Hecke and LS paths have to satisfy certain {\em positive-folding} condition at each end-point of $p_i, i=1,\ldots,\ell-1$ (the positivity condition is the same for Hecke and LS paths). We refer to \cite{KapovichMillson2} and \cite{GL2} for the precise definition.

We let $Hecke^1_{x,y,\mu}$ and $LS^1_{x,y,\mu}$ denote the set of generalized Hecke and LS paths  connecting $x$ to $y$. We also define sets $Hecke^{1,+}_{x,y,\mu}$ and $LS^{1,+}_{x,y,\mu}$ of generalized Hecke and LS paths contained in $\De$. It is proven in \cite{KapovichMillson2} that the set 
$$
\bigcup_{y} LS^1_{o,y,\mu}$$ 
coincides with the set of paths in $A$ obtained from $p_\mu$ by applying root operators. 
Thus, Littelmann's proof of Theorem \ref{LT1} goes through in the case of generalized LS paths and we obtain 

\begin{thm}
[\cite{KapovichMillson2}] \label{T2}
Cardinalities of $LS^+_{x_\la, x_{\eta}, \mu}$ and $LS^{1,+}_{x_\la, x_{\eta}, \mu}$ are the same and equal $n_{\la,\mu}^{\eta}$. 
\end{thm}

One of the key advantages of generalized Hecke paths is the following:

\begin{lem}\label{lem:scaling}
If $p$ is a generalized Hecke path, then $k_R\cdot p$ is a generalized LS path. 
\end{lem}
\proof Scaling by any natural number $k$ sends $Hecke^1_{x,y,\mu}$ to $Hecke^1_{kx,ky,k\mu}$. Moreover, since break-points of every path $p\in Hecke^1_{x,y,\mu}$ 
are at vertices of $A$, the break-points of $k\cdot p$ are at special vertices of $A$. Therefore, the path $k_R\cdot p$ satisfies the Maximality Axiom and, hence, is a generalized LS-path. \qed 

\begin{rem}
Already for $R=A_2$, there are Hecke paths which are not LS paths, even though $k_R=1$. 
\end{rem}

\subsection{Path model connection of triangles in buildings and tensor product multiplicities}
\label{sec:LS}

Pick a special vertex in $X$ which is the image of $o\in A$ under a chart; by abusing the notation we will again denote this vertex of $X$ by $o$. We then have the natural projection $\P_\Delta: X\to \Delta$, 
$$
x\mapsto d_\Delta(o,x). 
$$
It is easy to see that $\P_\Delta$ sends each geodesic path $\tilde{p}$ in $X$ to a piecewise-geodesic path $p$ in $\De$ so that
$$
length_\De({p})= length_\De(\tilde{p}). 
$$
Also, the image of every geodesic path $\tilde{p}=ox$ under $\P_\Delta$ is again a geodesic. With a bit more care, one proves that for every oriented geodesic 
triangle $\tau=xyz$, so that $x$ is a vertex of $X$, one can choose a special vertex $o\in X$ so that the paths $\P_\Delta(xy)$ and 
$\P_\Delta(zx)$ are still geodesic (see \cite{KapovichMillson2}). We will refer to such $\si=\P_\Delta(\tau)$ as a {\em broken triangle}. 

We now can state the key results connecting geodesics in $X$ and the path model(s) in $A$:

\begin{thm}
[\cite{KapovichMillson2}] \label{thm:unfolding}
1. For every geodesic path $\tilde{p}$ in $X$, its projection $p=\P_\De(\tilde{p})$ is a Hecke path. 

2. Conversely, every Hecke path $p$ in $\De$ is the $\P_\De$-projection of a geodesic path in $X$.  

3. Let $\phi: A\to X$ be a chart and let $p_\mu\subset A$ be a model generalized Hecke path. Then $\P_\De(\phi(p_\mu))$ is a generalized Hecke path in $\De$.  
\end{thm} 

\begin{rem}
1. Some of the arguments in \cite{KapovichMillson2} were simplified in \cite{GGPR}. 

2. In \cite{KapovichMillson2} we could not prove that every generalized Hecke path in $\De$ can be unfolded to a model generalized Hecke path in an apartment in $X$. 
This was accomplished later on by Gaussent and Littelmann in \cite{GL2}. 
\end{rem}

We can now apply these results to triangles in $X$. First of all, if $\si\subset \Delta$ is a Littelmann triangle with the side-lengths $\la, \mu, \nu\in \Delta$, then $\si$ is also a Hecke triangle. In view of Part 2 of Theorem \ref{thm:unfolding}, 
the broken triangle $\si$ can be unfolded to an (oriented) geodesic triangle $\tau\subset X$ whose vertices are 
in $Gr_G$ and whose $\De$-side lengths are still $\la, \mu, \nu$. This gives an alternative proof of the inclusion
$$
Rep(G^\vee)\subset Hecke(G). 
$$
(The first proof was based on the Satake transform, see section \ref{sec:satake}.) 

The second corollary is a {\em saturation theorem} for the set $Hecke(G)$. Suppose that $\tau=xyz$ is an oriented geodesic triangle in $X$  (with the $\De$-side lengths $\la, \mu, \nu$), whose vertices are vertices of $X$. Then multiplication by $k=k_R$ sends the {\em broken triangle} $\si:=\P_\Delta(\tau)$ to a new broken triangle $\si'=k \si$, whose vertices are special vertices of $A$. The broken side of $\si'$ is still a Hecke path, thus, 
the new broken triangle $\si'$ is a Hecke triangle in $A$. In view of Part 2 of Theorem \ref{thm:unfolding}, 
the Hecke triangle $\si'$ can be unfolded to a special geodesic triangle 
$\tau'$ in $X$ whose $\De$-side lengths are $k\la, k\mu, k\nu$. Furthermore: 

\begin{thm}[\cite{KapovichMillson2}] 
\label{thm:Hecke-sat}
Suppose that $(\la,\mu,\nu)\in (L_+)^3$ and $\la+\mu+\nu\in Q(R^\vee)$. Then
$$
\exists N\in \N, N(\la,\mu,\nu)=(\la',\mu',\nu')\in Hecke(G) \Rightarrow k_R(\la,\mu,\nu)\in Hecke(G)
$$
\end{thm}
\proof By assumption, there exists an oriented geodesic triangle $\tau'$ in $X$ with $\De$-side lengths $(\la',\mu',\nu')$. Then every $\psi\in Gauss(\tau')$ is semistable (Theorem \ref{thm:semistable}). By the same theorem, since semistability is preserved by scaling, there exists an oriented triangle $\tau\in \T_{\la,\mu,\nu}$, whose vertices are vertices of $X$. Thus, for $k=k_R$, the broken triangle $\si'=k(\P_\Delta(\tau))$ is a Hecke triangle. Hence, by Theorem \ref{thm:unfolding} (Part 2), this Hecke triangle can be unfolded to a geodesic triangle 
$$
\tau''\in \T^{sp}_{k\la,k\mu,k\nu}. 
$$
Hence, 
$$
k(\la,\mu,\nu)\in Hecke(G). \qed 
$$

Similarly, we obtain 

\begin{thm}
[\cite{KapovichMillson2}] \label{thm:hecke->rep}
$$
k_R\cdot Hecke(G)\subset Rep(G^\vee). 
$$
\end{thm}
\proof Suppose $(\la,\mu,\nu)\in Hecke(G)$. Take a special triangle $\tau=xyz\subset X$ with the $\De$-side lengths $\la,\mu,\nu$, where $x=o$. Next, consider an apartment 
$A'=\phi(A)\subset X$ containing $y, z$ and replace the geodesic $yz$ with the model generalized Hecke path $p_\mu$ in $\Delta'\subset A'$, connecting $y$ to $z$ and having the $\De$-length $\mu$. Here $\Delta'\subset A'$ is a Weyl chamber with the tip $y$ containing the point $z$. The result is a ``broken triangle'' $\Pi\subset X$ (actually, $\Pi$ is a polygon but we prefer to think of the concatenation $p_\mu$ as a broken side of a triangle). Now, projecting $\Pi$ to $\De$ via $\P_\Delta$ results in a generalized Hecke triangle $\Si$ with two geodesic sides $x x_\la, x_\nu x$ as before and the broken side $\P_\Delta(p_\mu)$ which is a generalized Hecke path $p$. Scaling by $k=k_R$ sends 
$p$ to a generalized LS path $k\cdot p$ (see Lemma \ref{lem:scaling}). Thus, the rescaled polygon 
$k\cdot \P_\Delta(\Si)$ is a generalized Littelmann triangle with the $\De$-side-lengths $k \la, k \mu, k\nu$. Hence, 
$$
n_{\la',\mu',\nu'}\ne 0
$$
where $(\la',\mu',\nu')=k(\la,\mu,\nu)$. \qed 

\begin{cor} 
1. $k_R\cdot Hecke(G)\subset Rep(G^\vee) \subset Hecke(G)$. 

2. For a root system $R$ of type $A$, $Rep(G^\vee) = Hecke(G)$. 
\end{cor}

Theorem \ref{thm:unfolding} was improved by Gaussent and Littelmann as follows: 

\begin{thm}
[S. Gaussent, P. Littelmann, \cite{GL2}] Every generalized Hecke path $p$ in $\De$ with $length_\De( p)=\mu$, there exists a chart $\phi: A\to X$, so 
that $p=\P_\De(\phi(p_\mu))$. 
\end{thm} 

Using this result, they computed in \cite{GL2} the structure constants $m_{\la,\mu}^\eta$ for the spherical Hecke ring ${\mathcal H}_G$ using the path model based on generalized Hecke paths simplifying the earlier work by C.~Schwer \cite{Schwer}. 

\subsection{Saturation Theorems and Conjectures}

Recall that semigroup $Rep(G^\vee)$ is contained in the {\em eigencone} $C=C(G^\vee)$ which is the set of positive real linear combinations of elements of $Rep(G^\vee)$. 
In particular,
$$
(\la,\mu,\nu)\in (L_+)^3\cap C \iff \exists N>0, (N\la,N\mu,N\nu)\in Rep(G^\vee). 
$$
We define the semigroup
$$
C_L= C \cap (L_+)^3\cap  \{\la+\mu+\nu\in Q(R^\vee)\}
$$
containing  $Rep(G^\vee)$. Recall also that
$$
Hecke(G)= \{(\la,\mu,\nu): \widehat{\T}^{sp}_{\la,\mu,\nu}\ne \emptyset\} \subset Tri(X)= \{(\la,\mu,\nu): \widehat{\T}_{\la,\mu,\nu}\ne \emptyset\}.  
$$
In view of Theorem  \ref{thm:semistable}, the set $Tri(X)$ is stable under scaling (since scaling preserves semistability).

\begin{thm}
[\cite{KapovichLeebMillson3, KapovichMillson2}] \label{thm:main}

\begin{enumerate}
\item $$
 k_R C_L \subset Hecke(G) \subset C_L  .  
$$ 
\item 
$$
k_R Hecke(G)\subset Rep(G) \subset Hecke(G).  
$$
\item $$k_R^2 C_L \subset Rep(G^\vee).$$ 
\end{enumerate}
\end{thm}

\proof The only result that we did not yet explain is the inclusion $Hecke(G) \subset C_L$. Note that 
$$
k_R Hecke(G)\subset Rep(G)\subset C_L. 
$$ 
Since $C$ is a cone, 
$$
Hecke(G)\subset \frac{1}{k_R} C_L\subset C. 
$$
On the other hand,
$$
Hecke(G)\subset L^3\cap \{\la+\mu+\nu\in Q(R^\vee)\}. 
$$
\qed 

The inclusions (1) in this theorem are strengthened to

\begin{thm}
[\cite{KapovichLeebMillson1}] $C(G^\vee)=Tri(X)$. 
\end{thm}
\proof This theorem is proven in \cite{KapovichLeebMillson1} by direct geometric arguments; here we will present 
an indirect partial proof using the results that we explained so far.  Let us first verify the inclusion $C\cap (\Q \otimes L)^3\subset Tri(X)$. 
Since both sets are stable under rational scaling, it suffices to consider a triple $(\la,\mu,\nu)\in N\cdot C_L$ for large $N$, i.e.,   
$$
(\la,\mu,\nu)\in Rep(G^\vee) \subset Hecke(G)\subset Tri(X). 
$$ 
Note, furthermore, that in view of local compactness of $X$ (or, by appealing to projections of triangles in $X$ to $\Delta$), the set $Tri(X)$ is closed. 
Since rational triples are dense in $C$, we obtain the inclusion
$$
C\subset Tri(X). 
$$
In the same fashion one proves that
$$
Tri(X)\cap (L\otimes \Q)^3\subset C. 
$$
One can finish a proof by arguing that rational points are dense in $Tri(X)$. This, of course, follows from the results of \cite{KapovichLeebMillson1}, where it is proven by $Tri(X)$ is a rational cone. One can also give an alternative argument using root operators acting on generalized LS paths following the arguments used in \cite{KapovichMillson2}. \qed 
 
As a corollary, we obtain:

\begin{thm}
[Saturation Theorem, \cite{KapovichMillson2}]
If $\la+\mu+\nu\in Q(R^\vee)$ and $\la, \mu, \nu$ are dominant weights of $G^\vee$ such that   
$$
\exists N>0 \quad (V_{N\la} \otimes V_{N\mu} \otimes V_{N\nu})^{G^\vee}\ne 0 $$
then
$$ 
(V_{k\la} \otimes V_{k\mu} \otimes V_{k\nu})^{G^\vee}\ne 0. 
$$
for $k=k_R^2$. In particular, for $R=A_\ell$, $k_R=1$ and we recover the Saturation Theorem of Knutson and Tao \cite{KnutsonTao}. 
\end{thm}
\proof This theorem follows immediately from the combination of the results in Theorem \ref{thm:main}. It is useful, however, sketch the overall 
argument. Let $N>0$ be such that
$$
(N\la, N\mu, N\nu)\in Rep(G^\vee). 
$$ 
By the inclusion
$$
Rep(G^\vee)\subset Hecke(G)
$$
 (proven either using Satake correspondence as in section \ref{sec:satake} or via LS path model as in section \ref{sec:LS}), 
$$
(N\la, N\mu, N\nu)\in Hecke(G). 
$$
Let $\tau$ be a special oriented triangle in $X$ with the $\Delta$-side lengths $(N\la, N\mu, N\nu)$. Then, by Theorem \ref{thm:semistable}, every weighted configuration 
$$
\psi\in Gauss(\tau)
$$
is semistable. Since semistability is preserved by scaling, the weighted configuration
$$
\frac{1}{N}\psi
$$
is still stable. Thus, by Theorem \ref{thm:semistable}, there exists an oriented triangle in $X$ with the $\Delta$-side lengths $(N\la, N\mu, N\nu)$, whose vertices are vertices  of $X$ (note the condition 
$\la+\mu+\nu\in \Q(R^\vee)$). Therefore, by Theorem \ref{thm:Hecke-sat}, 
$$
(\la', \mu', \nu')= k_R(\la, \mu, \nu)\in Hecke(G). 
$$
Lastly, by Theorem \ref{thm:hecke->rep}, 
$$
k_R^2(\la,\mu,\nu)=k_R(\la', \mu', \nu')\in Rep(G^\vee). \qed 
$$

\bigskip 
It is, then, natural to ask to what extent the ``saturation factors'' $k_R$ and $k_R^2$ are needed in the above results.

\begin{thm}\label{thm:examples}
Let $R$ be a non-simply laced root system. Then there are triples $(\la,\mu,\nu)$ so that 
$$
(\la,\mu,\nu)\in C_L$$
but
$$
(\la, \mu,\nu)\notin Rep(G^\vee)$$
i.e.,
$$
(V_{\la} \otimes V_{\mu} \otimes V_{\nu})^{G^\vee}=0. 
$$
Moreover, in these examples, the triple $(\la,\mu,\nu)$ belongs to $Hecke(G)$.  
\end{thm}
\proof It is convenient to switch now the notation from $R^\vee$ to $R$ and from $G^\vee$ to $G$. 

In \cite{KapovichLeebMillson3} we constructed examples of such triples for $R=B_2=C_2$ and $R=G_2$. 
Below we will explain how to generalize these examples to the case $R=F_4, R=C_\ell, R=B_\ell$. In each case, we will use a triple of weights 
$(\la,\la,\la)$, where $\la$ is one of the fundamental weights.  In all cases we will choose $\la$ which belongs to the root lattice, and, hence, the condition
$$
3\la\in Q(R)
$$
(necessary for $(\la,\la,\la)\in Hecke$) is trivially satisfied. Thus, $(\la,\la,\la)$ belongs to $Hecke$, see Appendix to \cite{Haines2}.

\medskip
We now specify the weight $\la$: 

\begin{enumerate}
\item For $R=F_4$ we take $\la=\varpi_2$ (note that $\varpi_3$ does not give an example as $(\varpi_3,\varpi_3,\varpi_3)\in Rep(F_4)$). The proof in this case is 
an unilluminating computation using the LiE program for tensor product decomposition. 

\item For $R=C_{2m}$ we take $\la=\varpi_{\ell}$ (the longest fundamental weight), while for $R=C_{2m+1}$ we will take  
$\la=\varpi_{\ell-1}$ (the next to the longest fundamental weight). 

\item For $R=B_\ell$, $\ell>2$, we take $\la=\varpi_{1}$. 
\end{enumerate}

Note that for the root systems of type $B$ and $C$, we have chosen $\la$ so that for the point $x=x_\la$, 
the interior of the segment $ox$, intersects affine walls in exactly one point (the mid-point). 
We will give a proof that $(\la,\la,\la)\notin Rep$ for $R=C_\ell$ since the $B_\ell$ case is done by the same method. 

Consider first the case when $\ell$ is even. Then 
$$
\la=(1,\ldots,1)
$$
in the Bourbaki coordinates. Suppose that $(\la,\la,\la)\in Rep(Sp(\ell))$. Let $x=x_\la=(1,\ldots,1)$. Since the interior of the segment $ox$ intersects only one wall, every (positive) LS path $p$ connecting $x$  to itself has exactly one break point, a point $y\in \Delta$. Set $\mu=\ov{yx}$. Then 
$$
\mu\in W\cdot \frac{1}{2} \la. 
$$
Since $y\in \Delta$, we conclude that 
$$
y=(\frac{3}{2},\ldots, \frac{3}{2}, \frac{1}{2},\ldots, \frac{1}{2}). 
$$
The path $p$ is the concatenation of the segments $xy$ and $yx$. We claim that unless $y=(\frac{1}{2},\ldots, \frac{1}{2})$, 
the path $p$ is not a Hecke path. Indeed, in order for $p$ to be a Hecke path we would need at least
$$
\ov{xy}\le_{\De^*} \ov{yx} \iff 2\mu=(-1,\ldots, -1, 1,\ldots, 1)\in \Delta^*. 
$$
The latter is clearly false, unless the vector $-2\mu$ has only positive coordinates, i.e., $y=\frac{1}{2}x$. In the latter case, however, the path $p$ is Hecke but not LS, analogously to \cite{KapovichMillson1}: The Maximality Axiom is violated. 

Suppose now that $\ell$ is odd. Then 
$$
\la=(1,\ldots,1,0). 
$$
Again, every positive LS path connecting $x$ to itself has exactly one break, at a point $y\in \Delta$. Thus, the point $y$ has to be of the form
$$
(\frac{3}{2},\ldots, \frac{3}{2}, \frac{1}{2},\ldots, \frac{1}{2}, 0)
$$
(as the last coordinate has to be nonnegative). Now, the argument is exactly the same as in the even case. \qed 

\medskip 
On the other hand, all known examples fail for simply-laced groups. Furthermore, in all known examples, at least one weight is singular. 

\begin{conj}
[Saturation Conjecture] 1. If $R$ is a simply-laced root system, then 
$$
(\la,\mu,\nu)\in C_L \iff (V_{\la} \otimes V_{\mu} \otimes V_{\nu})^{G^\vee}\ne 0. 
$$

2. In general,  
$$
(\la,\mu,\nu)\in C_L  \Rightarrow (V_{2\la} \otimes V_{2\mu} \otimes V_{2\nu})^{G^\vee}\ne 0. 
$$

3. If $\la, \mu, \nu$ are regular weights then
$$
(\la,\mu,\nu)\in C_L \iff (V_{\la} \otimes V_{\mu} \otimes V_{\nu})^{G^\vee}\ne 0.
$$
\end{conj}

We refer the reader to \cite{KapovichMillson1} for more detailed discussion of the semigroup $Rep(G^\vee)$ and the set $Hecke(G)$. 

\newpage

\bibliographystyle{plain}
\def\noopsort#1{}

\vskip5ex

\noindent
Address: Shrawan Kumar,\\
Department of Mathematics,\\
University of North Carolina,\\
Chapel Hill, NC  27599--3250 \\
\noindent
email: shrawan@email.unc.edu

\bigskip 

\noindent
Address: Michael Kapovich,\\
Department of Mathematics,\\
University of California,\\
Davis, CA 95616 \\
\noindent
email: kapovich@math.ucdavis.edu

\end{document}